\documentclass{amsproc}
\usepackage{euscript}
\usepackage{cases}
\usepackage{mathrsfs}
\usepackage{bbm}
\usepackage{amssymb}
\usepackage{txfonts}
\usepackage{amsfonts,amsmath,amsxtra,mathdots,mathabx}
\usepackage{color}
\usepackage{hyperref}
\usepackage{tikz}

\allowdisplaybreaks

\DeclareFontFamily{U}{musix}{}%
\DeclareFontShape{U}{musix}{m}{n}{%
	<-12>   musix11
	<12-15> musix13
	<15-18> musix16
	<18-23> musix20
	<23->   musix29
}{}%

\newcommand*\musix{\usefont{U}{musix}{m}{n}\selectfont}
\DeclareTextFontCommand{\textmusix}{\musix}

\DeclareMathSymbol{\Lt}{3}{matha}{"CE}
\DeclareMathSymbol{\Gt}{3}{matha}{"CF}

\newcommand{\BC}{{\mathbb {C}}} 
 \newcommand{\BF}{{\mathbb {F}}}
 \newcommand{\BH}{{\mathbb {H}}}

\newcommand{\BQ}{{\mathbb {Q}}} \newcommand{\BR}{{\mathbb {R}}}

 \newcommand{\BZ}{{\mathbb {Z}}}

 \newcommand{\RN}{{\mathrm {N}}}

\def\frO{\text{$\text{\usefont{U}{BOONDOX-cal}{m}{n}O}$}\hskip 1pt}
\def\frOO{\text{$\text{\usefont{U}{BOONDOX-cal}{m}{n}O}$} }

\newcommand{\GL}{{\mathrm {GL}}} \newcommand{\PGL}{{\mathrm {PGL}}}
\newcommand{\SL}{{\mathrm {SL}}} \newcommand{\PSL}{{\mathrm {PSL}}}

 \newcommand{\Tr}{{\mathrm{Tr}}}

\newcommand{\sstyle}{\scriptstyle}

\newcommand{\ra}{\rightarrow}

\def\fra{\mathfrak{a}}
\def\-{^{-1}}

\def\mod{\mathrm{mod}\ }

\def\sumx{\sideset{}{^\star}\sum}

\def\tv{\textit{v}}

\def\tw{\textit{w}}

\def\lp {\left (}
\def\rp {\right )}

\def\BCx{\BC^\times}
\def\Voronoi{Vorono\" \i \hskip 2.5 pt }

\def\boldJ {\boldsymbol J}

\renewcommand{\Im}{{\mathrm{Im} }}
\renewcommand{\Re}{{\mathrm{Re} }}
\def\Nm{\mathrm{N}}
\def\tv{\textit{v}} 
\def\tw{\textit{w}} 
\def\shskip{\hskip 0.5 pt}

\makeatletter
\g@addto@macro\normalsize{\setlength\abovedisplayskip{3pt}}
\makeatother

\makeatletter
\g@addto@macro\normalsize{\setlength\belowdisplayskip{3pt}}
\makeatother

\newcommand{\delete}[1]{}

\theoremstyle{plain}

\newtheorem{thm}{Theorem}[section] \newtheorem{cor}[thm]{Corollary}
\newtheorem{lem}[thm]{Lemma}  \newtheorem{prop}[thm]{Proposition}

\newtheorem {rem}[thm]{Remark}

\newtheorem*{acknowledgement}{Acknowledgements}

\numberwithin{equation}{section}

\begin{document}

	\title[Subconvexity over the Gaussian number field]{Subconvexity for twisted $L$-functions on $\mathrm{GL}_3$ over the Gaussian number field}

	\author{Zhi Qi}
\address{School of Mathematical Sciences\\ Zhejiang University\\Hangzhou, 310027\\China}
\email{zhi.qi@zju.edu.cn}
	

	\begin{abstract}
		Let $q \in \mathbb{Z} [i]$ be prime and $\chiup $ be the primitive quadratic Hecke character modulo $q$. Let $\pi$ be a self-dual Hecke automorphic cusp form for $\mathrm{SL}_3 (\mathbb{Z} [i] )$ and $f$ be a Hecke cusp form for $\Gamma_0 (q) \subset \mathrm{SL}_2 (\mathbb{Z} [i])$. Consider the twisted $L$-functions $ L (s, \pi \otimes f  \otimes   \chiup) $ and $L (s, \pi \otimes \chiup)$ on $\mathrm{GL}_3 \times \mathrm{GL}_2$ and $\mathrm{GL}_3$. We prove the subconvexity bounds
		\begin{equation*}
		L \big(\tfrac 1 2, \pi \otimes f  \otimes   \chiup \big) \Lt_{\, \varepsilon,  \pi,  f } \mathrm{N} (q)^{5/4 + \varepsilon}, \hskip 8 pt  L \big(\tfrac 1 2 + it, \pi   \otimes   \chiup \big) \Lt_{\, \varepsilon,  \pi,  t } \mathrm{N} (q)^{5/8 + \varepsilon}, 
		\end{equation*}
		for any $\varepsilon > 0$.
	\end{abstract}
	
	\maketitle

\section{Introduction}

%\subsection{Backgrounds}

Subconvexity for $L$-functions is one of the central problems in analytic number theory. The principal aim is to get bounds
for a given $L$-function that are better than what the functional equation together with the Phragm\'en-Lindel\"of convexity principle would imply.

The subconvexity problem for $\GL_1$ and $\GL_2$ over arbitrary number fields was completely solved in the seminal work of Michel and Venkatesh \cite{Michel-Venkatesh-GL2}.  More recent work on the subconvexity for $\GL_2$ over number fields may be found in \cite{Blomer-Harcos-TR}, \cite{Maga-Shifted}, \cite{Maga-Sub} and \cite{WuHan-GL2}. 

Xiaoqing Li \cite{XLi2011} and Blomer \cite{Blomer} made the first progress on the subconvexity for $\GL_3$ and $\GL_3 \times \GL_2$ in the  $t$-aspect and the level aspect, respectively, in which they both study the first moment of $L$-functions for $\GL_3 \times \GL_2$ and exploit the nonnegativity of central $L$-values. It closely resembles the cubic moment of  $L$-functions for $\GL_1 \times \GL_2$ studied by Conrey-Iwaniec \cite{CI-Cubic}. Recently, there are developments of new techniques on stationary phase  and the analysis of Bessel integrals which yield  numerical improvement of the subconvexity exponent or hybrid subconvexity bounds in both the level and the $t$-aspect. See for example \cite{Ye-GL3}, \cite{Huang-GL3} and \cite{Nunes-GL3}. Some ideas in the latter two should be attributed to Young \cite{Young-Cubic}, in which he established the hybrid version of  Conrey and Iwaniec's results. Another approach to the $\GL_3$ subconvexity problem is the circle method technique elaborated in the series `The circle
method and bounds for $L$-functions I--IV' of Munshi \cite{Munshi-Circle-I}-\cite{Munshi-Circle-IV}. His method was further developed and simplified in \cite{Munshi-Circle-IV2}, \cite{Munshi-Sym2}, \cite{HoNe-ZeroFr}, \cite{Lin-GL3} and \cite{SunZhao-GL3}. Moreover, as an  application of the  $\GL_3$ Kuznetsov formula in \cite{Buttcane-Kuz}, Blomer and Buttcane \cite{Blomer-Buttcane-GL3,Blomer-Buttcane-GL3-2} successfully solved the subconvexity problem in the aspect of the $\GL_3$ Archimedean Langlands parameter (the $\mu$-aspect). All the current subconvexity results for $\GL_3$  in the literature are over the rational field $\BQ$.

In this paper, we shall obtain for the first time subconvexity results for $\GL_3$ over a number field other than $\BQ$. More precisely, we shall extend the work of Blomer \cite{Blomer} to the Gaussian number field $\BQ (i)$. 
Our main tools are the Kuznetsov formula for Hecke congruence subgroups of $\SL_2 (\BZ[i])$ and the \Voronoi summation formula for $\SL_3 (\BZ[i])$; the former is indeed established in \cite{B-Mo2} for imaginary quadratic fields and the latter in \cite{Ichino-Templier}  for arbitrary number fields.  As alluded to above, the advances along the moment method approach rely heavily on the innovations  in the analytic aspect. Likewise, our main focus here will be on the analysis of the $\GL_2 (\BC)$-Bessel kernel and Bessel integral arising in Kuznetsov and the $\GL_3 (\BC)$-Hankel transform in Vorono\"i. For this we must resort to the analytic theory of high-rank Bessel functions in \cite{Qi-Bessel}, especially  the asymptotic formula for the $\GL_3 (\BC)$-Bessel kernel.

\subsection{Main results}
Let $\BF = \BQ (i)$ be the Gaussian number field and $\frO = \BZ [i]$ be the ring of Gaussian integers. %and $\frO' = \frac 1 2 \, \BZ [i] = \frac 1 2  \, \frO$ be the dual lattice of $\frO$ with respect to the trace $\Tr = \Tr_{\BF/\BQ} = 2 \shskip \mathrm{Re}$. 
Let $\RN = \RN_{\BF/ \BQ} = |\  |^2$ denote the norm on $\BF$.

Let $q \in \frO$ be a square-free Gaussian integer such that $\mathrm{N} (q) \equiv 1 (\mod 8)$ and   $\chiup = \chiup_{\,q }$ be the primitive {\it quadratic} Hecke character  of conductor $q$ and frequency $0$. %Let  $ L \lp s ,   \chiup \rp $ be the Dirichlet $L$-function for $\chiup$. %See \S \ref{sec: Hecke characters} for more details.

%Let $q \in \frO$ be a squarefree Gaussian integer relatively prime to $2$, $k \in \BZ$ be an integer and   $\chiup = \chiup_{q, k}$ be the primitive quadratic Hecke character of conductor $q$ and frequency $2 k$. It comes from the Jacobi symbol $\chiup_q (n) =  \mbox{\larger[2]\text{$\big(\frac {\hskip 1.5 pt n \hskip 1 pt} q\big)$}}  $ on $\frO/q\frO $ and the character $\eta^{2k} (z) = (z/|z|)^{2k}$ on $\BCx$ such that $  \mbox{\larger[2]\text{$\big(\frac {\hskip 1.5 pt i \hskip 1 pt} q\big)$}}   = (-1)^k$.  See \cite[\S 3.8]{IK} for more details. 

For   $q' | q$ let $H^{\star} (q')  $ be  the  set of the $L^2$-normalized    Hecke newforms  on  $\Gamma_0 (q') \backslash \BH^3$   in the $L^2$-discrete spectrum of the  Laplace-Beltrami operator.  Here $\Gamma_0  (q') \subset \SL_2 (\frO)$ is the Hecke congruence group of level $q'$ as defined   in \eqref{1eq: congruence group, SL}. Put $B^{\star} (q) = \bigcup_{q'|q}  H^{\star} (q') $ and let $ B^{\star} (q) = \{f_j\}_{j \geqslant 1} $.  Let the Laplacian eigenvalue of $f_j$ be $   1   + 4 t_j^2$ and denote by $\lambdaup_j  (  n)$, $n \in \frO  \smallsetminus \{0\} $,  the   Hecke eigenvalues of $f_j$.

Let $\pi$ be a fixed {\it self-dual}  Hecke-Maass cusp form for $\SL_3 (\frO) $. 
Let $A (n_1, n_2)$, $n_1$, $n_2 \in \frO \smallsetminus \{0\}$, denote the Fourier coefficients of $\pi$, Hecke-normalized so that $A (1, 1) = 1$. % Since $\pi$ is self-dual, $A(n_1, n_2) = A(n_2, n_1)$.

\delete{Suppose that the Archimedean Langlands parameter of $\pi$ is the triple $(- \mu , 0, \mu )$ so that the gamma factor for $\pi$ is 
\begin{align*}
\gamma  (s, \pi) =  2^3 (2 \pi)^{- 3 s} \Gamma (s + \mu  )  \Gamma (s    ) \Gamma (s - \mu  ).
\end{align*}
} 

 %  Let  $\lambdaup_j (2 n)$, $n \in \frO'$, be  its   Hecke normalized Fourier coefficients  such that $\lambdaup_j (1) = 1$. For $\Re \, s$ large, the Rankin-Selberg $L$-function of $f$ and $f_j$ is defined by
We consider the twisted $L$-function
\begin{align}
L (s, \pi \otimes \chiup) = %\frac 1 4 \sum_{n \, \in \, \frO \smallsetminus \{0\} } \frac {A(1, n) \chiup (n)} {\mathrm{N} (n)^{s}} =
 \sum_{(n) \neq 0 } \frac {A(1, n) \chiup (n)} {\mathrm{N} (n)^{s} }
\end{align}
and for $f_j$ even the Rankin-Selberg $L$-function
\begin{align}\label{0eq: L (f uk)}
L (s, \pi \otimes f_j  \otimes \chiup) %= \frac 1 {16} \underset{n_1, n_2 \in \frO \smallsetminus \{0\} }{\sum \sum} \frac {\lambdaup_j (n_2) A (n_1, n_2) } {\left|n_1^2 n_2\right|^{2 s} } 
= \underset{(n_1), (n_2) \neq 0 }{\sum \sum} \frac { A (n_1, n_2) \lambdaup_j  (n_2) \chiup (n_2)} {\Nm  ( n_1^2 n_2 )^{ s} }
\end{align}
($ \lambdaup_j  (n_2)$ is independent on the representative of the ideal $(n_2)$ only when $f_j$ is even).

%admit an analytic continuation to an entire function on the whole complex plane. %The conductor of the former $L$-function is $( |q| (|t|+3))^6$ while that of the latter is $(|q| (|t|+3))^{12}$.
%Our second theorem  is as follows.

\begin{thm}\label{thm: main}
	%Let $f$ be a self-dual  Hecke-Maass form   on $\SL_3 (\frO) \backslash \PGL_3 (\BC) / \mathrm{PSU} (3)$ and   $\left\{  f_j \right\}_{j \geqslant 1} $ be    an orthonormal basis of  Hecke-Maass   forms on  $\PGL_2 (\frO) \backslash \PGL_2 (\BC)  / \mathrm {PSU} (2)$ corresponding to the Laplace eigenvalue $\frac 1 4 + t_j^2$, with $t_j \geqslant 0$, then for $\varepsilon > 0$, $T$, $M$ large and   $ T^{3/4} \leqslant M \leqslant T^{1- \varepsilon}$ we have \marginpar{\footnotesize I have assumed $M \geqslant T^{3/4}$ throughout this sketch. I will correct this when writing the article.} 
	Let notation  be as above. Assume that $q$ is prime. Let $T \geqslant 1$. For   $\varepsilon > 0$  and  $A = A(\varepsilon)$ sufficiently large,  we have
	\begin{equation}\label{0eq: thm}
	\begin{split}
	\sideset{}{'}{\sum}_{ |t_j| \leqslant T }      L \lp \tfrac 1 2 , \pi \otimes f_j  \otimes \chiup \rp 
	 +       \int_{- T}^{T}    \left| L \lp \tfrac 1 2 + it, \pi \otimes \chiup \rp \right|^2 \frac {t^2 d t} {t^2+1} \Lt T^{A} \mathrm{N} (q)^{5/4 + \varepsilon} ,
	\end{split}
	\end{equation}
where $\sum' $ restricts to the even Hecke cusps  forms in $B^{\star} (q)$. The implied constant depends only on   $\varepsilon$ and $\pi$. 
\end{thm}

By the nonnegativity theorem of Lapid in \cite{Lapid}, we have 
\begin{equation}\label{0eq: L > 0}
L \lp \tfrac 1 2 , \pi \otimes f_j  \otimes \chiup \rp \geqslant 0 .
\end{equation} 
As a consequence of \eqref{0eq: L > 0}, we derive from \eqref{0eq: thm} the following bound for the individual $L$-values.

\begin{cor}\label{cor: 1}
	 Let notation  be as above. Assume that $q$ is prime. We have
	 \begin{align}
	 L \lp \tfrac 1 2 , \pi \otimes f_j  \otimes \chiup \rp \Lt \RN (q)^{5/4 + \varepsilon}, 
	 \end{align}
	 for any $\varepsilon > 0$, the implied constant depending only  on $\varepsilon$, $\pi$ and $t_j$.
\end{cor}

Moreover, ignoring the contribution of the cuspidal spectrum in \eqref{0eq: thm} by nonnegativity \eqref{0eq: L > 0}, we have
\begin{align*}
\int_{- T}^{T}    \left| L \lp \tfrac 1 2 + it, \pi \otimes \chiup \rp \right|^2 \frac {t^2 d t} {t^2+1} \Lt \mathrm{N} (q)^{5/4 + \varepsilon} ,
\end{align*}
By the arguments in \cite[\S 1]{CI-Cubic} (see also \cite[\S 4]{Blomer}), we have the following corollary.

\begin{cor}\label{cor: 2}
	Let notation  be as above. Assume that $q$ is prime. We have
	\begin{align}
	 L \lp \tfrac 1 2 + it, \pi \otimes \chiup \rp \Lt \RN (q)^{5/8 + \varepsilon}, 
	\end{align}
	for $\varepsilon > 0$ and $t$ real, the implied constant depending only  on $\varepsilon$, $\pi$ and $t $.
\end{cor}

Since the corresponding convexity bounds for  $ L \lp \tfrac 1 2 , \pi \otimes f_j  \otimes \chiup \rp$ and $L \lp \tfrac 1 2 + it , \pi   \otimes \chiup \rp$ are $\RN(q)^{3/2 + \varepsilon} $ and $\RN(q)^{3/4 + \varepsilon} $, respectively, the above bounds in Corollary \ref{cor: 1} and \ref{cor: 2} break their convexity bounds.

\subsection{Remarks}

We regard the primary novelty of this work as not in the arithmetic but rather in the analysis of $\GL_2$- and $\GL_3$-Bessel kernels and the $\GL_3$-Hankel transform over $\BC$. For example, the   computations of character sums  in \cite{Blomer} may be applied here without any change (thanks to the fact that $\BQ (i)$ is of class number one). 
As for the analysis, although there are obvious similarities compared to \cite{Blomer}, Bessel kernels over $\BC$ are not only different from but more difficult to analyze than those over $\BR$.\footnote{This has already been alluded to in some previous works of the author. For example, Stirling's asymptotic formula for the gamma function, used by Xiaoqing Li and Blomer in \cite{XLi,Blomer} to establish the asymptotic formula for Bessel kernels over $\BR$, becomes fairly useless in the asymptotic theory of Bessel kernels over $\BC$ in \cite{Qi-Bessel}.}   A remarkable reflection of this difference is that the method in \cite{XLi2011} surprisingly fails to work over $\BQ (i)$  in the aspect of analytic conductor (the $t$-aspect). Roughly speaking, the analysis in \cite{XLi2011} breaks down on $\BC$ because there is not sufficient oscillation in the weights. %Other interesting implications may be found in \cite{Qi-Sph,Qi-II-G} on the Fourier transform of the Bessel kernels for $\SL_2 (\BC)$.  

This  paper may be viewed as the second  application of the \Voronoi summation formula over number fields in Ichino-Templier \cite{Ichino-Templier} and the asymptotic theory of high-rank Bessel functions in the author's work \cite{Qi-Bessel}. The first is   the Wilton and Miller bounds for additively twisted sums of $\GL_2$ and $\GL_3$ Fourier coefficients over arbitrary number fields in \cite{Qi-Wilton}.  

%Can the subconvexity results in \cite{Blomer} and this paper be further extended to other number fields?

The results here over the Gaussian number field may be extended straightforwardly  to the other eight imaginary quadratic number fields of {class number one}.
Furthermore, it is very likely that the subconvexity problem under consideration may be solved over an arbitrary number field in the same manner. First, the \Voronoi summation formula over number fields is well established  in \cite{Ichino-Templier}. Second, the spherical Kuznetsov trace formula  over number fields may be found in \cite{BM-Kuz-Spherical} or \cite{Venkatesh-BeyondEndoscopy} (see also \cite{Maga-Kuz}), although the  class of weight functions therein is not quite as general as in \cite{Kuznetsov} and \cite{B-Mo,B-Mo2}. %in \cite{Maga-Kuz}, especially its class of weight functions, is not general enough for the purpose, but this may be amended by the  extension method  in \cite[\S 10]{B-Mo} or \cite[\S 11.2]{B-Mo2}.  
Moreover, the works in \cite{Blomer} and this paper have completed the analysis of Bessel kernels, the Bessel integral and the Hankel transform over an Archimedean local field. Yet, many  details in the non-Archimedean aspect, which could be quite complicated, remain to be worked out if the class number is not one.

\begin{acknowledgement}
		This work was done during my stay at Rutgers University. I would like to acknowledge the Department of Mathematics for its hospitality and thank Stephen D. Miller and Henryk Iwaniec for their help. I also thank the referee for careful readings and  extensive	comments.
\end{acknowledgement}

\section{Preliminaries}

\subsection{General notation}
 
Throughout this article, we set $e (z) = e^{2 \pi i z}$. 

Let $\BF = \BQ (i)$ and $\frO = 2   \frO' = \BZ [i]$. %Let $(n) $ denote the ideal $n \frO$. 
Then $\frO' $ is the dual lattice of $\frO$ with respect to the trace $\Tr = \Tr_{\BF/\BQ} = 2 \shskip \mathrm{Re}$. 
Let $\frOO^{\times} = \{1, -1, i, - i\}$ be the group of units. As convention, let $(p)$ always stand for prime ideal of $\frO$.  

%$\mu (n)$ is the M\"obius function on $\frO$, $\varphi (n)$ is the Euler $\varphi$ function for  $\frO$.

The standard notation  for certain arithmetic functions on $\BZ$ will also be used for $\BZ [i]$, like the  M\"obius function $\mu $ and the Euler function $\varphi$. Namely, $\mu (n) = (-1)^{k}$ if $n$ is a square-free Gaussian integer with $k$ many prime factors, $\mu (n) = 0$ if $n$ is not square-free, and $\varphi (n)  = |n|^2 \prod_{ \sstyle  (p) \supset (n)    } \lp 1 - 1/ |p|^2 \rp$.

Let $d z$ be twice  the Lebesgue measure on $\BC$. In the polar coordinates, we have $d z = 2 x d x \shskip d \phi$ for  $z = x e^{i \phi}$.

 %Denote $\BCx =\BC \smallsetminus \{0\}$. %Define $[z] = z/|z|$ for $z \in \BC \smallsetminus \{0\}$. 
 
 \subsection{Kloosterman sums}
 
For $n_1$, $n_2  \in \frO$ and $c \in \frO \smallsetminus\{0\}$ define the Kloosterman sum \begin{equation}
\label{1eq: Kloosterman} S (n_1, n_2; c) = \sumx_{\sstyle a  (\mod c) } e \lp \mathrm{Re} \lp \frac {n_1 a + n_2 \widebar a }    {  c}  \rp \rp,
\end{equation}
where $\sumx$ means that $a$ runs over representatives of $( \frO / c \frO)^{\times} $ and $a \widebar a \equiv 1 (\mod c)$. 
%We have Weil's bound $ \left| S (n_1, n_2; c) \right| \Lt \left|(n_1, n_2 , c) \right| |c|^{1 + \varepsilon}$. 
%Let $R  (n; c) = S (0, n; c)$ be the Ramanujan sum, for which we have
%\begin{align}\label{1eq: Ramanujan sum}
%R  (n; c) = \frac {1} 4  \sum_{\sstyle c_1  c_2 = \, c \atop \sstyle c_2 |   n  }   {\mu (c_1)} {|c_2|^2}.
%\end{align}
%Thus $|R  (n; c) | \leqslant |(n, c)|^2$.

\subsection{Hecke characters} \label{sec: Hecke character}
Define $\eta (z) %= [z] 
= z/|z|$, $z \in \BC \smallsetminus \{0\}$. For $ q \in \frO \smallsetminus\{0\} $ let $I_q $ denote the group of fractional ideals that are relatively prime with $q$, that is, $I_q = \big\{ (n_1) (n_2)\- : (n_1, q) = (n_2, q) = \frO   \big\}$. Let $\omega$  be a character of $(\frO / q \frO)^{\times}$  and $k$ be an integer satisfying the units consistency condition: $\omega (\epsilon) \epsilon^k= 1$ for all $\epsilon \in \frOO^{\times}$. We may then form a Hecke character (Gr\"o\ss encharakter) $\chiup$ on $I_q$ such that
\begin{align*}
\chiup ((n)) = \omega (n) \eta^k (n), 
\end{align*}
for every $n \in \frO$, $(n, q) = \frO$. In addition, we assume $\chiup ((n)) = 0$ if $(n, q) \neq \frO$. The integer $k$ is called the frequency of $\chiup$. The Gauss sum $\tau (\chiup)$ associated with $\chiup$ is defined by
\begin{align*}
\tau (\chiup ) = \eta^k (q) \sumx_{a (\mod q) }   \omega (a) e \lp \Re \lp \frac a q \rp \rp. 
\end{align*}
The root number $\varepsilon (\chiup)  = i^{-k} \tau (\chiup) / \sqrt {\RN (q)} $.
See \cite[\S 3.8]{IK} for more details. 

Now assume that $q$ is odd and square-free. Let  $\omega$ be the   quadratic symbol   \text{\mbox{\larger[2]\text{$\big(\text{\raisebox{0.15ex}{$\frac {  \hskip 0.5 pt \text{\raisebox{-0.25ex}{\scalebox{1.5}{$\cdot$}}} \hskip 0.5 pt } q$}}\big)$}}}. Note that $ \mbox{\larger[2]\text{$\big(\text{\raisebox{0.15ex}{$\frac  {\hskip 1.5 pt i \hskip 1 pt}  q$}}\big)$}} = (-1)^{(\mathrm{N}(q) - 1) / 4} $, so one needs $2 k \equiv \mathrm{N}(q) - 1 (\mod 8)$ for the units consistency condition. In this article, we assume that  $\chiup = \chiup_q$ is quadratic (real), then the frequency $k = 0$  (in other words, $\chiup$ is trivial at the Archimedean place) and hence $\mathrm{N} (q) \equiv 1 (\mod 8)$.  In this case, we claim that $\tau (\chiup) = \sqrt {\RN (q)}$ and hence $\varepsilon (\chiup) = 1$. 
To see this, let us first assume that $q$ is prime. When $\RN (q) = q \widebar q$ splits, the character $\chiup_q $ is equal to the Legendre symbol $\xi_{\RN (q)} = \text{\mbox{\larger[2]\text{$\big(\text{\raisebox{0.15ex}{$\frac {  \hskip 0.5 pt \text{\raisebox{-0.25ex}{\scalebox{1.5}{$\cdot$}}} \hskip 0.5 pt } {\RN(q)}$}}\big)$}}}$ under the isomorphism $\frO / q \frO \cong \BZ / \RN (q) \BZ$, so  $\tau (\chiup_q) = \tau (\xi_{\RN (q)}) = \sqrt {\RN (q)}$ as $\RN (q) \equiv 1 (\mod 4)$.  When $q \in \BZ$ and $q \equiv -1 (\mod 4)$ so that $q$ is inert, the character $\chiup_{q} $ is induced from the Legendre symbol $\xi_q = \text{\mbox{\larger[2]\text{$\big(\text{\raisebox{0.15ex}{$\frac {  \hskip 0.5 pt \text{\raisebox{-0.25ex}{\scalebox{1.5}{$\cdot$}}} \hskip 0.5 pt } {q}$}}\big)$}}}$ on $\BZ / q \BZ$ via the norm map $\RN : \frO / q\frO \ra \BZ / q \BZ $, namely $\chiup_q = \xi_q {\scriptsize \text{\raisebox{0.4ex}{ o }}} \RN$, and by \cite[\S 3.8, Example 5]{IK} we have $\textstyle \tau (\chiup) = \xi_q (-1) \cdot \tau ( \xi_q )^2 = - \big( i \sqrt q \big)^2 = \sqrt {\RN (q)}. $
The quadratic reciprocity law for $\frO$ may be applied to prove $ \tau (\chiup) = \sqrt {\RN (q)}$ for any $q$ square-free.

\subsection{Stationary phase (the Van der Corput lemma)}

%We have the following stationary phase estimate which generalizes \cite[Lemma 1.1.6]{Sogge}, allowing the weight functions (amplitudes) be moderately oscillatory. 

\begin{lem}[Van der Corput]\label{Van der Corput}
	Let $K \subset \BR^{d}$ be a compact set that contains $0$, $U$ be an open neighborhood of $K$.  Let $S >0$ and $\sqrt {\lambdaup} \geqslant X \geqslant 1$. Let $u     (  x) \in C_0^{\infty} (K)$ and $f (x) \in C^{\infty} (U)$. Suppose that  $
    ( \partial / \partial x)^{\alpha} u     (  x)  \Lt_{\, \alpha }     {S X^{  \alpha}}$ 
	and that $f (x)$ is real-valued, $f (0) = 0$, $ f' (0) = 0$, $\det f'' (0) \neq 0$ and $f' (x) \neq 0$ in $K \smallsetminus \{0\}$. Then for any given multi-index $\gamma$, 
	\begin{align*}
	I_{\gamma } (\lambdaup) = \int_K e  (   \lambdaup f (x)  )   u     (x) x^{\gamma} d x \Lt  S / \lambdaup^{   (  |\gamma| + d) / 2 },
	\end{align*}
with the implied constant depending only on $\gamma$   when $f$ stays in a bounded set in $C^{\infty} (K)$ and $|x| / |f'(x)|$ has a uniform bound.
\end{lem}

Lemma \ref{Van der Corput} is a generalization of \cite[Lemma 1.1.6]{Sogge}, but here  $X$ can be as large as $\sqrt {\lambdaup}$, while $X = 1$ in \cite{Sogge}, which means that the amplitudes are allowed to have moderate oscillation.

In the settings of \cite{Sogge}, one works with amplitudes that involve  the parameter $\lambdaup$, namely, $u     (x) = u     (  \lambdaup, x)$, and, using the Van der Corput lemma, one may deduce the following stationary estimate as in \cite[Theorem 1.1.4]{Sogge},
\begin{align*}
(\partial / \partial \lambdaup )^{\beta} I_0 (\lambdaup) \Lt_{\, \beta}  S X^{\beta}  / \lambdaup^{d/2 + \beta}  
\end{align*}
from 
\begin{align*}
(\partial / \partial x )^{\alpha} (\partial / \partial \lambdaup )^{\beta} u     (\lambdaup, x  ) \Lt_{\, \alpha, \shskip \beta} S X^{\alpha} / \lambdaup^{\beta}.  
\end{align*}
%This estimate is very useful. 
In our settings however we  have to also differentiate with respect to an additional angular parameter $\theta$ involved  in the phase $f (x)$, so  \cite[Theorem 1.1.4]{Sogge} is not sufficient; nevertheless the  Van der Corput lemma would still do the job for us. 
  
\begin{proof}
	Following \cite{Sogge}, Lemma \ref{Van der Corput} in higher dimensions may be deduced from the one-dimensional case by an induction and the Morse lemma. When $d = 1$, Lemma \ref{Van der Corput}  can be proven by the   arguments of Van der Corput  as in the proof of \cite[Lemma 1.1.2]{Sogge}. Indeed, at the end we would get
	 \begin{align*}
	 I_{\gamma} (\lambdaup) \Lt_{\, \gamma, \shskip N} S r^{\,\gamma + 1} \lp  1 + \lp \frac {X + 1/ r} {\lambdaup \shskip r} \rp^N  \rp,
	 \end{align*}
	 for any $r > 0$ and $ N \geqslant \gamma +2$. The right side is smallest when $r = \big( X + \sqrt {X^2 + 4 \lambdaup} \big) / 2 \lambdaup$, which yields
	 \begin{align*}
	 I_{\gamma}  (\lambdaup) \Lt S \lp \frac {X + \sqrt {X^2 + 4 \lambdaup}} {2 \lambdaup} \rp^{\gamma + 1} \Lt S \lp \frac {X + \sqrt \lambdaup} {\lambdaup} \rp^{\gamma + 1}.
	 \end{align*}
	 Hence, we have the desired stationary phase bound $S / \lambdaup^{(\gamma + 1)/2}$ when $X \leqslant \sqrt \lambdaup$. 
\end{proof}

We shall use the following variant of the two-dimensional Van der Corput lemma in the polar coordinates. 

\begin{lem}\label{lem: Van der Corput, polar}
	Let $S >0$ and $\sqrt {\lambdaup} \geqslant X \geqslant 1$. In the polar coordinates, let $u      \lp x , \phi  \rp$  be a smooth function with support in the annulus $A [b, c] = \left\{ (x , \phi) : x  \in \left[b, c \right] \right\}$ and derivatives satisfying $
	\partial_x ^{\alpha} \partial_\phi^{\shskip \beta} u      \lp x , \phi   \rp \Lt_{ \alpha, \shskip \beta  } S   X^{  \alpha + \beta}$ for all $\alpha$, $\beta$.
	Let $f (x, \phi)$ be a smooth real-valued function such that $f (a, \theta) = 0$, $ f' (a, \theta) = 0$, $\det f'' (a, \theta) \neq 0$ and that $f' (a, \theta) \neq 0$ in $A [b, c] \smallsetminus \{(a, \theta)\}$. Then for any given  $\alpha$ and $\beta$, 
	\begin{align}\label{1eq: def I alpha beta}
I_{\alpha \beta} (\lambdaup) = \hskip -1 pt \int_{0}^{2\pi} \hskip -2 pt \int_0^{\infty} \hskip -1 pt e  (  \lambdaup f (x, \phi)  )   u     (x, \phi) (x -a )^{\alpha} \sin^{\shskip \beta} (   (\phi - \theta) / 2 )  \hskip 1 pt d x d \phi \Lt \hskip -1 pt S / \lambdaup^{   (  \alpha + \beta + 2) / 2 },
	\end{align}
	with the implied constant depending only on $\alpha$, $\beta$, when $f (x, \phi)$  stays in a bounded set in $C^{\infty} (A[b, c])$ and $\big((x-a)^2 + \sin^2  (\phi - \theta) / 2 ) \big) / \big( (\partial_x f (x, \phi))^2 + (\partial_\phi f (x, \phi))^2 \big)$ has a uniform bound.
\end{lem}

\section{A review of automorphic forms and $L$-functions}

\subsection{Automorphic forms on $\BH^3$} \ 

\vskip 5 pt

\subsubsection{The three-dimensional hyperbolic space}
We let $$\BH^3 = \left\{ w = z + j r = x + i y + j r : x, y, r \text{ real}, r > 0 \right\} %\cong \PSL_2 (\BC) / \mathrm {PSU} (2)
$$ denote the three-dimensional hyperbolic space, with the action of $\GL_2 (\BC)$ or $\PGL_2 (\BC)$($= \PSL_2 (\BC)$) given by 
\begin{align*}
z (g \cdot w) = \frac { (az+ b) (\overline {c} \overline z + \overline d) + a \overline c r^2 } {|cz+d|^2 + |c|^2 r^2} , \hskip 10 pt r (g \cdot w) = \frac {r |\det g| } {{|cz+d|^2 + |c|^2 r^2}}, 
\end{align*}  
$\displaystyle g = \begin{pmatrix}
a & b \\ 
c & d \\
\end{pmatrix} \in \GL_2 (\BC)$, while the action of $\GL_2 (\BC)$ on the boundary $\partial \BH^3 =   \BC \cup \{\infty \}$ is by the M\"obius transform.
$\BH^3$ is equipped with the $\GL_2 (\BC)$-invariant hyperbolic metric $    ( d x^2 + d y^2 + d r^2) / r^2$ and hyperbolic measure $    d x\, dy\, d r / r^3$.  The associated hyperbolic Laplace-Beltrami operator is given by $\varDelta =    r^2 \lp \partial^2/\partial x^2 + \partial^2/\partial y^2 + \partial^2/\partial r^2  \rp -   r \partial /\partial r$.

\delete
{The following lemma is an analogue of \cite[Lemma 2.1]{D-I-Kuz}.

\begin{lem}\label{1lem: Gamma a}
	Let $g \in \SL_2 (\BC)$ and $z \in \BC \cup \{\infty \}$ be such that $g z = z$. Then $\mathrm{tr}\, g = - 2, 0, 2$ if and only if 
	\begin{align*}
	g = \begin{pmatrix}
	\pm 1 & b \\
	& \pm 1
	\end{pmatrix} \text{ or }  \begin{pmatrix}
	\pm i & b \\
	& \mp i
	\end{pmatrix},
	\end{align*} 
	when $z = \infty$, and
	\begin{align*}
	g = \begin{pmatrix}
	cz \pm 1 & - c z^2\\
	c & -cz \pm 1
	\end{pmatrix}  \text{ or } , 
	\begin{pmatrix}
	c z \pm i & - c z^2 \mp 2 i z\\
	c & - cz \mp i
	\end{pmatrix}
	\end{align*} 
	when $z \neq \infty$. 
\end{lem}
}

\vskip 5 pt

\subsubsection{Hecke congruence groups}

For $q \in \frO \smallsetminus\{0\} $ define the Hecke congruence group 
\begin{align}\label{1eq: congruence group, SL}
\Gamma_0 (q) = \left\{ \begin{pmatrix}
a & b \\ 
c & d \\
\end{pmatrix} \in \SL_2 (\frO) : c \equiv 0\, (\mod q) \right\}. %/ \left\{ \epsilon I_2 : \epsilon \in \frOO^{\times} \right\}.
\end{align}
%\begin{align}
%\Gamma_0 (q) = \widetilde \Gamma_0 (q) \cap \SL_2 (\frO).
%\end{align}
$\Gamma = \Gamma_0 (q) $ %\subset \PGL_2 (\frO)$ 
is a discrete subgroup of $\SL_2 (\BC)$ which is   cofinite but not cocompact. Subsequently, we shall always assume that $q$ is square-free.

 %We have
%$$\Gamma_{ \infty}  = \left\{ \begin{pmatrix}
%\epsilon & n \\ &   \epsilon^{-1}
%\end{pmatrix} : \epsilon  \in \frOO^{\times}, n \in \frO \right\}. $$ 
\delete{For the cusp $\fra = 1/\varv$, it follows from Lemma \ref{1lem: Gamma a} that  $\Gamma_{\fra} $ is the union
\begin{align*}
 \left\{ \hskip - 2 pt \begin{pmatrix}
c \varw \pm 1 &  \hskip - 8 pt - c  \varw/\varv\\
c q & \hskip - 8 pt -c \varw \pm 1
\end{pmatrix} \hskip - 2 pt : \hskip - 1 pt c\equiv 0 \, (\mod \varv  ) \hskip - 1 pt \right\} \hskip - 2 pt \cup  \hskip - 1 pt
\left\{ \hskip - 2 pt \begin{pmatrix}
c \varw \pm i & \hskip - 6 pt (- c \varw \mp 2 i) /\varv \\
c q &  \hskip - 6 pt - c\varw \mp i
\end{pmatrix} \hskip - 2 pt : \hskip - 1 pt c \equiv \mp 2 i \hskip 0.5 pt \overline \varw \, (\mod \varv  ) \hskip - 1 pt \right\},
\end{align*}
in which $\varw = q/\varv$ is the complementary divisor. An easy calculation shows that one may choose}

%/ \left\{ \epsilon I_2 : \epsilon \in \frOO^{\times} \right\}. $$  %  Define $$\mathrm{N}_2 (\BC) = \left\{ \begin{pmatrix}
%1 & b \\ & 1
%\end{pmatrix} : b \in \BC \right\} $$ and $\Gamma_{\fra}' = \Gamma_{\fra} \cap \sigma_{\fra} \mathrm{N}_2 (\BC) \sigma_{\fra}\- $. 

\vskip 5 pt

\subsubsection{Maass cusp forms}
The $L^2$-discrete spectrum  of the  Laplace-Beltrami operator $\varDelta$ on $\Gamma  \backslash \BH^3$ comprises the constant function $f_0 \equiv 1/ \sqrt {\mathrm{Vol} (\Gamma \backslash \BH^3) }$ and an orthonormal basis of Maass cusp forms $\{f_j\}_{j \geqslant 1}$ which are eigenfunctions of $\varDelta$.
For $f_j $ with Laplacian eigenvalue $    1 + 4 t_j^2$, we have the Fourier expansion
\begin{align*}
f_j (z, r) = \sum_{ n \, \in \frO' \smallsetminus \{0\} } \rho_j (2n) r K_{2 i t_j} (4 \pi |n| r) e (\Tr (n z)).
\end{align*}
We recall the Kim-Sarnak bound in \cite{Blomer-Brumley} over the field $\BF$,
\begin{align}\label{1eq: Kim-Sarnak}
|\Im \, t_j| \leqslant \tfrac 7 {64},
\end{align}
so $t_j$ is either real or   imaginary with $|i \hskip 0.5 pt t_j| \leqslant \frac 7 {64}    $. 

%For $\epsilon \in \frOO^{\times} = \{1, -1, i, - i\}$  %denote $a (\epsilon) = \begin{pmatrix}
%\epsilon & \\ & 1 
%\end{pmatrix}$. 
%Next, we  consider the action of $\begin{pmatrix} \epsilon & \\ & 1 \end{pmatrix}$ for $\epsilon\in \frO^{\times}$. 
Since $\displaystyle i  \hskip -0.5 pt \cdot \hskip -2 pt  \begin{pmatrix}
-1 & \\ & 1 
\end{pmatrix}  \in \Gamma_0 (q)$, we infer that $f_j$ is invariant under the action of $\displaystyle \begin{pmatrix}
-1 & \\ & 1 
\end{pmatrix}$ and hence $ \rho_j (- n) = \rho_j (n)$. 
%Now we  consider the action of $\begin{pmatrix} i & \\ & 1 \end{pmatrix}$. 
A Maass cusp form for $\Gamma_0 (q) $ is said to be even or odd if it is an eigenfunction of the action of $\displaystyle \begin{pmatrix}
i & \\ & 1 
\end{pmatrix}$ with eigenvalue $1$ or $-1$, respectively.  %This terminology is justified as the the associated $L$-function has functional equation of sign $1$ or $-1$ according as  the form is even or odd, and in the latter case the central $L$-value vanishes for trivial reasons. The reader may compare these with \cite[\S 3.9, 3.13]{Goldfeld}. 
We may require that each $f_j$ is either even or odd. Note that $f_j$ is even if and only if $\rho_j (\epsilon n) = \rho_j (n)$ for all $\epsilon \in \frOO^{\times}$ and $n \in \frO \smallsetminus \{0\}$.

\begin{rem}
	It would be of some interest to  introduce the congruence group $  \Gamma'_0 (q) \subset \GL_2 (\frO)$  defined similarly as in \eqref{1eq: congruence group, SL},
	\begin{align}\label{1eq: congruence group}
	\Gamma_0' (q) = \left\{ \begin{pmatrix}
	a & b \\ 
	c & d \\
	\end{pmatrix} \in \GL_2 (\frO) : c \equiv 0\, (\mod q) \right\}.
	\end{align}
	It is clear that   an even  Maass cusp form for $\Gamma_0 (q) $ is indeed a Maass cusp form for $\Gamma_0' (q)$. 
\end{rem}

%Since $f_j$ is invariant under $a (\epsilon)$, we infer that $\rho_j (\epsilon n) = \rho_j (n)$. %\footnote{This is analogous to the condition that defines even Maass forms on $\BH^2$ and hence the sign of the functional equation is $+1$ (compare \cite[\S 3.9, 3.13]{Goldfeld}). It is for this reason that we consider here the congruence subgroups of $\PGL_2 (\frO)$ instead of $\PSL_2 (\frO)$.}.

%\begin{rem}\label{1rem: even Maass forms}
	%Conventionally, the congruence group $\Gamma_0 (q)$ is a subgroup of $\SL_2 (\frO)$ or $\PSL_2 (\frO)$ similarly defined as in \eqref{1eq: congruence group}.
%	If only the $\Gamma_0' (q) = \Gamma_0 (q) \cap \SL_2 (\frO)$-automorphy is assumed, then one would only get $ \rho_j (- n) = \rho_j (n)${\rm;} note that $i   \cdot \hskip -0.5 pt  a (-1)  \in \SL_2 (\frO)$. 
%Observe that $\Gamma_0'(q)$ is of index $2$ in $\Gamma_0 (q)$ and $\Gamma_0 (q) = \{ a(i), a(-i) \} \cdot \Gamma_0' (q)$. A Maass cusp form for $\Gamma_0'(q) $ is said to be even or odd if it is an eigenfunction of the action of $a (i)$ with eigenvalue $1$ or $-1${\rm;} in particular,  if the form is even then it is indeed a Maass cusp form for $\Gamma_0 (q)$. 

%\subsubsection{The Hecke-normalized Kuznetsov trace formula}

%\subsubsection{Hecke operators and newforms}

For $n \in \frO \smallsetminus\{0\}$, we define the Hecke operator $T_n$ by
\begin{equation*}
T_n f (w) = \frac 1 {4 |n|} \sum_{ a d = n} \ \sum_{ b (\mod d)} f \lp \begin{pmatrix}
a & b \\
& d
\end{pmatrix}
w \rp.
\end{equation*}
%if $f$ is $\Gamma_0 (q)$-automorphic. % (see \cite[\S 8]{B-Mo}). %The Hecke algebra consisting of all the Hecke operators is commutative. 
Hecke operators commute with each other as well as  the Laplacian operator. 
We may   further assume that every $f_j$ is an eigenfunction of all the Hecke operators $T_n$ with $ (n, q) = \frO $. 
%For $u = f_j$, we have the Fourier expansion
%\begin{align*}
%T_m f_j (z , r) = \frac 1 4 \sum_{ \sstyle n \, \in \frO' \smallsetminus \{0\} } \left( \sum_{\sstyle a \, \in \frO \smallsetminus \{0\} \atop {\sstyle m/a \, \in \frO, \,  \sstyle n/a \, \in \frO' } }    \rho_j \lp \frac {2m n} {a^2} \rp \right) r K_{2 i t_j} \lp 4 \pi  | n  | r \rp e \lp 2 \Re \lp n z  \rp \rp.
%\end{align*}
Let $\lambdaup_j (n)$ denote the the Hecke eigenvalue of $T_n$ for $f_j$, then % the formula above yields the relation 
\begin{equation}\label{1eq: relation Fourier Hecke}
\rho_j (n) = \rho_j \lp 1 \rp \lambdaup_j ( n),  
\end{equation} 
if $(n, q) = \frO$. The Hecke eigenvalues $\lambdaup_j (n)$ are real and satisfy the Hecke relation
\begin{align}\label{1eq: Hecke rel}
\lambdaup_j (n_1) \lambdaup_j (n_2) = \frac 1 4 \sum_{ d | n_1, \shskip d| n_2 } \lambdaup_j \big(n_1 n_2 / d^2\big),
\end{align}
if $(n_1 n_2, q) = \frO$.

Finally, let $H^{\star} (q) $ %= \{ f_{j, q} \}_{j \geqslant 1}$ 
be the   set of the $L^2$-normalized   newforms for $\Gamma_0 (q)$ %(newforms for $\Gamma'_0 (q)$) 
which are eigenfunctions of all the $T_n$.   Later in \S \ref{sec: Kuznestov}, the orthonormal basis $\{ f_j \}_{j \geqslant 1}$will be constructed from all the newforms for $\Gamma_0 (q')$ with $q'|q$. 

\vskip 5 pt

\subsubsection{Eisenstein series}
%There  $\PGL_2 (\frO)$ has only one cusp at infinity up to equivalence and we l
%Let $\mathrm{N}_2 (\frO)  = \left\{  \begin{pmatrix}
%1 & n \\
%& 1 
%\end{pmatrix} :  n \in \frO \right\}  . $
%be the associated  maximal unipotent subgroup fixing infinity. 
%\red{cusps...}

%The eigenpacket in the continuous spectrum of the  Laplace-Beltrami consists of Eisenstein series $E_{\fra} (w; s)$ on the line $\Re s = \frac 1 2$. 
For each cusp $\fra$ of $\Gamma = \Gamma_0 (q)$, we form the Eisenstein series
\begin{align*}
E_{\fra} (w; s) =  \sum_{ \gamma \, \in \, \Gamma_{\fra} \backslash \Gamma } r (  \sigma_{\fra}\- \gamma \cdot w)^{2 s}, 
\end{align*}
if $\Re \, s > 1$ and by analytic continuation for all $s$ in the complex plane.  Here $\Gamma_{\fra}$ denote the stability group of $\fra$ in $\SL_2 (\frO)$ and $\sigma_{\fra} \in \SL_2 (\BC)$ is such that
\begin{align}\label{1eq: scaling}
\sigma_{\fra} \infty = \fra    \hskip 5 pt
\text{ and } \hskip 5 pt \sigma_{\fra}\- \Gamma_{\fra} \sigma_{\fra} = \Gamma_{ \infty}.
\end{align}
The Fourier expansion of $E(z, r; s)$ is similar to that of a cusp form. Precisely
\begin{align}\label{1eq: E a Fourier}
E_{\fra} (z, r; s) = \varphi_{\fra} r^{2 s} +   \varphi_{\fra} (s) r^{ 2-2s} 
+  \sum_{ n \, \in \frO \smallsetminus \{0\} }    \varphi_{\fra} (n, s) r K_{2s - 1 } (2 \pi |n | r ) e (\Re (n z)) ,
\end{align}
with $\varphi_{\fra} =1 $ if $\fra \sim \infty$ or $\varphi_{\fra} = 0 $ if otherwise.

The continuous  $L^2$-spectrum of the  Laplacian comprises all the $E_{\fra} \lp w, \tfrac 1 2 + i t \rp$.

Following \cite[\S 3]{CI-Cubic}, we compute the Fourier coefficients of $E_{\fra} (w; s)$ by using the Eisenstein series for $\SL_2 (\frO) = \Gamma_0 (1)$,
\begin{align}\label{1eq: Eisenstein SL}
E (z, r; s) =  \frac 1 4 r^{2 s} \underset{(c, \shskip d) = \frOO}{\sum \sum} \big( |cz+d|^2 + |c|^2 r^2 \big)^{- 2s} ,
\end{align}
which is known to have an explicit Fourier expansion as in \eqref{1eq: E a Fourier}
%\begin{align}\label{1eq: Eisenstein SL}
%E (z, r; s) =  r^{2 s} +   \varphi (s) r^{ 2-2s} 
%+  \sum_{ n \, \in \frO' \smallsetminus \{0\} }    \varphi (2n, s)    r K_{2s - 1 } (4 \pi |n | r ) e (\Tr (n z)),
%\end{align}		
with
\begin{align}\label{1eq: Fourier Eisenstein}
&\varphi (s)   = \frac {\pi \zeta_{\BF} (2s - 1)} {  \lp 2 s -   1   \rp \zeta_\BF (2s )}, \hskip 10 pt \varphi (n, s)   = \frac {  2\pi^{2s } {\eta (n, s)}    } { \Gamma (2s  ) \zeta_\BF (2s )},   
\end{align}
in which $\zeta_\BF $ is the Dedekind zeta function associated with $\BF$, 
\begin{align}\label{1eq: zeta}
\zeta_\BF (s) = \sum_{ (n) \neq 0} \frac 1 {|n|^{2s}} = \frac 1 4 \sum_{ n \neq 0 } \frac 1 {|n|^{2s}},
\end{align}
%with $(m)$ the principal ideal generated by $m \in \frO \smallsetminus \{0\}$, 
and for $n \in \frO \smallsetminus \{0\}$
\begin{align}\label{1eq: eta}
\eta (n, s) =  \sum_{ \sstyle (a) \supset (n)   }   {\left| \frac n {a^2} \right|^{2 s - 1}} = \frac 1 4  \sum_{   a |  n }   {\left| \frac n {a^2} \right|^{2 s - 1}},
\end{align} 
which satisfies the same Hecke relation as $\lambdaup_j (n)$, namely,
\begin{align}\label{1eq: Hecke rel, eta}
\eta (n_1, s) \eta (n_2, s) = \frac 1 4 \sum_{ d | n_1, \shskip d| n_2 } \eta (n_1 n_2 / d^2, s),
\end{align}
for $(n_1 n_2, q) = \frO$.
For more details, see for example \cite[\S 3.4, 8.2]{EGM}.

Since $q$ is square-free, every cusp of $\Gamma = \Gamma_0 (q)$ is equivalent to one of $ 1/ \varv$ with $\varv | q$, and  $1/ \varv \sim 1/\varv'$ if and only if $(\varv) =   (\varv')$. %and  $1/q \sim \infty$.   
These may be verified by the arguments in \cite[\S 1.6]{Shimura}. For the cusp $\fra = 1 / \varv$ let  $\varw = q/\varv$ be the complementary divisor and define the scaling matrix 
\begin{align*}%\label{1eq: scaling matrix}
\sigma_{\fra} = \begin{pmatrix}
\sqrt {\varw } & \\
\sqrt {q \varv  } &   1 / \sqrt {\varw}
\end{pmatrix}.
\end{align*}
It is easy to show that \eqref{1eq: scaling} holds; see  \cite[\S 2.1, 2.2]{D-I-Kuz}.  Note that
\begin{align*}%\label{1eq: sigma a Gamma}
\sigma_{\fra}\- \Gamma = \left\{  \begin{pmatrix}
a/ \sqrt {\varw} & b / \sqrt {\varw} \\
c \sqrt {\varw} & d \sqrt{\varw} 
\end{pmatrix} : \begin{pmatrix}
a & b \\ c & d
\end{pmatrix} \in \SL_2 (\BZ), c + a \varv \equiv 0 (\mod q) \right\}.
\end{align*}
Hence the Eisenstein series for the cusp $\fra = 1/ \varv$ is given by
\begin{align*}
E_{\fra} (z, r; s) = \sum_{ \tau\, \in \, \Gamma_{ \infty} \backslash \sigma_{\fra}\- \Gamma}   r (\tau   (z, r))^{2s} = \frac 1 4 \lp \frac {r} {|\varw|} \rp^{2s} \underset{\sstyle (c, \, d \varw) = \frOO \atop{\sstyle \varv | c} }{\sum \sum} \big( |cz+d|^2 + |c|^2 r^2 \big)^{- 2s} .
\end{align*}
By the calculations in \cite[\S 3]{CI-Cubic}, we have
\begin{align}\label{1eq: E a = sum of E}
E_{\fra} (z, r; s) =  \frac {\mu (\varv) \zeta_{\BF, \shskip q} (2 s)  } {16 |q\varv  |^{ 2 s}}   \sum_{ \beta | \varv} \sum_{ \gamma | \varw} \mu (\beta \gamma) |\beta / \gamma  |^{2 s} E \lp \beta \gamma  z,  |\beta \gamma| r ; s \rp,
\end{align}
where $\zeta_{\BF, \shskip q} ( s)  $ is the local zeta-function
\begin{align}\label{1eq: zeta q}
\zeta_{\BF, \shskip q} (s) = \prod_{\sstyle (p) \supset (q)   }  \lp 1 - |p|^{- 2 s} \rp \-. 
\end{align}
In conclusion, by \eqref{1eq: Fourier Eisenstein}-\eqref{1eq: zeta q} we deduce that for $n \in \frO \smallsetminus \{0\}$, $(n, q) = \frO$,
\begin{align}\label{1eq: phi a (n, s)}
 \varphi_{\fra} (n, s)   = \frac { 2 \mu (\varv) \pi^{2s } \zeta_{\BF, \shskip q} (2s) \eta (n, s)  } { |q\varv  |^{2 s} \Gamma (2s  ) \zeta_\BF (2s )}. 
\end{align}

\subsection{The spectral Kuznetsov formula for $\Gamma_0  (q) \backslash \BH^3$}

%After introducing the Bessel functions for $\PGL_2 (\BC)$, we shall 

\subsubsection{Bessel functions for $\GL_2 (\BC)$}

\vskip 5 pt

Let $\mu \in \BC$ and $m \in \BZ$. We define 
\begin{equation}\label{0def: J mu m (z)}
J_{\mu,\shskip  m} (z) = J_{- 2\mu - \frac 12 m } \lp  z \rp J_{- 2\mu + \frac 12 m  } \lp  {\overline z} \rp,  
\end{equation}
%with $J_{\nu} (z)$ the Bessel function of the first kind of order $\nu$. 
with $J_{\nu} (z)$   the classical $J$-Bessel function of order $\nu$. 
The function $J_{\mu,\shskip  m} (z)$ is well-defined in the sense that the   expression on the right of \eqref{0def: J mu m (z)} is independent on the choice of the argument of $z$ modulo $2 \pi$. Next, we define
\begin{equation}\label{0eq: defn of Bessel, general}
\boldJ_{ \mu,\shskip  m} \lp z \rp = 
\left\{ 
\begin{split}
& \frac {2 \pi^2} {\sin (2\pi \mu)} \big(  J_{\mu,\shskip  m} (4 \pi   z) -  J_{-\mu,\shskip  -m} (4 \pi   z) \big), \hskip 10 pt \text {if } m \text{ is even},\\
& \frac {2 \pi^2 i} {\cos (2\pi \mu)} \big( J_{\mu,\shskip  m} (4 \pi   z) + J_{-\mu,\shskip  -m} (4 \pi   z) \big), \hskip 9 pt \text {if }  m \text{ is odd}.
\end{split}
\right.
\end{equation} 
It is understood that in the nongeneric case when  $4 \mu \in 2\BZ + m$ the right-hand side should be replaced by its limit. Moreover, $ \boldJ_{ \mu,\shskip  m} \lp z \rp $ is an even function when $m$ is even. See   \cite[\S 15.3]{Qi-Bessel}.

Let $H_{\nu}^{(1)} (z)$ and $ H_{\nu}^{(2)} (z) $ be the Hankel functions of order $\nu$. It follows from the connection formulae between $H_{\nu}^{(1)} (z)$, $H_{\nu}^{(2)} (z)$ and $ J_{  \nu} (z)$, $J_{-\nu} (z)$ in \cite[3.61 (1, 2)]{Watson} that
\begin{align}\label{0eq: J = H + H}
\boldJ_{\mu, \shskip  m} \lp z \rp = \pi^2 i \big( e^{2 \pi i \mu} H^{(1)}_{ \mu, \shskip  m   }   (4 \pi  z )   + (-)^{m+1} e^{- 2 \pi i \mu} H^{(2)}_{ \mu, \shskip  m   } ( 4 \pi  z )   \big).
\end{align}
with 
\begin{equation}\label{0def: H mu m (z)}
H^{(1, 2)}_{\mu,\shskip  m} (z) = H^{(1, 2)}_{ 2\mu + \frac 12 m } \lp  z \rp H^{(1, 2)}_{  2\mu - \frac 12 m  } \lp  {\overline z} \rp.
\end{equation}
It should be warned that the product in \eqref{0def: H mu m (z)} is {\it not well-defined} as function on $\BC \smallsetminus \{0\}$.
For latter applications, we have the following lemma. See \cite[\S 7.13.1]{Olver}.

\begin{lem}\label{lem: Hankel H(1, 2)}
	Let $K$ be a nonnegative integer.  We have
	\begin{align}\label{3eq: asymptotic H(1)}
	H^{(1)}_{\nu} (z) = \lp \frac 2 {\pi z} \rp^{1/2} e^{  i (z - \frac 1 2 \pi \nu - \frac 1 4 \pi)} \lp  \sum_{ k= 0}^{K-1} \frac {(-)^k (\nu, k) } {(2i z)^k}   + E^{(1 )}_K (z) \rp,
	\end{align}
	\begin{align}\label{3eq: asymptotic H(2)}
	H^{(2)}_{\nu} (z) = \lp \frac 2 {\pi z} \rp^{1/2} e^{- i (z - \frac 1 2 \pi \nu - \frac 1 4 \pi)} \lp  \sum_{ k= 0}^{K-1} \frac {  (\nu, k) } {(2i z)^k}   + E^{(2)}_K (z) \rp,
	\end{align}
	 with $ (\nu, k) = \Gamma \big(\nu + k + \tfrac 1 2 \big)/ k! \Gamma \big(\nu - k + \tfrac 1 2 \big)$,
	of which {\rm\eqref{3eq: asymptotic H(1)}} is valid when $z$ is such that $- \pi + \delta \leqslant \arg z \leqslant 2 \pi - \delta$, and {\rm\eqref{3eq: asymptotic H(2)}} when $- 2\pi + \delta \leqslant \arg z \leqslant   \pi - \delta$, $\delta $ being any acute angle, and 
	\begin{align}\label{3eq: estimates for E}
z^{\alpha}	(d/dz)^\alpha E_K^{(1, 2)} (z)  \Lt_{\shskip \delta, \shskip \alpha, \shskip K} (|\nu|^2+1)^K / |z|^{K}, 
	\end{align} 
	for $|z|\Gt  |\nu|^2+1$ and $\arg z$  in the range indicated as above.
\end{lem}

	 Moreover, by  \cite[Corollary 6.17]{Qi-Bessel}, we have the following integral representation
	\begin{equation}\label{1eq: integral representation, 1}
	\boldJ_{\mu, \shskip  m} \lp x  e^{i \phi} \rp   = 4 \pi  i^m  \int_{0}^\infty y^{4 \mu - 1} E \big(y e^{  i \phi} \big)^{-m} J_{m} \big( 4 \pi  x  Y \big(y e^{  i \phi} \big)  \big) d y,
	\end{equation}
	with
	\begin{align*}
	& Y (z) = \left| z + z^{-1}  \right|, %= \lp \lp y + y\- \rp \cos \lp \tfrac 1 2 \phi \rp + \lp y - y\- \rp \sin \lp \tfrac 1 2 \phi \rp \rp^{\frac 1 2}, \\
	\hskip 10 pt E (z) = \lp z + z\- \rp /\left| z + z^{-1}  \right|.
	\end{align*}
	The integral on the right of \eqref{1eq: integral representation, 1} is absolutely convergent if   $|\Re \, \mu| < \frac 1 8$. %Recall that  $[z] = z/|z|$. %Then \eqref{1eq: integral representation of J it} follows immediately from \eqref{1eq: integral representation, 1} by the change of variables $\nu = e^r$. 
%\end{proof}

In this article, we are mainly concerned with the {\it spherical} Bessel function \begin{equation}\label{0eq: defn of Bessel}
\boldJ_{\mu} (z) = \boldJ_{ \mu, \shskip 0} (z) = \frac {2 \pi^2} {\sin (2\pi \mu)}  \big( J_{- 2\mu   } \lp 4 \pi  z \rp J_{- 2\mu    } \lp 4 \pi  {\overline z} \rp    - J_{ 2\mu   } \lp 4 \pi  z \rp J_{ 2\mu }  ( 4 \pi  {\overline z} )   \big)   ,
\end{equation}
which is associated with the spherical principal series representation of $\PSL_2 (\BC)$ induced from the character   $\displaystyle \chiup_{\,\mu} \hskip -1 pt \begin{pmatrix}
z & \\ & z\-
\end{pmatrix} = |z|^{4 \mu}$. 
%See \cite{B-Mo,  Qi-Thesis}. 
Non-spherical Bessel functions $ \boldJ_{\mu, \shskip \pm 1} (z) $ and $ \boldJ_{\mu, \shskip\pm 2} (z) $ however will arise  in the derivatives of $ \boldJ_{\mu} (z) $. %and  be needed for their estimations. 

\vskip 5 pt

\subsubsection{The spectral Kuznetsov formula  for $\Gamma_0  (q) \backslash \BH^3$}\label{sec: Kuznestov}

%Let $n_1, n_2 \in \frO \smallsetminus \{0\}$. 
Let $h (t)$ be an even function satisfying the following two conditions,
\begin{itemize}
	\item[-] $h (t)$ is holomorphic on a neighborhood of the strip $|\Im\, t| \leqslant \sigma$,
	\item[-] $h (t) \Lt ( |t| + 1)^{-  \vartheta}$, %\footnote{Note that $- 4$ is improved to $- 3$ in \cite[\S 11]{B-Mo2} with slightly better extension method for the test function.}
\end{itemize}
for some $\sigma > 1/2 $ and $\vartheta > 3$.
In view of \cite[Theorem 11.3.3]{B-Mo2}, along with (\ref{1eq: relation Fourier Hecke}, \ref{1eq: phi a (n, s)}), we have the following spectral Kuznetsov trace formula, specialized to the spherical case. For $(n_1 n_2, q) = \frO$, %\footnote{The formula here is for  $\PGL_2 (\frO)$, so it is slightly different from that in \cite[Theorem 10.1]{B-Mo} for $\PSL_2 (\frO)$.},
\begin{equation}\label{1eq: Kuznetsov}
\begin{split}
\sideset{}{'}{\sum}_{\hskip -1 pt j}  \omega_j  &  \, h  ( t_j ) 
   \lambdaup_j  (n_1)   \overline {\lambdaup_j (n_2)} 
+  \frac 1 {2 \pi} \sum_{ \fra }   \int_{-\infty}^{\infty}   \omega_{\fra} (t) h (   t ) 
\eta \lp n_1, \tfrac 1 2 + i t \rp  \eta \lp n_2 , \tfrac 1 2 - i t \rp    d t \\
&  \hskip 17 pt  =    \frac { 1 } {  2 \pi^2  }    \sum_{ \epsilon \hskip 0.5 pt \in \frOO^{\times}} \delta_{  n_1, \,  \epsilon n_2}  \, H + \frac 1 { 8 \pi^2} \sum_{ \epsilon \hskip 0.5 pt \in \frOO^{\times} / \frOO^{\times  2} } \, \sum_{  q | c } \frac {S (  n_1, \epsilon n_2; c)} {|c|^2}  H \lp  \frac {\sqrt {\epsilon n_1 n_2} } {2 c} \rp ,
\end{split}
\end{equation}
where $\sum' $ restricts to the even Hecke cusps  forms for $\Gamma_0 (q)$, % (Hecke cusps forms for $\Gamma_0' (q)$),
\begin{align}
&\label{1eq: omegas} \omega_j = \frac { \left| \rho_j \lp 1 \rp \right|^2 t_j }   {\sinh (2 \pi t_j)},  \hskip 10 pt \omega_{\fra} (t) = \frac {\left| \varphi_{\fra} \lp 1, \tfrac 1 2 + it \rp \right|^2 t }   {\sinh (2 \pi t )}   , \\
& \label{1eq: H's} H = \int_{-\infty}^{\infty} h (  t  )    t^2  d t,  \hskip 6 pt H (z) =  \int_{-\infty}^{\infty} h (   t ) \boldsymbol{J}_{i t} (z)   t^2 d t, 
\end{align}
in which $\boldJ_{i t} (z)$ is  the Bessel function  as  in \eqref{0eq: defn of Bessel},  $\delta_{  n_1, \shskip \epsilon n_2} $ is the Kronecker $\delta$-symbol, $\frOO^{\times 2} = \left\{ \epsilon^2 : \epsilon \in \frO^{\times} \right\}  = \{1, -1\}$,  and $S (n_1, \epsilon n_2, c)$ is the Kloosterman sum defined by \eqref{1eq: Kloosterman}. 
It follows from \eqref{1eq: phi a (n, s)} that
\begin{align}\label{1eq: omega a}
\omega_{\fra} (t)   = \frac {2 \pi |\zeta_{\BF, \shskip q} (1 + 2 i t )|^2} {|q\varv|^2 |\zeta_\BF (1 + 2 i t )|^2  },
\end{align}
if $\fra = 1/\varv$. 

%We make three remarks here.  
\begin{rem}
	 The spectral weight function $h (2 it, p)$ in {\rm\cite[Theorem 11.3.3]{B-Mo2}} is chosen  to be supported on $\left\{ (2it, p) : p = 0 \right\}$ and here $h (t) = h (2 it, 0)$. When $p = 0$, the corresponding representations of $\SL_2 (\BR)$ are spherical and the cusp forms $f_j$ are their $\mathrm{SU}_2$-fixed vectors.
	 Moreover, for $n \in \frO' \smallsetminus\{0\}$, we have  $\rho (2 n) =    \pi\hskip 1 pt C (n; 2 it, 0) / \Gamma \lp 1 + 2 it \rp$ and similarly $\varphi_{\fra} \lp 2 n , \tfrac 1 2 + it \rp = \pi B_{\fra} (n; 2 it, 0) / \Gamma \lp 1 + 2 it \rp  $ if $C (n; 2 it, 0)$ and $B_{\fra} (n; 2 it, 0)$ are the Fourier coefficients in  \cite{B-Mo2}. 
	 When $i t_j$ is imaginary so that the infinite component of $f_j$ is a unitary principal series, it follows from Euler's reflection formula that
	 \begin{align*}
	  \omega_j =  \frac {|\rho_j(1) \Gamma (1+2it_j)|^2} {2 \pi} = \frac  {|\rho_j(1)|^2 \Gamma (1+2it_j) \Gamma (1-2it_j)} {2 \pi} = \frac { |\rho_j(1)|^2  t_j} {\sinh (2 \pi  t_j)}.
	 \end{align*}
	 When $s_j = i t_j$ is real, say in  $ \left(0, \frac 1 2 \right)$, and we are in the complementary series case {\rm(}although Selberg's conjecture asserts that this case does not exist{\rm)}, the formula in {\rm\cite[Theorem 11.3.3]{B-Mo2}} is not so accurate. As  suggested in {\rm\cite[Theorem 2.1]{Qi-Kuz}}, there should be a correction factor $G_{ s_j, 0} = \Gamma (1+2 s_j) / \Gamma (1-2 s_j)$ in the denominator, and hence
	 \begin{align*}
	 \omega_j =  \frac {|\rho_j(1) \Gamma (1+2 s_j)|^2} {2 \pi} \frac {\Gamma (1-2 s_j)} {\Gamma (1+2 s_j)} = \frac  {|\rho_j(1)|^2 \Gamma (1+2 s_j) \Gamma (1-2 s_j)} {2 \pi} = \frac { |\rho_j(1)|^2  s_j} {\sin (2 \pi  s_j)}.
	 \end{align*}
	The author however overlooked the simple fact that complementary series are spherical {\rm(}$d = 0$ in the notation of \cite{Qi-Kuz}{\rm)} and would like to take the chance here to correct and simplify the formula of $G_{s, \shskip d} $ in {\rm\cite[Theorem 2.1]{Qi-Kuz}} as follows,
	\begin{equation*}
	G_{s, \shskip d} = \left\{
	\begin{split}
&	1, \hskip 89 pt \text{ if } s \text{ is imaginary}, \\
&	\Gamma (1+2 s) / \Gamma (1-2 s), \hskip 10 pt \text{ if } s \in \lp 0, \tfrac 1 2 \rp \text{ and } d = 0.
	\end{split}\right.
	\end{equation*}
\end{rem}

\begin{rem}
In the real case,  the Bessel kernel is  
\begin{equation*}
\left\{ \begin{split}
&\frac {\pi i} {\sinh (\pi t)} \big( J_{2it} (4\pi x) - J_{- 2it} (4\pi x)\big), \hskip 10 pt \text{ if } \epsilon = 1, \\
& 4 \cosh (\pi t) K_{2it} (4 \pi x), \hskip 65.5 pt \text{ if } \epsilon = - 1.
\end{split} \right.
\end{equation*}
The Bessel functions $J_{2it }$ and $K_{2it}$  have quite different asymptotics on $\BR_+$  and must be treated separately.
However, unlike the real case, the unit $\epsilon$ here does not play an essential role. 
\end{rem}

  For $q = q' q''$ and    Hecke newform $f \in H^{\star} (q')$, let $S (q'' ; f) $ denote  the linear space spanned by the forms $f_{| \hskip 0.5 pt d} (z, r) = f (d z, |d| r)$, with $d | q''$. The space of    cusp forms  for $\Gamma_0 (q)$   decomposes into the orthogonal sum of  $S (q'' ; f) $.  By the calculations in  \cite[\S 2]{ILS-LLZ}, one may construct an orthonormal basis  of $S (q'' ; f) $ in terms of $f_{| \hskip 0.5 pt d}$.   Recall that $ (n_1  n_2 , q) = \frO $. Using this collection of bases as our $\{ f_j \}$, the sum in \eqref{1eq: Kuznetsov} can be arranged into a sum over the even newforms in $ B^{\star} (q) = \bigcup_{q' | q}  H^{\star} (q')$. With ambiguity, we denote by $  f_j$ the newforms in $B^{\star} (q)$ (instead of the cusp forms in an orthonormal basis for $\Gamma_0(q)$), let $t_j$, $\lambdaup_j (n)$ be as before, and denote by $\omega^{\star}_j$  the new weights. %For $ f_{j, q' } \in H^{\star} (q') $ let  $t_{j, q'}$ be its spectral parameter, $ \lambdaup_{j, q'}  (m)$ its Fourier coefficients and  $\omega_{ j, q'}^{\star} $ be its new weight. 
  In addition, let
  \begin{align*}
  \omega^{\star} (t) = \sum_{\fra} \omega_{\fra} (t).
  \end{align*}
  We then obtain for $(n_1  n_2 , q) = \frO$,
   \begin{equation}\label{1eq: Kuznetsov, new}
   \begin{split}
    \sideset{}{'}{\sum}_{ j }  \omega_{j }^{\star}   h  & \, ( t_{j } )  
    \lambdaup_j  (n_1)   \overline {\lambdaup_j (n_2)} 
   +  \frac 1 {2 \pi}    \int_{-\infty}^{\infty}   \omega^{\star} (t) h (   t ) 
   \eta \lp n_1, \tfrac 1 2 + i t \rp  \eta \lp n_2 , \tfrac 1 2 - i t \rp    d t \\
   & \hskip 4 pt  =    \frac { 1 } {  2 \pi^2  }    \sum_{ \epsilon \hskip 0.5 pt \in \frOO^{\times} } \delta_{ n_1, \,  \epsilon n_2}  \, H + \frac 1 { 8 \pi^2} \sum_{ \epsilon \hskip 0.5 pt \in \frOO^{\times} / \frOO^{\times  2}} \, \sum_{  q | c } \frac {S ( n_1, \epsilon n_2; c)} {|c|^2}  H \lp  \frac {\sqrt {\epsilon n_1 n_2} } {2 c} \rp .
   \end{split}
   \end{equation}
  Following \cite[\S 2]{ILS-LLZ}, we may derive the formula 
   \begin{align}
\omega_{j }^{\star} = \frac { \pi Z_{q} (1,   f_j ) } {\mathrm{Vol} (\SL_2 (\frO) \backslash \BH^3 ) |q|^2 Z (1, f_j ) },
   \end{align}
   with 
   \begin{align}
    Z (s,  f_j )  = \frac 1 4 \sum_{n \neq 0 }  \frac {\lambdaup_j (n^2)} {|n|^{2s}}, \hskip 10pt Z_q (s,  f_j ) = \frac 1 4 \sum_{n | q^{\infty} }  \frac {\lambdaup_j (n^2)} {|n|^{2s}}.
   \end{align}
%  The spectral weight $\omega_{j }^{\star} $ has an expression that involves $L (1, \mathrm{Ad}^2 f_j )$ but its explicit form does not matter here. 
According to \cite[\S 7.1, Theorem 1.1]{EGM}, we have $ \mathrm{Vol} (\SL_2 (\frO) \backslash \BH^3) = 2 \zeta_{\BF} (2) / \pi^2 $. We also note that $ Z_q (s,   f_j ) $ is a finite Euler product and $Z_q (1,   f_j ) $ is usually dispensable. Moreover, for $f_j \in H^{\star} (q')$, 
we have
\begin{align}
 Z (s,   f_j ) =  \zeta_{\BF, \shskip q'} (2s) \zeta_{\BF} (2s)\- L (s, \mathrm{Sym}^2 f_j ) . 
\end{align}
For our purpose, we only need the lower bound
  \begin{align}\label{1eq: lower bound for omega j}
  \omega_{ j}^{\star}  \Gt   |q|^{-2- \varepsilon} ( |t_j| + 1)^{- \varepsilon}. 
  \end{align}
  Moreover, in view of \eqref{1eq: omega a}, we have \begin{align}
  \omega^{\star} (t) =  \frac {2 \pi \nu (q) |\zeta_{\BF, \shskip q} (1 + 2 i t )|^2} {|q |^4  |\zeta_\BF (1 + 2 i t )|^2  },
  \end{align} with the standard definition $\nu (q) = |q|^2 \prod_{ \sstyle  (p) \supset (q)    } \lp 1 + 1/ |p|^2 \rp $. We   also have   the lower bound
  \begin{align}\label{1eq: lower bound for omega (t)}
  \omega^{\star} (t)  \Gt |q|^{-2 - \varepsilon} \min \left\{ |t|^{-\varepsilon}, |t|^2 \right\} .
  \end{align}
 These two lower bounds are consequences of the estimate in  \cite[Theorem 1]{Molteni-L(1)}  
  applied to $L (s, \mathrm{Sym}^2 f_j)$ and $\zeta_{\BF} (s + 2it)$, with the latter viewed as the $L$-function for the Hecke character %Gr\"o\ss encharakter 
  $\chiup_{-2it} ((n)) = | n |^{-4it}$. %, $n \in \frO \smallsetminus \{0\}$. 
  Suffice it to say, \cite[Theorem 1]{Molteni-L(1)} is a very broad generalization of \cite[Theorem 2]{Iwaniec-L(1)} to a large class of $L$-functions which satisfy an assumption much weaker than  the Ramanujan hypothesis.

%So it will be harmless for our purpose of estimation to simply omit these weights.

\subsection{Hecke-Maass cusp forms for $\SL_3 (\frO)$}
Let $\pi$ be a Hecke-Maass cusp  form (or a cuspidal representation) for $\SL_3 (\frO)  $. Suppose that the Archimedean Langlands parameter of $\pi$ is the triple $(\mu_1, \mu_2, \mu_3)$, satisfying $\mu_1 + \mu_2 + \mu_3 = 0$.   Let $A (n_1, n_2)$, with $n_1$, $n_2 \in \frO \smallsetminus \{0\}$, denote  the Fourier coefficients of $\pi$.  Since all the {$\displaystyle {\begin{pmatrix}
\epsilon_1 \epsilon_2 & & \\ & \epsilon_1 & \\ & & 1
\end{pmatrix}}$ }   may be generated from the diagonal elements in $\SL_3 (\frO)$ and the central elements $ \epsilon I_3$, we infer that $A (\epsilon_1 n_1, \epsilon_2 n_2) = A ( n_1,  n_2)$ for all $\epsilon_1, \epsilon_2 \in \frOO^{\times}$.  Further, we assume that $\pi$ is Hecke normalized in the sense that $A \lp 1, 1 \rp = 1$. We have the multiplicative relation 
\begin{align}\label{1eq: mult relation}
A (n_1 m_1, n_2 m_2) = A (n_1, n_2) A(m_1, m_2), \hskip 15pt (n_1 n_2, m_1 m_2) = \frO,
\end{align}
and the Hecke relation
\begin{align}\label{1eq: Hecke relation}
A (n_1, n_2) = \frac 1 4 \sum_{ d|  n_1, \shskip d | n_2 } \mu (d) A \lp n_1/d, 1 \rp A \lp 1, n_2 / d \rp. 
\end{align}

In this article, we assume that $\pi$ is self-dual so that  $(\mu_1, \mu_2, \mu_3) = ( \mu, 0, - \mu)$ and $A(n_1, n_2) = A (n_2, n_1)$($= \overline { A(n_1, n_2) }$). It is known by \cite{GJ-GL(2)-GL(3)} that $\pi$ comes from the symmetric square lift of a Hecke-Maass form on $\GL_2 (\frO)$. It follows from  the Kim-Sarnak bound  for $\GL_2 (\frO) $ in \cite{Blomer-Brumley} that $\mu$ is either  imaginary or real with  
\begin{align}\label{1eq: Kim-Sarnak, GL3}
	|\Re \, \mu | \leqslant \tfrac 7 {32},
\end{align} 
and, together  the Hecke relation \eqref{1eq: Hecke relation},  that
\begin{align}\label{3eq: K-S Fourier}
	A (n_1, n_2) \leqslant |n_1 n_2|^{7/16 + \varepsilon}.
\end{align}

%\subsubsection{Average of Fourier coefficients}
%Finally, thanks to the Rankin-Selberg theory, we know that $A (n_1, n_2)$  obey the Ramanujan conjecture on average. The Rankin-Selberg $L$-function
%\begin{equation*}
%L (s, \pi \otimes  \pi ) =    \mathop {\sum \sum}_{ (n_1) ,  \, (n_2) \neq 0} \frac { |A  (n_1, n_2) |^2 } {    \left| n_1^2 n_2 \right|^{2 s} },  
%\end{equation*}
%initially convergent for $\Re s$   large, has a meromorphic continuation to the whole complex plane with only a simple pole at $s = 1$  (see \cite{J-S-Rankin-Selberg}). It follows that
Rankin-Selberg theory (see \cite{J-S-Rankin-Selberg}) implies the bound
\begin{align} \label{3eq: R-S Fourier}
	\mathop {\sum  }_{   | n  |   \leqslant X } |A  (n , 1) |^2   \Lt_{ \, \pi }    X^2, \hskip 10 pt X \geqslant 1 .
\end{align}
%and moreover Cauchy-Schwarz implies
%\begin{align}\label{1eq: average of Fourier coeff}
%\mathop {  \sum}_{   |   n_2  |   \leqslant X } | A  (n_1, n_2) |    = O_\pi \lp |n_1|^2   X^2 \rp.
%\end{align}%\marginpar{\footnotesize I actually need the average estimates over sectors. Does this follow from those over disks? Otherwise, I need to assume Ramanujan's conjecture.\red{Solved! Simply sum over $c$ first!} }
From this and the Kim-Sarnak bound \eqref{3eq: K-S Fourier} we deduce that for $a_1, a_2 \in \frO$ and $X \geqslant 1$ that
\begin{align}\label{3eq: bounds for Fourier, 1}
	\sum_{ |n| \leqslant X} |A(a_1n, a_2)|^2 \Lt |a_1 a_2|^{7/8 + \varepsilon} X^2.
\end{align}
Indeed, in view of \eqref{1eq: mult relation}, the left-hand side is bounded by
\begin{align*}
	\sum_{ a | (a_1 a_2)^{\infty} } \sum_{ \sstyle |n| \leqslant X/ |a| \atop \sstyle (n, \shskip a a_1 a_2  )  = \frO} |A(a a_1 n, a_2) |^2 \leqslant \sum_{ a | (a_1 a_2)^{\infty} } | A(a a_1, a_2)|^2 \sum_{ \sstyle |n| \leqslant X / |a|} |A(n, 1)|^2,
\end{align*}
and \eqref{3eq: bounds for Fourier, 1} then follows from  \eqref{3eq: K-S Fourier} and \eqref{3eq: R-S Fourier}. By Cauchy-Schwarz, we have
\begin{align}\label{3eq: bounds for Fourier, 2}
	\sum_{ |n| \leqslant X} |A(a_1n, a_2)|  \Lt |a_1 a_2|^{7/16 + \varepsilon} X^2.
\end{align}

\subsection{The \Voronoi summation formula for $\SL_3 (\frO)$}\label{sec: Voronoi SL3}

The $\GL_3$ \Voronoi summation formula over $\BQ$ was first discovered by Miller and Schmid \cite{Miller-Schmid-2006}. In the ad\`elic setting, the extension of the formula over an arbitrary number field may be found in the work of Ichino and Templier \cite{Ichino-Templier}. 

\vskip 5 pt

\subsubsection{Hankel transforms and Bessel kernels for $\GL_3 (\BC)$}\label{sec: Hankel transform}
For a smooth compactly supported function $\textit{w}$ on $ \BC \smallsetminus \{0\}$, we associate a function $W$ on $\BC \smallsetminus \{0\}$ such that the following sequence of identities are satisfied, (see \cite[(1.1)]{Ichino-Templier} or \cite[Theorem 3.15]{Qi-Bessel})
\begin{equation}\label{1eq: Hankel transform identity, C}
\EuScript M _{-m} W (2 s ) = G_{m} (s, \pi)  \EuScript M _m \textit{w} ( 2 (1-s) ), \hskip 10 pt m \in \BZ,
\end{equation}
where $\EuScript M_m$ is the Mellin transform of order $m$,
\begin{equation*}%\label{1def: Mellin transform over complex}
\EuScript M _m \tw (s) = \int_{\BC \smallsetminus \{0\}} \tw (z) (z/|z|)^{ m} |z| ^{ s  - 2} d z = 2 \int_0^\infty \int_0^{2 \pi} \tw \big( x e^{i \phi} \big) e^{ i m \phi} d\phi \cdot x^{ s - 1} d x,
\end{equation*}
in which % $[z] = z/|z|$, 
%$d z = 2 x dx\hskip 1pt d \phi$ is twice  the Lebesgue measure on $\BC$, and 
$G_m (s, \pi)$ is the gamma factor\footnote{This   notation is in essence due to Miller and Schmid \cite{Miller-Schmid-2006}, but it is indeed the gamma factor $\gamma (s, \pi_{ \infty} \otimes \eta^m)$ in representation theory. Here $\pi_{ \infty} $ is the representation of $\GL_2 (\BC)$ for $\pi$ at the complex place and $\eta$ is the character of $\BCx$ defined by $\eta (z) =   z/|z|$.} for $\pi$ of order $m$ given by
\begin{equation*}%\label{1def: G m (s)}
G_{m} (s, \pi) = \prod_{l=1,\shskip 2,\shskip 3} i^{|m| } (2\pi)^{1-2 s } \frac { \Gamma \lp s - \mu_l + \frac 1 2{|m|}   \rp} { \Gamma \lp 1 - s + \mu_l + \frac 1 2{|m|}   \rp }.
\end{equation*}
$W$ is called the Hankel transform of $\tw$ (of index $(\mu_1, \mu_2, \mu_3)$). It is known that the Hankel transform admits an integral kernel $ \boldJ_{(\mu_1,\, \mu_2,\, \mu_3)}  $ (see \cite[\S 3.4]{Qi-Bessel}), namely,
\begin{equation}\label{1eq: Hankel, integral}
W (u)  =  \int_{\BC \smallsetminus \{0\}} \tw (z) \boldJ_{(\mu_1,\, \mu_2,\, \mu_3)} ( u z   ) d z.
\end{equation}
%Let us denote the Bessel kernel $\boldJ_f (z) = \boldJ_{(\mu_1,\, \mu_2,\, \mu_3)} (z)$ for simplicity.

\subsubsection{The \Voronoi summation formula for $\SL_3 (\frO)$}\label{sec: Voronoi}
Translating the ad\`elic form of the \Voronoi  summation formula in \cite[Theorem 1]{Ichino-Templier} into the classical language, we have the following formula in terms of classical quantities. See \cite[Proposition 3.4 and Remark 3.5]{Qi-Wilton}.

\begin{prop}\label{prop: Voronoi}
	Let $\tw $ be a smooth compactly supported function on $\BC \smallsetminus \{0\}$. For $n_1, n_2 \in \frO  \smallsetminus \{0\}$, let $A(n_1, n_2)$ be the $(n_1, n_2)$-th Fourier coefficient of a Hecke-Maass form $\pi$ for $\SL_3 (\frO)$. Let $a, \widebar a, c \in \frO$ be such that $c \neq 0$, $(a, c) = \frO$ and $a \widebar a \equiv 1 (\mod c)$. Then we have
	\begin{align*}
	\sum_{n_2 \, \in \frO  \smallsetminus \{0\}} & A (n_1, n_2)   e \lp  \mathrm{Re} \lp \frac { a n_2} { c} \rp \rp \tw \lp \frac  {n_2} 2  \rp \\
	& = \frac 1 {4 |c^2 n_1|^2 }\sum_{ n_3 |  c n_1   } |n_3|^2 \sum_{ n_4 \, \in \frO  \smallsetminus \{0\} } A (n_4, n_3) S \lp  \widebar a n_1,   n_4; c n_1 / n_3 \rp W \bigg(  \frac {n_3^2 n_4 } { 4 c^3 n_1} \bigg),
	\end{align*}
	where $S \lp \widebar a n_1, n_2; c n_1 / n_3 \rp$ is a Kloosterman sum as defined  in \eqref{1eq: Kloosterman} and $W$ is the Hankel transform of $\tw$ given by \eqref{1eq: Hankel transform identity, C} or \eqref{1eq: Hankel, integral}.
\end{prop}

\begin{rem}\label{rem: normalization of Hankel}
It should be warned that our normalization of Hankel transforms in \cite{Qi-Bessel} is slightly different  from   that in \cite{Miller-Schmid-2004-1, Miller-Schmid-2006}  in order to be consistent with the  Fourier transform   and the classical Hankel transform when the rank is one and two{\rm;} see \cite[Remark {\rm3.5}]{Qi-Wilton} for the difference. As such, the look of the \Voronoi formula here is also different from those in the literature.
\end{rem}

%Subsequently, we shall assume that $\pi$ is self-dual so that $(\mu_1, \mu_2, \mu_3) = (- \mu, 0, \mu)$ and $A(n_1, n_2) = A (n_2, n_1)$($= \overline { A(n_1, n_2) }$). It is known by \cite{GJ-GL(2)-GL(3)} that $\pi$ comes from the symmetric square lift of a Hecke-Maass form on $\GL_2 (\frO)$. %of Langlands parameter $\lp - \frac 1 2 \mu, \frac 1 2 \mu \rp$. 

\subsection{$L$-functions   
	 and their approximate functional equations} \label{sec: L-functions}
Let $q $ be   squarefree such that $\RN (q) > 1$ and $\mathrm{N} (q) \equiv 1 (\mod 8)$. Let $\chiup$ be the primitive quadratic Hecke character of conductor $q$ and frequency $0$. 
Let $f_j  $ be an even Hecke-Maass newform for $\Gamma_0 (q') $ with $q' | q$.  Let $E (w, s)$ be the Eisenstein series for $\SL_2 (\frO)$. Let $\pi$ be a fixed self-dual Hecke-Maass form for $\SL_3(\frO)$.

\subsubsection{$L$-functions $L (s, \pi \otimes \chiup)$,  $ L (s, \pi \otimes f_j \otimes \chiup) $ and $L (s, \pi \otimes E \lp \cdot, \tfrac 1 2 + i t \rp \otimes \chiup)$} 
The $L$-function attached to the twist $\pi \otimes \chiup$ is defined by
\begin{equation}
L (s,  \pi \otimes \chiup ) = \sum_{(n) \neq 0 } \frac {A(1, n) \chiup (n)} {\RN(n)^{  s} }.
\end{equation}
The Rankin-Selberg $L$-function $L (s, \pi \otimes f_j \otimes \chiup)$ is  defined by 
\begin{align}\label{1eq: L (pi f chi)}
L (s, \pi \otimes f_j  \otimes \chiup) %= \frac 1 {16} \underset{n_1, n_2 \in \frO \smallsetminus \{0\} }{\sum \sum} \frac {\lambdaup_j (n_2) A (n_1, n_2) } {\left|n_1^2 n_2\right|^{2 s} } 
= \underset{(n_1), (n_2) \neq 0 }{\sum \sum} \frac { A (n_1, n_2) \lambdaup_j  (n_2) \chiup (n_2)} {\RN  ( n_1^2 n_2  )^{  s} }. 
\end{align} 
This definition requires only the Hecke eigenvalues $\lambdaup_j  (n)$ with $(n, q) = \frO$.

Recall that the root number $ \varepsilon (\chiup) = 1$ (see \S \ref{sec: Hecke character}) and that $\pi$ is a symmetric lift of a $\GL_2 (\frOO)$-automorphic form (\cite{GJ-GL(2)-GL(3)}). By the results on local $\varepsilon$-factors in \cite[\S 1.2, 3]{Schmidt-Local} and \cite[\S 4]{Knapp}, we infer that the conductors of $ \pi   \otimes \chiup $ and $ \pi \otimes f_j  \otimes \chiup $ are $q^3$ and $q^6$, respectively, and that both of the root numbers $\varepsilon (\pi \otimes \chiup) = \varepsilon (\pi \otimes f_j  \otimes \chiup) =  1$. 

The completed $L$-function for $ \pi \otimes \chiup$ is
$\Lambda (s,  \pi \otimes \chiup ) = \RN(q)^{3 s /2} \gamma (s, \pi \otimes \chiup) L (s,  \pi \otimes \chiup ) ,$
where  $\gamma (s, \pi \otimes \chiup) = \gamma (s, \pi ) = \gamma (s)  $, with 
\begin{equation}\label{1eq: gamma (s)}
\gamma (s ) = \gamma (s, \mu) =  2^3 (2 \pi )^{-3 s} \Gamma (s+\mu) \Gamma (s) \Gamma (s-\mu). 
\end{equation}
$\Lambda (s,  \pi \otimes \chiup )$ is entire and  has the following functional equation
\begin{align*}
\Lambda (s,  \pi \otimes \chiup ) = \Lambda   (1-s,  \pi \otimes \chiup ).
\end{align*} 
Let $ \Lambda (s, \pi \otimes f_j \otimes \chiup) = \RN(q)^{3 s} \gamma (s , \pi \otimes f_j \otimes \chiup) L (s, \pi \otimes f_j \otimes \chiup),$
with the gamma factor $\gamma (s, \pi \otimes f_j \otimes \chiup) = \gamma (s  , t_j )$, $\gamma (s, t )$ defined by
\begin{align}\label{1eq: gamma (s, t)}
\gamma (s, t ) =  \gamma (s-it) \gamma (s+it).
% 2^6 (2 \pi)^{- 6 s}   \Gamma (s - it  + \mu  )  \Gamma (s   - it  ) \Gamma (s  - it  - \mu  )  
%  \Gamma (s + it  + \mu  )  \Gamma (s  + it   ) \Gamma (s  + it  - \mu  ).
\end{align}
$\Lambda (s, \pi \otimes f_j  \otimes \chiup )$ is also entire and satisfies the functional equation
\begin{align*}%\label{1eq: functional equation}
\Lambda (s, \pi \otimes f_j  \otimes \chiup) = \Lambda (1 - s, \pi \otimes f_j  \otimes \chiup).
\end{align*}

Similar as \eqref{1eq: L (pi f chi)},  we define %the $L$-function $L \lp s, \pi \otimes E \lp \cdot, \tfrac 1 2 + i t \rp \otimes \chiup \rp$, 
\begin{align}
L \lp s, \pi \otimes E \lp \cdot, \tfrac 1 2 + i t \rp \otimes \chiup \rp = \underset{(n_1), (n_2) \neq 0 }{\sum \sum} \frac {   A (n_1, n_2) {\eta} \lp n_2, \frac 1 2 + it \rp \chiup (n_2)} {\RN  ( n_1^2 n_2  )^{  s} }.
\end{align}
We have
\begin{align}\label{1eq: L (f x E) = L(f) L(f)}
L \lp s, \pi \otimes E \lp \cdot, \tfrac 1 2 + i t \rp \otimes \chiup \rp = L \lp s+it, \pi \otimes \chiup \rp L \lp s-it, \pi \otimes \chiup \rp,
\end{align}
%and $L \lp s, \pi \otimes E \lp \cdot, \tfrac 1 2 + i t \rp \otimes \chiup \rp$ satisfies a functional equation similar to   that for $L (s, \pi \otimes f_j \otimes \chiup)$. Moreover, one observes that
and hence
\begin{align}\label{1eq: L (f x E) = |L (f)|2}
L \lp \tfrac 1 2 , \pi \otimes E \lp \cdot, \tfrac 1 2 + i t \rp \otimes \chiup \rp = \left| L \lp \tfrac 1 2 + it, \pi \otimes \chiup \rp \right|^2.
\end{align}

\vskip 5 pt

\subsubsection{Approximate functional equations for  $ L (s, \pi \shskip \otimes f_j \shskip \otimes \shskip \chiup) $ and $L (s, \pi \shskip \otimes \shskip \chiup)$} \label{sec: afeq}

Following Blomer \cite{Blomer}, for positive integer $A'$, we introduce  
\begin{align}\label{1eq: def p(s, t)}
 p (s, t ) = p (s, t, \mu) =  \prod_{ k=0}^{A'-1} \prod_{\pm}     \lp s \pm it + \mu + k \rp (s \pm i t + k) \lp s \pm it - \mu + k \rp      
\end{align}
so that $p  (s, t )$ kills the right most   $A' $ many poles of each of the   gamma factors in $ \gamma  (s, t ) $  defined by (\ref{1eq: gamma (s)}, \ref{1eq: gamma (s, t)}). 

%In order to simplify the expositions, we have abused the notation    $ \gamma (s, t )  $ and $ p (s, t ) $, but their choices will be clear from the context. 

We have the following approximate functional equation for   $L (s, \pi \otimes f_j \otimes \chiup)$, (see 
\cite[Theorem 5.3]{IK}) 
\begin{equation}
\label{1eq: approximate functional equation, 1} 
 L \left(\tfrac 1 2,   \pi \otimes f_j \otimes \chiup\right)   =  2 \underset{(n_1), (n_2) \neq 0 }{\sum \sum} \frac { A (n_1, n_2) \lambdaup_j (n_2) \chiup (n_2) } { |n_1^2 n_2 |  }    V \big(  |n_1^2 n_2 | / |q|^3, t_j \big)   , 
\end{equation}
with %$V(y, t )$ is defined by
\begin{equation}%\label{1eq: def of V (y, t)}
\begin{split}
V (y, t ) = \frac 1  {2 \pi i} \int_{(3)} %\left( \cos \lp \frac {\pi \varv} {4A} \rp \right)^{- 48 A} 
  G  (\varv, t)  y^{ - 2 \varv} \frac { d \varv } {\varv}, \hskip 10 pt y > 0,  
\end{split}
\end{equation}
in which 
\begin{align}\label{1eq: def G}
  G (\varv, t) = \frac {\gamma \left(\frac 1 2 + \varv , t   \right)  }  {\gamma \left(\frac 1 2, t   \right)   } \cdot \frac { p \left(\frac 1 2 + \varv , t   \right) p \left(\frac 1 2 - \varv , t   \right)   \exp \lp {\varv^2} \rp  }  {  p \left(\frac 1 2, t    \right)^2 }.
\end{align} 
Note that the second quotient in \eqref{1eq: def G}   is even in $\varv$ and is equal to $1$ when $\varv = 0$.
%$ \gamma (s, t) $ and $p (s, t)$ defined by (\ref{1eq: gamma (s, t), 0}, \ref{1eq: gamma (s), 0}, \ref{1eq: def p(s, t), 0}) in the former and by (\ref{1eq: gamma (s, t)}, \ref{1eq: gamma (s)}, \ref{1eq: def p(s, t)}) in the latter.
Similarly, the   approximate functional equation for  $L \big( s, \pi \otimes E \lp \cdot, \tfrac 1 2 + i t \rp \otimes \chiup \big)$, along with  \eqref{1eq: L (f x E) = |L (f)|2}, yields 
\begin{align}
\label{1eq: approximate functional equation, 2} 
\left| L \lp \tfrac 1 2 + it, \pi \otimes \chiup \rp \right|^2  =  2  \underset{(n_1), (n_2) \neq 0 }{\sum \sum}   \frac {A (n_1, n_2) \eta \lp n_2, \frac 12 +it \rp  \chiup (n_2) } {\left|n_1^2 n_2\right| }    V \big( \left|n_1^2 n_2\right| / |q|^3, t \big)  \hskip - 1 pt .
\end{align}

\begin{lem}\label{lem: afq}
	%Let  $T  > 1 $ and $\varepsilon > 0$. 
	Let  $ U  > 1 $ and $\varepsilon > 0$. Let $A$ be a positive integer.  
	
	{\rm(1).}  We have
	\begin{align}\label{1eq: derivatives for V(y, t), 1}
	p \left(\tfrac 1 2, t \right)^2 V (y, t ) \Lt_{\,A, \shskip A' } (|t| + 1)^{12 A'} \lp 1 + \frac {y} {(|t| + 1)^{6}} \rp^{-A}.
	%\label{1eq: derivatives for V(y, t), 2}
	%y^{\alpha} \frac {\partial^{\alpha}} {\partial y^{\alpha}} V(y, t ) = \delta_{\alpha } + O \lp \lp \frac {y} {(|t| + 1)^6} \rp^c \rp,
	\end{align} 
	and
	\begin{align}\label{1eq: approx of V}
	V  (y, t) = \frac 1 {2   \pi i } \int_{ \varepsilon - i U}^{\varepsilon + i U}  G (\varv, t)   y^{- 2 \varv}   \frac {d \varv} {\varv} + O_{\varepsilon} \lp \frac {(|t|+1)^{6 \varepsilon} } {y^{ 2 \varepsilon} e^{U^2 / 2} } \rp.
	\end{align} 
	
	{\rm(2).}  When $\Re \, \varv > 0$, the function $ G (\varv, t) p \lp \tfrac 1 2, t \rp^2$ is even in $t$ and holomorphic in the region    $|\Im \, t | < A' + \frac {9} {32} = A' + \frac 1 2 - \frac 7 {32}$, and satisfies in this region the uniform bound
	\begin{align}\label{1eq: bound for p G}
	 %(\partial / \partial t)^{k} 
	   G (\varv, t) p \lp \tfrac 1 2, t \rp^2  \Lt_{\, A', \shskip  \Re\shskip \varv} (|t|+1)^{12 A'  + 6 \hskip 0.5 pt \Re \shskip \varv }.  
	\end{align}
\end{lem}

\begin{proof}
	\eqref{1eq: derivatives for V(y, t), 1} may be found in 
	\cite[Proposition 5.4]{IK}. 
	The expression of $V (y, t)$   in \eqref{1eq: approx of V} is due to Blomer, \cite[Lemma 1]{Blomer}. %Note that  no pole is crossed when shifting the integration contour from $\Re\, \varv = 3$ to $\Re \, \varv = \varepsilon$ because of the Kim-Sarnak bound  for $i t$ and $\mu$. 
	The estimate in \eqref{1eq: bound for p G} is a   consequence of  Stirling's approximation. %for the Gamma function.
\end{proof}
	
	\delete{and the polyGamma functions, namely, $\psi (s) = \Gamma'(s)/\Gamma (s)$ and its derivatives. In particular, for $\Re \, s = \sigma,   \Re \,\varv = \beta$ with  $\sigma > - 1$, $\beta + \sigma > - 1$, we have
	\begin{align*}
	\frac {\partial^m} {\partial s^m} \lp \psi (s+ \varv) + \frac 1 {s+\varv} - \psi (s ) - \frac 1 {s} \rp \Lt   \frac {|\varv|} {|s + 1|^{m+1}} \sum_{ l = \min \{1, m\} }^{m } \left|\frac {s+1} {s+\varv+1}\right|^{l},
	\end{align*}
	and
	\begin{align*}
		\frac {(s+\varv)\Gamma (s + \varv) } {  s \shskip \Gamma (s)} & \Lt \frac { |s+\varv+1|^{\sigma + \beta - 1/2} } {|s+1|^{\sigma - 1/2} } \exp \lp \pi (|s+1|-|s+\varv +1|) / 2 \rp \\
		& \Lt  (|s| + 4)^{  \beta } \exp \lp {\pi |\varv| / 2} \rp,  
	\end{align*}
	with   the implied constants depend only on $\sigma$ and $\beta$ (for the latter estimate, see the proof of \cite[Proposition 5.4]{IK}). }

It follows from (\ref{1eq: derivatives for V(y, t), 1}) that  $V \lp |n_1^2 n_2|/|q|^3 , t \rp$ decays  rapidly for   $|n_1^2 n_2 | > (|q| (|t|^2+1) )^{3+\varepsilon}$,   so  we can effectively take  the sums in  \eqref{1eq: approximate functional equation, 1} and \eqref{1eq: approximate functional equation, 2} above to be finite. 

%The expression of $V (y, t)$ given by (\ref{1eq: afe, V = sum}, \ref{1eq: Vkl}) is mainly due to Young in \cite[\S 5]{Young-Cubic} for the purpose of separating the variables $y$ and $t$. 
%In view of the  choice of the spectral weight $h(t) = e^{- (t-T)^2/M^2} + e^{-(t+T)^2 / M^2}$ given  in Theorem \ref{thm: main}, we may restrict ourselves to   $|t \mp T| \leqslant M \log T$  as $ h (t)$ is very small outside of these two intervals, so  the assumption $|t \mp T| \leqslant T^{1-  \varepsilon}$  is justified. %\marginpar{\footnotesize \red{There seems to be something wrong here when $T$ is considered as fixed! Blomer's version however is more complicated.}} 
 %Second,   we shall choose $U = \log (|q| T) $ in later applications. 

%where $\delta_0 = 1$, $\delta_{\alpha} = 0$ if $\alpha \neq 0$, $0 < 3 c \leqslant \frac 1 2 - \frac {21} {64}  $, and the implied constants depend only on $\alpha$, $A$, $c$ and $\mu$.
%Note that we have the bounds $|\Im \, t_j | \leqslant \frac 7 {64}$ and $|\Re \, \mu | \leqslant \frac 7 {32}$  towards the Selberg-Ramanujan conjecture\marginpar{\footnotesize Is it proven that $t_j$ are real for $\PGL_2 (Z[i])$? Does Vigneras' proof work here?}
%(see \cite{Blomer-Brumley} and the references therein).

%The approximate functional equation for $L \lp s, \pi \otimes E \lp \cdot, \tfrac 1 2 + i t \rp \otimes \chiup \rp$ is identical, namely,

\section{Properties of the rank-two Bessel kernel and the Bessel integral}

In this section, we shall analyze the  Bessel kernel $\boldJ_{it} (z)$ and the Bessel integral $H (z) $ that occur in the $\GL_2$ Kuznetsov trace formula. It should be noted that the estimates for $\boldJ_{it} (z)$ obtained here are very crude in the $t$ aspect.

\subsection{Analysis of the Bessel kernel}

\begin{lem}\label{lem: properties of J}
	Let $t$ be real. Let $K$ be a nonnegative integer. We have
	\begin{align}\label{1eq: J = W + W}
	\boldJ_{it} (z) =   {e (4 \Re \, z)}   \boldsymbol{W}   (z) +   {e (- 4 \Re\, z)}   \boldsymbol{W}   (- z) + \boldsymbol{E}  (z), \
	\end{align}
	where 
	$\boldsymbol{W} (z)$ and $ \boldsymbol{E}  (z) $ are real analytic functions satisfying 
	\begin{align}\label{1eq: derivatives of W}
	z^{\shskip\alpha} \widebar z^{ \,\beta} (\partial /\partial z)^\alpha (\partial / \partial \overline z)^{\beta} \boldsymbol{W}  (z) \Lt_{\, \alpha,\shskip \beta, \shskip K} 1 / |z|, \hskip 30 pt |z|  \geqslant (|t| + 1)^2, & \\
	\label{1eq: derivatives of E}
	(\partial /\partial z)^\alpha (\partial / \partial \overline z)^{\beta} \boldsymbol{E}_K   (z) \Lt_{\, \alpha,\shskip \beta, \shskip K} (|t| + 1)^{2K} / |z|^{1 + K}, \hskip 10 pt |z|  \geqslant (|t| + 1)^2. &
\end{align}
	We have the uniform bound
	\begin{align}\label{1eq: uniform bound for J}
	t \boldJ_{it} (z) \Lt  ({|t| + 1})^3 \min \big\{ 1, 1/ {  |z|} \big\}  .
	\end{align}
	%All the implied constants above are independent on $t$. 
\end{lem}

\begin{proof}
	The identity \eqref{1eq: J = W + W} comes from the expression of $\boldJ_{it} (z)$    as in (\ref{0eq: J = H + H}, \ref{0def: H mu m (z)}), with $m=0$, by truncating the asymptotic expansions of $H^{(1)}_{2it}  $  and $H^{(2)}_{2it}  $ as in \eqref{3eq: asymptotic H(1)} and \eqref{3eq: asymptotic H(2)} in Lemma \ref{lem: Hankel H(1, 2)}. %(see in particular \cite[\S 7.13.1]{Olver} and \cite[\S 11.4]{Qi-Bessel}). 
	To be precise, 
	\begin{align*}
	\boldsymbol{W} (z) =    {2 \pi }  \sum_{k = 0}^{K-1} \sum_{k' = 0}^{K-1} \frac { (2it, k) (2it, k')} {(- 8\pi i)^{k+k'} z^{k + 1/2}  \overline z^{k' + 1/2} }.
	\end{align*}
	Thus   \eqref{1eq: derivatives of W} is transparent and \eqref{1eq: derivatives of E} follows from \eqref{3eq: estimates for E}. 
	%Moreover, it follows from the error estimates for $H^{(1)}_{\nu} $ and $H^{(2)}_{\nu}$ \cite[\S 7.13.1]{Olver} that 
	Moreover, applying Lemma \ref{lem: Hankel H(1, 2)} with $K  = 0$, we derive the  estimate $ \boldJ_{it} (z) \Lt 1/|z| $ when $|z| \geqslant (|t|+1)^2$, which is actually stronger than \eqref{1eq: uniform bound for J}.  Furthermore,  estimating the integral in \eqref{1eq: integral representation, 1}, with $m=0$, by
	\begin{align}\label{4eq: J 0}
	J_0 (x) \Lt \min \left\{ 1,    1 / {\sqrt x}  \right\} , \hskip 10 pt x > 0,
	\end{align}
	we get
	\begin{align*}
	\boldJ_{it} \lp x  e^{i \phi} \rp  \Lt \int_{ 1/2}^2 \frac {d y}   y + \frac 1 {\sqrt x }\int_0^{1/2} +  \int_2^{\infty}  \frac {  d y } {y \sqrt {|y - 1/y| } }   \Lt 1,
	\end{align*}
	for all $x  > 1$. 
	This is also stronger than \eqref{1eq: uniform bound for J} when $1 < |z| < (|t|+1)^2$. Finally, suppose $|z| \leqslant 1$. We have
	\begin{align}\label{4eq: bound J nu}
	\left| J_{\nu} (z) \right| \Lt_{\, \Re \, \nu}  \frac {   \left|   z  ^{\nu} \right| \exp (|\Im\, z|) } {\left|\Gamma \big(\nu + \tfrac 1 2 \big) \right|},  \hskip 10 pt   \Re \, \nu  > - \tfrac 1 2,
	\end{align}
as a consequence of  Poisson's integral representation (see \cite[3.3 (6)]{Watson}),
\begin{align*}
J_{\nu} (z ) =  \frac {\big(\frac 1 2 z\big)^\nu } {\Gamma \big(\nu + \frac 1 2\big) \Gamma \big(\frac 1 2\big)} \int_0^{\pi} e^{i \shskip z \cos \theta} \sin^{2\nu} \theta \hskip 1pt d \theta.
\end{align*} 
The reader is referred to \cite[3.31 (1)]{Watson} for $\nu $ real, where the bound is slightly different and cleaner, but one should note that the integral of $\sin^{2\nu} \theta$  therein needs to be replaced by that of $\big| \sin^{2\nu} \theta \big| = (\sin \theta)^{2 \shskip \Re \, \nu}$ here for complex values of $\nu$. 
%	\begin{align*}
%	J_{\nu} (z) = \frac {\big(\tfrac 1 2 z\big)^{\nu}} {\Gamma \big(\nu + \tfrac 1 2 \big) \Gamma \big(   \tfrac 1 2 \big)} \int_0^\pi e^{i \hskip 0.5 pt z \cos \theta} \sin^{2 \nu} \theta d \theta, \hskip 10 pt \Re \, \nu  > - \tfrac 1 2.
%	\end{align*}
	Applying \eqref{4eq: bound J nu} with $\nu = \pm 2   i t$ to the two products of Bessel functions in the definition of $\boldJ_{it} (z)$ as in \eqref{0eq: defn of Bessel}, along with $ \left|\Gamma \lp \tfrac 1 2 + 2   i t \rp \right|^2 = \pi / \cosh (2 \pi t)$ (due to Euler's  reflection formula), we infer that 
	\begin{align*}
	t \boldJ_{it} (z) \Lt |t| |\coth (2 \pi t)| \Lt |t| + 1,
	\end{align*}
	if $|z| \leqslant 1$.
\end{proof}

For later applications, we would like to further generalize the second part of Lemma \ref{lem: properties of J} to the derivatives of $\boldJ_{it} (z)$. For this we need the following lemma.

\begin{lem}\label{lem: derivatives of J, 0}
Let $t$ be real. Let $\boldJ_{\mu, \shskip  m} (z) $ be the Bessel function defined by  \eqref{0def: J mu m (z)} and \eqref{0eq: defn of Bessel, general}. Then there are  polynomials $ P^{\shskip\alpha}_0 (Y, Z) $ and $P^{\shskip\alpha}_1 (Y, Z) $  of degree   $ \left\lfloor \alpha   / 2 \right\rfloor  $ and  $\left\lfloor (\alpha - 1) / 2 \right\rfloor $ respectively such that 
\begin{equation}\label{4eq: derivatives of J}
    \begin{split}
    & \hskip 7 pt 2 z^{\alpha} \widebar z^{\,\beta}  (\partial /\partial z)^\alpha (\partial / \partial \widebar z  )^{\beta}  \boldsymbol{J}_{it} (z) =   P^{\shskip\alpha}_0 \widebar P^{\shskip\beta }_0 \boldsymbol{J}_{it} (z) \\
   & + z \widebar z P^{\shskip\alpha}_1 \widebar P^{\shskip \beta}_1 \lp \boldJ_{it + 1/2} (z) + \boldJ_{it - 1/2} (z) + \boldJ_{it,\shskip 2} (z) + \boldJ_{it, -2} (z) \rp \\ 
   & + i z P^{\shskip\alpha}_1 \widebar P^{\shskip\beta }_0 \lp  \boldJ_{it + 1/4,   1}  (z) + \boldJ_{it - 1/4,  - 1}  (z) \rp + i \widebar z \widebar P^{\shskip\beta}_1 P^{\shskip\alpha}_0 \lp \boldJ_{it + 1/4,  -  1}  (z) + \boldJ_{it - 1/4,    1}  (z) \rp ,
    \end{split}
\end{equation}
in which $ P_0^{\shskip \alpha} $... are the shorthand notation  for $ P_0^{\shskip \alpha} \big( \hskip -2 pt - 4 t^2 , z^2 \big)$....
\end{lem}

\begin{proof}
	First, by the Bessel differential equation for $ J_{\nu} (z)$ (\cite[3.1 (1)]{Watson}), 
	\begin{align*}
	z^2 J_{\nu}'' (z) + z J_{\nu}' (z) + \big(z^2 - \nu^2\big) J_{\nu} (z) = 0,
	\end{align*}it is straightforward to prove that we may write
	\begin{align*}
	z^{\shskip\alpha} (d/dz)^{\alpha} J_{\nu} (z) = z P^{\shskip\alpha}_1 \big(\nu^2, z^2  \big) J_{\nu}' (z) + P^{\shskip \alpha}_0 \big(\nu^2, z^2   \big) J_{\nu} (z),
	\end{align*}
	for certain {\rm(}integer-coefficient{\rm)} polynomials $ P^{\shskip\alpha}_0 (Y, Z)$ and $P^{\shskip\alpha}_1 (Y, Z)$  of degree $ \left\lfloor \alpha   / 2 \right\rfloor  $ and  $\left\lfloor (\alpha - 1) / 2 \right\rfloor $ respectively. %with the constant term of $ P^{\shskip\alpha}_0 (Y, Z) $  vanishing for all $\alpha = 1, 2, ...$ (of course $P^{0 }_0 (Y, Z) = 1$). 
	Second, we have the recurrence  relation \cite[3.2 (2)]{Watson}
	\begin{align*}%\label{4eq: recurrence rel}
	2 J_{\nu}' (z) = J_{\nu -1} (z) - J_{\nu+1} (z).
	\end{align*}
	From these and the definitions in  \eqref{0def: J mu m (z)} and \eqref{0eq: defn of Bessel, general} follows  \eqref{4eq: derivatives of J} by direct calculations. 
\end{proof}

\begin{lem}\label{lem: derivatives of J}
	We have 
	\begin{align*}%\label{4eq: bounds for the derivatives of J}
t  (4 t^2+1 ) 	z^{\alpha} \widebar z^{\,\beta} \frac{\partial^{\alpha + \beta}  \boldsymbol{J}_{it} (z) }{\partial z^{\alpha} \partial \widebar z^{\shskip\beta}}     \Lt_{\, \alpha, \shskip \beta } (|t|+1)^5 \min \big\{ 1, 1/ {  |z|} \big\} (|z|+ |t| +1)^{\alpha + \beta } .
	\end{align*}
\end{lem}

\begin{proof}
	The estimates may be proven by the arguments in  the proof  of Lemma \ref{lem: properties of J}, applied to the Bessel functions in \eqref{4eq: derivatives of J} in Lemma \ref{lem: derivatives of J, 0}. The conscientious reader however might have already noticed that there are two technical issues: \eqref{4eq: bound J nu} is only valid for $\Re\, \nu > - \tfrac 1 2$ and \eqref{1eq: integral representation, 1}   for $|\Re\, \mu | < \tfrac 1 8$. These issues however may be addressed as follows.
	
	%In addition to  \eqref{4eq: bound J nu}, b
	By using the recurrence formula (\cite[3.2 (1)]{Watson})
	\begin{align*}
	J_{\nu-1} (z) = \frac {2 \nu} z J_{\nu} (z) - J_{\nu+1} (z),
	\end{align*} we deduce from  \eqref{4eq: bound J nu} that (see \cite[3.31 (2)]{Watson} for $\nu$ real)
	\begin{align}\label{4eq: bound J nu, 2}
J_{\nu} (z) \Lt_{\, \Re \, \nu} \frac {    \left|   z  ^{\nu } \right| \exp (|\Im\, z|)  } {\left|\Gamma \big(\nu + \tfrac 3 2 \big) \right|} \lp   {\left| \nu+1 \right|}  + \frac {   \left|   z  ^{2} \right|  } {\left|   \nu + \tfrac 3 2    \right|} \rp , \hskip 10 pt - \tfrac 3 2 < \Re \, \nu \leqslant - \tfrac 1 2.
	\end{align}
Thus we are now able to estimate all the Bessel functions  in \eqref{4eq: derivatives of J}, like $\boldJ_{it - 1/2} (z)$..., by applying either \eqref{4eq: bound J nu} or \eqref{4eq: bound J nu, 2}. 
%Therefore, when $|z| \leqslant 1$, we deduce from  \eqref{4eq: derivatives of J} and (\ref{4eq: bound J nu}, \ref{4eq: bound J nu, 2}) that 
%	\begin{align}\label{4eq: bound for derivatives, 1}
%	t \big(4 t^2+1\big) z^{\alpha} \widebar z^{\beta}  (\partial /\partial z)^\alpha (\partial / \partial \widebar z  )^{\beta}  \boldsymbol{J}_{it} (z) \Lt_{\, \alpha, \shskip \beta } (|t|+1)^3 (|z| + |t| + 1)^{\alpha + \beta} .
%	\end{align}
	
	Moreover, we have the following lemma.
	
	\begin{lem}\label{lem: bound for J mu m}
		Suppose that $ |\Re \, \mu | <   (2 k + 1) /8 $. For any $|z| > 1$, we have 
		\begin{align}\label{4eq: |z| > 1}
		\boldJ_{\mu, \shskip  m } (z) \Lt_{\, k, \shskip m }     1 +   {(|\mu|+1)^k} /  |z|^{k+1/2}   .
		\end{align}
	\end{lem}
\begin{proof}[Proof of Lemma \ref{lem: bound for J mu m}]
	First, for the integral in \eqref{1eq: integral representation, 1}, we introduce a smooth partition of unity $(1 - \tv (y)) + \tv (y) \equiv 1$ for the domain of integration $(0, \infty) = \left[1/3, 3 \right] \cup \big\{ (0, 1/2] \cup [2, \infty) \big\}$. 
	The integral over $[1/3, 3]$ is absolutely convergent (for all $\mu$) and bounded. Next, we consider the integral over $[2, \infty)$, while that over $ (0, 1/2]$ may be treated in the same way after changing the variable $y$ to $1/y$. We apply to the $J_{m} \big( 4 \pi  x  \left| y e^{i\phi} + y\- e^{-i\phi} \right|  \big)$ in \eqref{1eq: integral representation, 1} the following asymptotic expansion  (see \cite[7.21 (1)]{Watson}),
	\begin{align*}
	J_m (x) = \lp \frac 2 {\pi x} \rp^{1/2} \Bigg(   \cos \lp x - \tfrac 1 2 m \pi -\tfrac 1 4 \pi  \rp & \sum_{ l = 0 }^{\left\lfloor (k-1) / 2 \right\rfloor }  \frac {(-)^l \cdot (m, 2l) } {(2 x)^{2l} }  \\
	 -   \sin \lp x - \tfrac 1 2 m \pi -\tfrac 1 4 \pi  \rp & \sum_{ l = 0 }^{\left\lfloor (k-2) / 2 \right\rfloor  }  \frac {(-)^l \cdot (m, 2l+1) } {(2 x)^{2l+1} } + O \lp \frac 1 {x^{k}} \rp \Bigg),
	\end{align*}
	in which $(m, l)  $ is the coefficient defined as in Lemma \ref{lem: Hankel H(1, 2)}. The error term contributes an absolutely convergent integral with the bound $1 / x^{k+1/2}$. 
	 Thus we have to consider integrals of the form
	\begin{align*}
	\frac {1} {x^{\shskip l + 1/2}} \int_2^{\infty} \tw^{\pm}_{\mu, \shskip  m, \shskip l} (y, \phi) e \lp \pm 2 x y \rp d y,
	\end{align*}
	with 
	\begin{align*}
	  \tw^{\pm}_{\mu, \shskip  m, \shskip l} (y, \phi) = \frac {y^{4 \mu - 1} \lp y e^{-i\phi} + y\- e^{i\phi} \rp^{ m} e \lp \pm 2 \lp \left| y  e^{i\phi} + y\- e^{-i\phi} \right|- y \rp \rp } { \left| y e^{i\phi} + y\- e^{-i\phi} \right|^{m + l + 1/2}  }. 
	\end{align*}
	Note that $   \left| y  e^{i\phi} + y\- e^{-i\phi} \right|- y = O \lp  1 \rp $ when $y \geqslant 2$. So by repeating   partial integration  $k - l$ many times we get absolute convergence when $ |\Re \, \mu | <   (2 k + 1) / 8 $ and the bound $ (|\mu | + 1)^{k - l} / x^{k+ 1/2} $. Now \eqref{4eq: |z| > 1} follows after collecting the bounds that we have established.
\end{proof}
	
	%As for the integral representation \eqref{1eq: integral representation, 1} of $\boldJ_{\mu, \shskip m} \lp x  e^{i \phi} \rp $, roughly speaking, the range of absolute convergence $|\Re \, \mu | < \tfrac 1 8$ may be extended onto $|\Re \, \mu | < \tfrac 1 8 (2k + 1)$ by performing $k$ many times of partial integration. Assume now $x > 1$. 

%	Let $|z| > 1$.  Lemma \ref{lem: bound for J mu m} applies to the Bessel functions occurring in \eqref{4eq: derivatives of J}, with $k = 0, 1, $ or $2$, giving
%	\begin{align}\label{4eq: bound for derivatives, 2}
%z^{\alpha} \widebar z^{\beta}  (\partial /\partial z)^\alpha (\partial / \partial \widebar z  )^{\beta}  \boldsymbol{J}_{it} (z) \Lt_{\, \alpha, \shskip \beta }    (|z| + |t| + 1)^{\alpha + \beta}.
%	\end{align}
	
%	Furthermore, when $|z| \geqslant (|\mu|+1)^2$, we have uniformly
%	\begin{align}
%	\boldJ_{\mu, \shskip m } (z) \Lt_{\, m }  1 / |z|.
%	\end{align} 
%	So
%	\begin{align}\label{4eq: bound for derivatives, 3}
%	z^{\alpha} \widebar z^{\beta}  (\partial /\partial z)^\alpha (\partial / \partial \widebar z  )^{\beta}  \boldsymbol{J}_{it} (z) \Lt_{\, \alpha, \shskip \beta }  |z|^{-1} (|z| + |t| + 1)^{\alpha + \beta},
%	\end{align}
%	if $|z| \geqslant (|t|+1)^2$. 
	 
	Using (\ref{4eq: bound J nu}, \ref{4eq: bound J nu, 2}), \eqref{4eq: |z| > 1} along with (\ref{0eq: J = H + H}, \ref{0def: H mu m (z)}), we may deduce the estimates in this lemma from \eqref{4eq: derivatives of J} in Lemma \ref{lem: derivatives of J, 0}.
\end{proof}

\subsection{Bounds for the Bessel integral}\label{sec: the Bessel integral}

Let $A'$ be a positive integer as in \S \ref{sec: afeq}. Subsequently, we shall fix the choice of spectral weight 
\begin{equation}\label{4eq: h}
h (t) =    k (t) G (\varv , t),   
\end{equation}
with
\begin{align}\label{4eq: h, 1}
k (t)    = e^{- t^2 / T^2} p   \big(\tfrac 1 2, t \big)^{2} p (t) / g (t) ,  
\end{align} 
$p  \big(\tfrac 1 2, t \big)$   defined as in  \eqref{1eq: def p(s, t)}, $ G  (\varv , t)$ as in \eqref{1eq: def G}, and
\begin{align}\label{4eq: g and k}
p (t)   = \prod_{k=0}^{2 A'-1} \lp 4 t^2 + (k+1)^2  \rp \hskip - 1 pt, \hskip 10 pt g  (t) = \lp t^2 + (A'+1)^2 \rp^{ 8 A'} \hskip -2 pt .
\end{align}
Note that by the Kim-Sarnak bound for $i t_j $ and $\mu$ in (\ref{1eq: Kim-Sarnak}, \ref{1eq: Kim-Sarnak, GL3}) we have 
\begin{align}\label{4eq: h (t) > 0}
k (t)  > 0,
\end{align} 
and 
\begin{align}\label{1eq: p (1/2, t)}
k (t) \asymp_{A' }  e^{- t^2 / T^2}, \hskip 10 pt   t \ra \infty,
\end{align} 
if $t$ is the spectral parameter $t_j$ of a Maass form $ f_j$ or the Eisenstein series.
In view of \eqref{1eq: bound for p G} in Lemma \ref{lem: afq} (2), we have
\begin{align}\label{4eq: bound for h (t)}
h (t) \Lt_{\, A', \, \mu} (|t| + 1)^{6 \varepsilon} e^{- (\Re \shskip t)^2 / T^2}, \hskip 10 pt |\Im \, t | < A' + \tfrac {9} {32}. 
\end{align}

\begin{lem}\label{lem: Bessel integral} Let $H (z) $ be defined as in  \eqref{1eq: H's} with the choice of $h (t)  $ given by  \eqref{4eq: h}-\eqref{4eq: g and k} and also {\rm(\ref{1eq: gamma (s)}, \ref{1eq: gamma (s, t)},  \ref{1eq: def p(s, t)}, \ref{1eq: def G})}.   We have
	\begin{align}\label{5eq: H(z) for z small}
	H (z) \Lt  T^{5+\varepsilon}  \min \mbox{\larger[1]\text{${\big\{}$}} |z|^{4 A'} , 1/ |z| \mbox{\larger[1]\text{${\big\}}$}}  .
	\end{align}
	
\end{lem}
 
\begin{proof}
First let $|z| \leqslant 1$. By the definitions in \eqref{0eq: defn of Bessel} and \eqref{1eq: H's}, we have
\begin{align*}
H (z) =   {4 \pi^2} i \int_{-\infty}^{\infty} h (t)  \frac {J_{    2it    } \lp 4 \pi  z \rp J_{  2it    } \lp 4 \pi  {\overline z} \rp} {\sinh \lp 2\pi  t   \rp }  t^2 d t.
\end{align*}
%Note that the Plancherel measure $t^2 d t$ vanishes at $t =0$. 
Since we have included the polynomial $p (t)$ in the definition of $h (t)$ as in \eqref{4eq: h}-\eqref{4eq: g and k}, $h (t) t^2 / \sinh (2 \pi  t)$ is holomorphic when $|\Im \, t| < A'  + \frac {9} {32}$. We now shift  the line  of integration to $\Re (i t) =  A'  $ and estimate the resulting integral by  \eqref{4eq: bound for h (t)} and \eqref{4eq: bound J nu} along with Stirling's formula.
Then follows   \eqref{5eq: H(z) for z small} for $|z| \leqslant 1$.  The case $|z| > 1$ follows from \eqref{1eq: uniform bound for J} in Lemma \ref{lem: properties of J}. 
\end{proof}

%, so $t = - \frac 1 2 i $ is the only pole crossed. 
\delete{ We get
\begin{align*}
H(z) =  %& - \pi^2 h \big( \hskip - 2 pt - \hskip - 1 pt \tfrac 1 2 i \big) J_{1} (4 \pi z) J_{1} (4 \pi \overline z) \\
   4 \pi^2  \int_{-\infty}^{\infty} h \lp   t - \hskip - 1 pt \tfrac 1 2 i (A' + \delta) \rp   \frac {J_{  2it +A' + \delta } \lp 4 \pi  z \rp J_{  2it + A' + \delta   } \lp 4 \pi  {\overline z} \rp} {\sin \lp 2\pi it + \delta   \pi \rp }    \lp t - \tfrac 1 2 i (A' + \delta) \rp^2 d t.
\end{align*} }

\delete{Hence,  Stirling's formula yields the estimates
\begin{align*}%\label{1eq: J for |z|<1}
%J_{1} (4 \pi z) J_{1} (4 \pi \overline z) \Lt |z|^2, \hskip 10 pt 
\frac {J_{  2it + A' + \delta } \lp 4 \pi  z \rp J_{  2it + A' + \delta   } \lp 4 \pi  {\overline z} \rp} {\sin \lp 2\pi it + \delta   \pi \rp }   \Lt    \lp \frac { |z|  } { |t| + 1   } \rp^{2 A' + 2 \delta}    , \hskip 10 pt |z| \leqslant 1.
\end{align*} 
Moreover, %$h \big( \hskip - 2 pt - \hskip - 1 pt \tfrac 1 2 i \big)$ and 
$h \lp   t - \hskip - 1 pt \tfrac 1 2 i (A' + \delta) \rp $ may be estimated by \eqref{1eq: bound for p G} in Lemma \ref{lem: afq} (2).
Consequently, %on choosing $\delta = 3 \hskip 1 pt \Re \, \varv$, these estimates yield 
}

\section{Properties of the rank-three Bessel kernel}

In order to apply the \Voronoi formula for  $  \SL_3 (\frO)$, one must understand the asymptotic behaviour of the Hankel transform or its Bessel kernel for $\GL_3 (\BC)$. This is done in the author's work \cite{Qi-Bessel}, combining the high-dimensional stationary phase method, due to H\"ormander, for certain formal integrals and the asymptotic theory for Bessel differential equations. Some discussions may be found in the remarks after \cite[Proposition {\rm 5.3}]{Qi-Wilton}. 

Let the Bessel kernel  $ \boldJ_{(\mu_1,\, \mu_2,\, \mu_3)} (z)  $ be defined %as the integral kernel of the Hankel transform 
by \eqref{1eq: Hankel transform identity, C} and \eqref{1eq: Hankel, integral} in \S \ref{sec: Hankel transform}.  The following asymptotic expansion formula for $\boldJ_{(\mu_1,\, \mu_2,\, \mu_3)} (z)$ will be the foundation of our analysis in this article. See \cite[Theorem 16.6]{Qi-Bessel}.
\begin{lem}\label{prop: asymptotic J}
	Let the Bessel kernel  $\boldJ_{(\mu_1,\, \mu_2,\, \mu_3)} (z)  $ be given as in {\rm\S  \ref{sec: Hankel transform}}. Let $K$ be a nonnegative integer. For $|z| \Gt 1$, we have the asymptotic expansion 
	\begin{equation*}%\label{3eq: asymptotic, Bessel}
	\begin{split}
	\boldJ_{(\mu_1,\, \mu_2,\, \mu_3)} (z) = \sum_{ \xi^3 = 1} \frac { e \big( 3 \big(\xi z^{1/3} + \overline \xi \overline z^{1/3}\big) \big) } {|z|^{2/3} }  \lp \underset{ k + l \leqslant K -1 }{\sum \sum} %\sstyle k, l = 0, 1, ..., K-1 \atop \sstyle 
	B_{k } B_{l}  \xi^{-k+l}  z^{-k/3} \overline z^{\, - l/3} \rp  & \\
	+ O_{K,\, \mu_1,\, \mu_2,\, \mu_3} \mbox{\larger[1]\text{${\big(}$}} |z|^{- (K + 2) / 3} \mbox{\larger[1]\text{${\big)}$}} &,
	\end{split}
	\end{equation*}
	where $B_{k } = B_k (\mu_1, \mu_2, \mu_3)$ is a symmetric polynomial in $\mu_1, \mu_2, \mu_3$ of degree $2 k $. %, with $B_0 = 1$. %Note that the right hand side is independent of the choice of the argument of $z$ and the branch of the cubic root. 
	Moreover, similar asymptotic expansions are valid for all the partial derivatives of $\boldJ_{(\mu_1,\, \mu_2,\, \mu_3)} (z) $. 
	%such that
	%\begin{equation*}
	%|z|^{\alpha + \beta} \partial_{z}^{\alpha} \partial_{\overline z}^{\beta} \boldsymbol W_f (z; \xi) \Lt_{\alpha, \, \beta,\, f} 1, \hskip 10 pt |z| \geqslant 1, 
	%\end{equation*}
	%or, in the polar coordinates,
	%\begin{equation*}
	%x^{\alpha} \partial_x^{\alpha} \partial_{\phi}^{\beta} \boldsymbol W_f \big( x e^{i \phi} ; \xi \big) \Lt_{\alpha, \, \beta,\, f} 1, \hskip 10 pt x \geqslant 1.
	%	\end{equation*}
\end{lem}

Moreover, we have the following bounds for $\boldJ_{(\mu_1,\, \mu_2,\, \mu_3)} (z)$ when $|z| \Lt 1$. See \cite[Lemma 5.1]{Qi-Wilton} and its proof.

\begin{lem}\label{lem: bound for z<1, J}
	Suppose that $\Re\, \mu_1, \Re \, \mu_2, \Re \, \mu_3 < \sigma$. For $|z| \Lt 1$, we have
	\begin{align*}
	z^{\shskip \alpha} \widebar z^{\, \beta} (\partial /\partial z)^\alpha (\partial / \partial \overline z)^{\beta} \boldJ_{(\mu_1,\, \mu_2,\, \mu_3)} (z)  \Lt_{ \, \alpha, \shskip  \beta, \, \mu_1,\, \mu_2,\, \mu_3} 1/ |z|^{ 2 \sigma}.
	\end{align*}
\end{lem}

Both \cite[Theorem  16.6]{Qi-Bessel} and \cite[Lemma 5.1]{Qi-Wilton} concern only $\boldJ_{(\mu_1,\, \mu_2,\, \mu_3)} (z)   $ but not its derivatives. However, the former may be generalized to the derivatives of $\boldJ_{(\mu_1,\, \mu_2,\, \mu_3)} (z)   $ by using the full \cite[Theorem 11.24]{Qi-Bessel}, while the generalization of the latter is very straightforward.

\section{Analysis of the Hankel transform}

This section is devoted to the analysis of the $\GL_3$-Hankel transform of the $\GL_2$-Bessel integral. %$H (z)$.  %in the $\GL_3$ \Voronoi summation formula.

\subsection{Oscillatory double integrals}  \label{sec: oscillatory integrals, 2}

Consider
\begin{align}\label{3eq: defn I, 2}
I  (\lambdaup, \theta ) = \int_{0}^{2 \pi}   \int_0^\infty   e \lp   \lambdaup f (x, \phi; \theta ) \rp   \tw  \lp  x ,   \phi; \lambdaup, \theta  \rp    d x d \phi, 
\end{align}
with  
\begin{align}\label{4eq: phase f}
f (x, \phi; \theta ) =   3 x^2  \cos \lp 2 \phi +      \theta   \rp - 2  x^3  \cos   3  \phi   -    \cos 3 \theta  .
\end{align}
Suppose that $\tw  \lp  x ,   \phi; \lambdaup , \theta   \rp$ is supported in $\left\{ (x, \phi) : x \in \left[ \rho  ,   2^{1/6}  \rho   \right] \right\}$ and its derivatives satisfy 
\begin{align}\label{3eq: bound w}
x^{\alpha} \lambdaup^{\gamma} \partial_x^{\alpha} \partial_\phi^{\shskip \beta} \partial_{\lambdaup}^{\gamma} \partial_{\theta}^{\delta} \tw  \lp  x ,   \phi; \lambdaup, \theta \rp \Lt_{\, \alpha, \shskip  \beta, \shskip \gamma, \shskip \delta  }        S X  ^{  \alpha + \beta + \gamma + \delta}, 
\end{align}
with $ S > 0$, $X \geqslant 1$. We have
\begin{align*}
f'  (x, \phi; \theta ) & =   \lp 6 x \lp  \cos \lp 2 \phi +      \theta   \rp -   x \cos   3 \phi \rp,  - 6 x^2 \lp  \sin \lp 2 \phi +      \theta   \rp -   x \sin   3 \phi \rp \rp, 
\end{align*}
and hence $f (x, \phi; \theta )$  has a unique stationary point at $ (1 , \theta)$.

The modified integral $ e(2 \lambdaup \cos 3 \theta ) I  (\lambdaup, \theta )$ may be considered as a  $\GL_3 \times \GL_2$-type Fourier-Hankel convolution and results in an oscillatory function of $\GL_{1}$-type. Similar integrals of type $\GL_3 \times \GL_1$, $ \GL_2 \times \GL_1 $ and $\GL_1 \times \GL_2$ have been studied in the author's previous work.  See   \cite[\S 5.3]{Qi-Wilton} and \cite[\S 5.1.3]{Qi-II-G}.

%Subsequently, we shall only consider the case $2 = 2$ and drop $2$ from the notation.

First, we would like to apply H\"ormander's elaborate partial integration  (see the proof of \cite[Theorem 7.7.1]{Hormander}) in the polar coordinates. 
For this, we define
\begin{equation*}%\label{3eq: expression of g}
\begin{split}
g (x, \phi; \theta ) & =   \lp \partial_x f (x, \phi; \theta ) \rp^2 / x^3 +    \lp \partial_\phi f (x, \phi; \theta ) \rp^2 / x^5 \\ %= x^2 + x^{-2} - 2 \cos \lp \phi - \theta \rp. %
& = 36 \big( \big(  \sqrt  x   -  1/ \sqrt x   \big)^2 + 2 ( 1 - \cos   (\phi - \theta)   ) \big).
\end{split}
\end{equation*}
An important observation is that  
\begin{equation}\label{3eq: bound for g}
g \lp x , \phi; \theta  \rp > \left\{ 
\begin{split}
& 4  x, \hskip 14 pt \text{ if }  x  > 3/2, \\
& 4 / x , \hskip 10 pt \text{ if }  x  <   2 / 3.
\end{split}\right.
\end{equation}
We then introduce the differential operator
\begin{align*}
\mathrm{D}  =       \frac {   \partial_x f(x , \phi; \theta )} { x^3 g (x , \phi; \theta ) } \frac {\partial} {\partial x}      +    \frac {  \partial_{\phi} f(x , \phi; \theta )} { x^5 g (x , \phi; \theta ) } \frac {\partial} {\partial \phi}   
\end{align*}
so that $\mathrm{D}  \lp e \lp \lambdaup f \lp x , \phi; \theta  \rp \rp \rp =  2 \pi i \shskip  \lambdaup \cdot  e \lp \lambdaup f \lp x , \phi; \theta  \rp \rp$. Consequently, 
\begin{align*}
\mathrm{D}^* \hskip -2 pt = -   \frac {\partial} {\partial x}   \frac {   \partial_x f \lp x , \phi; \theta  \rp} { x^3 g (x , \phi; \theta ) }   - \frac {\partial} {\partial \phi}   \frac {  \partial_{\phi} f (x , \phi; \theta )} { x^5 g (x , \phi; \theta ) }    
\end{align*}
is the adjoint and
\begin{align*}
I (\lambdaup, \theta )  = \frac 1 {( 2 \pi i \shskip \lambdaup)^A } \int_0^{2 \pi} \int_{ 0}^{\infty} e \lp \lambdaup f \lp x , \phi; \theta  \rp \rp  \mathrm{D}^{* \shskip  A}     \tw  \lp  x ,   \phi; \lambdaup    \rp  d x  d \phi
\end{align*}
for any integer $A \geqslant 0$.
By a straightforward inductive argument, it may be shown that $ \mathrm{D} ^{* \shskip  A}    \tw   $ is a linear combination of all the terms occurring in  the  expansions of
\begin{align*}
\partial_{x}^{\alpha } \partial_{\phi}^{\beta }  \mbox{\larger[1]\text{${\big\{}$}} \hskip -2 pt \lp \partial_{x} f /  x^{3}     \rp^{\alpha}   \lp  \partial_{\phi} f / x^{5}  \rp^{\beta} g^{A}   \shskip  \tw    \mbox{\larger[1]\text{${\big\}}$}} / g^{ 2 A}, \hskip 10 pt \alpha + \beta = A.
\end{align*}
Moreover, we have
\begin{align*}
%& x^{-1} \partial_{x} f \lp x , \phi; \theta  \rp, \ \ x^{-2} \partial_{\phi} f \lp x , \phi; \theta  \rp \Lt |x - 1| + \left|\sin \tfrac 1 2  (\phi - \theta)  \right|, \\
& x^{\alpha + 2} \partial_{x}^{\alpha } \partial_{\phi}^{\shskip
	\beta } \lp   \partial_{x} f \lp x , \phi; \theta  \rp /x^3  \rp, \ \ 
x^{\alpha + 3} \partial_{x}^{\alpha } \partial_{\phi}^{\beta } \lp  \partial_{\phi} f \lp x , \phi; \theta  \rp /x^5 \rp \Lt    x  + 1    , \\
&x^2 \partial_{x}  g \lp x , \phi; \theta  \rp \Lt  (x+1)^2    , \hskip 5 pt x^{ \alpha + 1} \partial_{x}^{\alpha }   g \lp x , \phi; \theta  \rp \Lt    1  ,    \hskip 5 pt \partial_{\phi}^{\shskip \beta } g \lp x , \phi; \theta  \rp \Lt 1,  \hskip 5 pt \text{($ \alpha \geq 2, \ \beta \geq 1$)}, \\
& \partial_{x} \partial_{\phi}  g \lp x , \phi; \theta  \rp \equiv 0.
\end{align*}
Let $\acute \alpha$, $\grave \alpha \leqslant \alpha$ and $\acute \beta$, $\grave \beta \leqslant \beta$. From the  estimates above, it is straightforward to prove that  
\begin{align*} 
x^{\acute \alpha} \partial_{x}^{\acute \alpha } \partial_{\phi}^{\shskip  \acute \beta } \mbox{\larger[2]\text{${\big\{}$}} \hskip -2 pt \lp   \partial_{x} f \lp x , \phi; \theta  \rp /x^3 \rp^{\alpha}    \lp  \partial_{\phi} f \lp x , \phi; \theta  \rp /x^5 \rp^{\beta} \hskip -1 pt \mbox{\larger[2]\text{${\big\}}$}} \Lt \frac {(x+1)^{\alpha + \beta   } }  {x^{ 2 \alpha + 3 \beta}} ,
\end{align*}   
\begin{align*}
\frac { x^{\grave \alpha} \partial_{x}^{\grave \alpha } \partial_{\phi}^{\shskip \grave \beta } \shskip  g \lp x , \phi; \theta  \rp^A } {g \lp x , \phi; \theta  \rp^{2A} } \Lt \sum_{ \alpha_1 + 2 \alpha_2 \leqslant \grave \alpha} \ \sum_{ \beta_1   \leqslant \grave \beta } \frac { (x+1)^{2 \alpha_1}   } 
{ x^{\alpha_1 + \alpha_2  } |g \lp x , \phi; \theta  \rp|^{A + \alpha_1 +   \beta_1 +  \alpha_2   }  },
\end{align*}
Combining these, %by the product rule for differentiation, 
we deduce the following  estimate,
\begin{align} \label{3eq: ibp}
I (\lambdaup, \theta ) \Lt  \sum_{ \alpha_1  +   \beta_1 + 2 \alpha_2  + \alpha    + \beta  \leqslant   A } \frac 1 { \lambdaup^A} \hskip -1 pt \int_0^{2 \pi}  \hskip -3 pt  \int_{ 0}^{\infty} \hskip -1 pt  \frac { \lp  x  +  1 \rp^{A  + 2 \alpha_1}  \big| x^{\alpha } \partial_{x}^{\alpha } \partial_{\phi}^{\shskip \beta } \tw   \big|}    
{ x^{3 A + \alpha_1 + \alpha_2}  | g  |^{ A + \alpha_1  + \beta_1 + \alpha_2    }}   d x  d \phi.
\end{align}

\begin{lem}\label{lem: off the range}
	For either $ \rho > 2  $ or $\rho < 1 / 2 $, we have
	\begin{align*}
	\lambdaup^{\gamma} \frac {\partial^{\gamma + \delta}} {\partial \lambdaup^{\gamma} \partial \theta^{\shskip\delta } } 	I (\lambdaup, \theta) \Lt_{\, \gamma, \shskip \delta, \shskip A} S \rho \lp \lambdaup \lp \rho^3 + 1 \rp + X \rp^{\gamma + \delta} \lp \frac {X} {\lambdaup \shskip \rho^2 (\rho+1) } \rp^A . 
	\end{align*}
\end{lem}

\begin{proof}%[Proof of Lemma {\rm\ref{lem: off the range}}]
	When $ \rho > 2  $, it follows from \eqref{3eq: bound w},   \eqref{3eq: bound for g} and \eqref{3eq: ibp} that $I (\lambdaup, \theta )  $ is bounded by
	\begin{align*} 
	\hskip 12 pt  \sum_{ \alpha_1 +  \beta_1 + 2 \alpha_2   + \alpha   + \beta  \leqslant A }  &   \frac 1 {\lambdaup^A } \int_0^{2 \pi}   \int_{ \rho  }^{2^{1/6} \rho } 
	\frac {   x^{A + 2 \alpha_1 } \cdot S X^{\alpha + \beta } } { x^{4 A + 2 \alpha_1 + \beta_1 + 2\alpha_2  }   }   dx d \phi  
	%& \Lt \sum_{ \alpha_1 +  \beta_1 + 2 \alpha_2    + \alpha  + \beta  \leqslant A }     \frac {S \rho X^{\alpha + \beta }} { \lambdaup^A \rho^{  3 A  + \beta_1 + 2\alpha_2  }}  
	\Lt    { S \rho} \lp \frac {X} {\lambdaup \shskip \rho^3 } \rp^A.
	\end{align*} 
	Similarly, when $\rho < 1 / 2 $,    $I (\lambdaup, \theta )  $ is   bounded by
	\begin{align*} 
	\hskip 12 pt \sum_{ \alpha_1  +  \beta_1 + 2 \alpha_2   + \alpha   + \beta   \leqslant A }  &   \frac 1 {\lambdaup^A } \int_0^{2 \pi}   \int_{ \rho  }^{2^{1/6} \rho }
	\frac {     S X^{\alpha  + \beta } } {x^{2 A -  \beta_1  }   }   dx d \phi 
	%&\Lt \sum_{ \alpha_1 +  \beta_1 + 2 \alpha_2     + \alpha  + \beta  \leqslant A }     \frac {S \rho  X^{\alpha + \beta } } {\lambdaup^A  \rho^{2 A - \beta_1   }  }  
	\Lt   { S \rho   } \lp  \frac X {  \lambdaup  \shskip \rho^2  } \rp^A. %\Lt  \frac { K R^2  /\lambdaup^{2/3}} { (  R^3 X)^A }.
	\end{align*} 
	In general, one applies the two estimates   above to the integrals obtained after differentiations.
\end{proof}

Next, we consider the case when $ 1 / 2  \leqslant \rho \leqslant 2 $ and aim to get the stationary phase estimate for the derivatives of $I(\lambdaup, \theta)$. For this, the following lemma is very useful,  which can be viewed as a substitute for   the Morse lemma. It may be proven by straightforward algebraic and trigonometric computations.
\begin{lem}\label{lem: f = p q r} Let $f (x, \phi; \theta )$ be defined as in {\rm\eqref{4eq: phase f}}. 
	We may write
	\begin{align*}
	f    =      p \cdot  (x - 1)^2     + 2 q  \cdot  (x - 1) \sin  \lp (\phi - \theta) /2 \rp    
	+   r \cdot \sin^2      \lp (\phi - \theta) /2 \rp     ,
	\end{align*}
	with
	\begin{align*}
	& p (x, \phi, \theta ) = - (2  x + 1) \cos \lp 2 \phi + \theta \rp , \\
	& q (x, \phi, \theta ) = 2 \lp x^2 + x + 1 \rp \sin   \lp (5 \phi + \theta) / 2 \rp  , \\
	& r (x, \phi, \theta ) = 8 \cos  \lp  (\phi - \theta)  /2 \rp \lp 2 \cos 2 \lp 2  \phi +   \theta \rp  + \cos 2 (\phi + 2 \theta) \rp.
	\end{align*}
\end{lem}

When $1/2  \leqslant x \leqslant 2^{7/6}$, it is clear that
\begin{align}\label{3eq: derivatives of p}
\partial_{x}^{\alpha} \partial_{\phi}^{\beta} \partial_{\theta}^{\gamma} p   (x, \phi, \theta ), \ 
\partial_{x}^{\alpha} \partial_{\phi}^{\beta} \partial_{\theta}^{\gamma} p   (x, \phi, \theta ), \ 
\partial_{x}^{\alpha} \partial_{\phi}^{\beta} \partial_{\theta}^{\gamma} p   (x, \phi, \theta )  \Lt_{ \alpha, \, \beta, \, \gamma  } 1.
\end{align} 

We now   compute  the derivatives of $I  (\lambdaup, \theta)$ explicitly. On changing the variable $\phi$ to $\phi + \theta$ and then differentiating under the integral, with the help of Lemma \ref{lem: f = p q r},
some calculations show that $\lambdaup^{\gamma} \partial_{\lambdaup}^{\gamma} \partial_{\theta}^{\delta} I  (\lambdaup, \theta)$ may be expressed as a linear combination of integrals of the form 
\begin{equation}\label{3eq: expanding derivatives of I}
\begin{split}
\lambdaup^{ (\alpha  +\kappa + \beta + \tau) /2  }    \int_{0}^{2 \pi}   \int_0^\infty  & e \lp   \lambdaup f (x, \phi + \theta  , \theta) \rp  (x - 1)^{\alpha + \kappa}  \sin \lp \phi / 2 \rp  ^{\beta + \tau} \\
&  \tw_{ \kappa \tau}  (x, \phi, \theta) \partial_{\theta}^{\delta'} \mbox{\larger[1]\text{${\big\{}$}}  \tv_{\alpha  \beta}  (x, \phi + \theta, \theta)   \lambdaup^{\gamma'} \partial_{\lambdaup}^{\gamma'} \tw \lp  x ,   \phi + \theta, \lambdaup, \theta  \rp \hskip -2 pt \mbox{\larger[1]\text{${\big\}}$}}     d x d \phi,
\end{split}
\end{equation}
with $\alpha + \beta + 2 \gamma' = 2 \gamma$ and $\kappa + \tau + 2 \delta' \leqslant 2 \delta$ such that $\kappa + \tau$ is even and that $\delta' = \delta$ if $\kappa = \tau = 0$. Moreover,
$\tv_{\alpha \beta}$  are defined by
\begin{align*}%\label{3eq: defn of v}
\lp p \cdot  Y^2     + 2 q  \cdot  YZ
+ r   \cdot Z^2    \rp^{\gamma - \gamma'} = \sum \tv_{\alpha \beta}  \cdot Y^{\alpha} Z^{\beta},
\end{align*}
and $\tw_{\kappa \tau}$ are  defined by $\tw_{00} \equiv 1$ and
\begin{align*}%\label{3eq: defn of w}
\frac { \partial_{\theta}^{\delta - \delta' } e \lp \lambdaup \lp p \cdot  Y^2     + 2 q  \cdot  YZ 
	+ r  \cdot Z^2 \rp  \rp } { e \lp \lambdaup \lp p \cdot  Y^2     + 2 q \cdot  YZ 
	+ r   \cdot Z^2 \rp  \rp } = \sum \lambdaup^{(\kappa + \tau) / 2} \tw_{\kappa \tau} \cdot Y^{\kappa} Z^{\tau};
\end{align*}
in the second identity, the arguments of  $p , q$ and $r$ on the left are $(x, \phi + \theta, \theta)$ while those of $\tw_{\kappa \tau}$ on the right are $ (x, \phi  , \theta) $. 

The integral in \eqref{3eq: expanding derivatives of I} is of the form $I_{\alpha+\kappa, \, \beta+ \tau} (\lambdaup)$ as defined in \eqref{1eq: def I alpha beta} in Lemma \ref{lem: Van der Corput, polar}, whose amplitude $u (x, \phi)$ satisfies
\begin{align*}
\partial_{x}^{\nu} \partial_{\phi}^{\shskip \mu} u (x, \phi) \Lt_{\, \nu, \shskip \mu} S X^{ \gamma + \delta - (\alpha + \kappa + \beta + \tau)/2 } \cdot X^{ \nu + \mu },
\end{align*}
which follows easily from  \eqref{3eq: bound w} and \eqref{3eq: derivatives of p}. In conclusion, Lemma \ref{lem: Van der Corput, polar} yields the following stationary phase estimates.

\begin{lem}\label{lem: bound for I, sp}
	Assume that $ 1 / 2  \leqslant \rho \leqslant 2 $ and that $1 \leqslant X \leqslant \sqrt \lambdaup$. We have
	\begin{align*}
	\lambdaup^{\gamma} \frac {\partial^{\gamma + \delta}} {\partial \lambdaup^{\gamma} \partial \theta^{\shskip\delta } } 	I (\lambdaup, \theta) \Lt_{\, \gamma, \shskip \delta } \frac {S X^{\gamma + \delta}} {\lambdaup} . 
	\end{align*}
\end{lem}

\subsection{Analysis of the Hankel transform}  \label{sec:analysis of Hankel transform}

To start with, the following lemma is a direct consequence of Lemma \ref{prop: asymptotic J} and \ref{lem: bound for z<1, J}. Note that in Lemma  \ref{lem: bound for z<1, J} we may choose $\sigma = \frac 1 3$ ($> \frac 7 {32} \geqslant |\Re \,\mu|$), say. 

\begin{lem}\label{lem: pre bound for W}
	Suppose that $\tw (z)$ is a smooth function with support in $\left\{ z : |z| \in [1, 2] \right\}$. Let $W (u)$ be the  Hankel transform of $\tw (z)$ defined as in \eqref{1eq: Hankel, integral}.  Then
	\begin{align*}
	u^{\alpha} \widebar u^{\shskip\beta} (\partial /\partial u)^\alpha (\partial / \partial \widebar u  )^{\beta} W (u)  \Lt_{\, \alpha,\shskip \beta}   \|\tw\|_{L^\infty} \cdot \big( |u|^{1/3} + 1 \big)^{\alpha + \beta} / |u|^{  2/3} ,
	\end{align*}
	where $ \|\tw\|_{L^\infty} $ is the sup-norm of $\tw$. 
\end{lem}

Another simple consequence of Lemma   \ref{prop: asymptotic J} is the following lemma. 

\begin{lem}\label{lem: bound for the Hankel of E}
Let $S > 0$ and $X \geqslant 1$. Let $\gamma$, $\delta$ and $A$ be   nonnegative integers. Suppose that $\tw (z)$ is a smooth function with support in $\left\{ z : |z| \in [1, 2] \right\}$ and derivatives satisfying $  (\partial/\partial z)^{\alpha} (\partial/\partial \overline z)^{\beta} \tw (z) \Lt_{\, \alpha, \shskip \beta} S X^{\alpha + \beta }  $ for all $\alpha$,  $\beta$. Let $W (u)$ be the Hankel transform of $\tw (z)$ defined as in \eqref{1eq: Hankel, integral}.   Then for $|u| \Gt 1$ we have
	\begin{align*}
		u^{\shskip\gamma} \widebar u^{\shskip\delta} (\partial /\partial u)^\gamma (\partial / \partial \widebar u  )^{\delta} W (u)  \Lt_{\, \gamma,\shskip \delta, \shskip A}  S X^{2A} / |u|^{(2A+2 - \gamma  - \delta  )/3}  .
	\end{align*} 
\end{lem}

\begin{proof}
	By Lemma \ref{prop: asymptotic J}, with $K$ large, say $K = 2 A$, we have to bound   integrals  of the following form 
	\begin{align*}   
	\sum_{ \xi^3 =  1 }    {\iint}   e \big(    3   \xi (u z )^{1/3}  + 3 \overline \xi (\overline{ uz} )^{1/3}     \big) a \big(\xi z^{1/3} \big)  d z ,
	\end{align*} 
	with the amplitude $a (z)$ satisfying
	\begin{align*}
 (\partial/\partial z)^{\alpha} (\partial/\partial \overline z)^{\beta} a (z) \Lt S X^{\alpha + \beta}  |u |^{ (\gamma + \delta - 2 ) /3}.
	\end{align*}
	Changing the variable $ z$ to $z^3$, the integral turns into
	\begin{align*}
	 9 {\iint}   e \big(    3      u^{1/3} z + 3 \overline u^{1/3} \overline z    \big) a (z ) |z|^4 d z.
	\end{align*} Now define the differential operator $\mathrm{D} = (\partial/\partial z)  (\partial/ \partial \overline z)$ so that $\mathrm{D} \big( e \big(    3    u^{1/3} z + 3 \overline u^{1/3} \overline z  \big)   \big) = - 36 \pi^2 |u|^{2/3} e \big(    3    u^{1/3} z + 3 \overline u^{1/3} \overline z  \big)   \big)$. By repeating partial integration with respect to $\mathrm{D}$ we get the desired estimate. 
\end{proof}

Next, we consider the Hankel transform of functions of the form \begin{align}\label{4eq: w = e v + e v}
\tw (z, \varLambda) = e (- 4 \Re (\varLambda \sqrt z)) \tv (\sqrt z) + e (4 \Re (\varLambda \sqrt z)) \tv (-\sqrt z). 
\end{align}
Here $ \varLambda  $ is a nonzero complex number and $\tv (z)$ is a smooth function with compact support in $\big\{ z : |z| \in \left[1, \sqrt 2 \right] \big\}$ and $ (\partial/\partial z)^{\alpha} (\partial/\partial \overline z)^{\beta} \tv (z) \Lt_{\, \alpha, \shskip \beta} S X^{\alpha + \beta } $ for $S > 0$ and $X \geqslant 1$. Note that these estimates  are equivalent to $  (\partial /\partial x) ^{\alpha}  (\partial /\partial \phi)^{\shskip \beta} \tv  \big( x e^{i\phi}\big)   \Lt_{\, \alpha, \shskip \beta  } S   X^{  \alpha + \beta} $ in the polar coordinates.

\begin{lem}\label{lem: bound for tilde W}
	Let $\gamma$, $\delta$, $A, K $ be  nonnegative integers. Let $ \tw (z, \varLambda) $ be defined as above and $W (u, \varLambda)$ be the  Hankel transform of $\tw (z, \varLambda)$. Define
	\begin{align*}
	\widetilde W (u, \varLambda) = e \lp - 2 \shskip \Re \lp u/\varLambda^2 \rp \rp W (u, \varLambda). 
	\end{align*}
	Let $y \Gt 1$. When either $ |\varLambda| > 2 y^{1/3} $ or $|\varLambda| < y^{1/3}/ 2 $, we have 
	\begin{align*}
	y^{\gamma} \frac {\partial^{\gamma + \delta}} {\partial y^{\gamma} \partial \theta^{\shskip \delta} } \widetilde W \big(y e^{i \theta}, \varLambda \big)  \Lt \frac {S} {y^{2/3} } \Bigg( \hskip -1 pt \lp |\varLambda| + \frac y {|\varLambda|^2} + X \rp^{\gamma + \delta}       & \min \left\{  \frac {X} {|\varLambda|}, \frac {X} {y^{1/3}} \right\}^A \\
	& + \frac {\big(y^{1/3} + y / |\varLambda|^2\big)^{\gamma+\delta}} {y^{K/3}} \hskip -1 pt \Bigg).
	\end{align*}
	When $y^{1/3}/ 2  \leqslant |\varLambda| \leqslant 2 y^{1/3}$, under the assumption $ X \leqslant  \sqrt y / |\varLambda| $, we have
	\begin{align*}
	y^{\gamma} \frac {\partial^{\gamma + \delta}} {\partial y^{\gamma} \partial \theta^{\shskip \delta} } \widetilde W \big(y e^{i \theta}, \varLambda \big)  \Lt \frac {S} {y^{2/3} } \lp \hskip -1 pt \frac {|\varLambda|^2 X^{\gamma + \delta} } {y } + \frac {y^{(\gamma+\delta)/3}} {y^{K/3}} \hskip -1 pt \rp.
	\end{align*}
	All the implied constants depend  only on $\gamma$,   $\delta$,   $A$, $ K$.
\end{lem}

\begin{proof}
	In view of Lemma \ref{prop: asymptotic J}, we have to bound  the following integral and its derivatives
	\begin{align*}
	I \big(   u^{1/3} , \varLambda \big) \hskip -1 pt = \hskip -1 pt \sum_{ \xi^3 =  1 } \sum_{ \zeta^2 = 1 } \hskip -1 pt  \sideset{}{ }{\iint}   e \big(   \Re \big( 6 \shskip \xi (u z )^{1/3} \hskip -1 pt  - 4 \shskip \zeta \varLambda \sqrt {z}  - 2 u / \varLambda^2  \big)  \big) \tv \big(\zeta \sqrt z \big)   \boldsymbol{W} \big(\xi (uz)^{1/3} \big)  d z ,
	\end{align*} 
	in which $$  \boldsymbol{W} (z) = \underset{ k + l \leqslant K -1 }{\sum \sum}  \frac { 
	B_{k } B_{l} }   {  z^{1+k } \shskip \overline z^{\shskip 1 +l } }, $$ with the coefficients $B_k = B_k (\mu, 0, - \mu)$ as in  Lemma \ref{prop: asymptotic J}. Substituting the variables $z$ and $u$ by $  z^6$ and $u^{3}$,  we have
	\begin{align*}%\label{4eq: I (u, r), 1}
	I   ( u  , \varLambda )   =   \iint e \lp   \Re \big(  6 u z^2   - 4 \varLambda z^3    - 2 u^3  / \varLambda^2  \big) \rp  a   ( z  ) d z / |z| ,
	\end{align*}
	where the weight function $ a (z ) = 36 \shskip |z|^{11} \tv  ( z^{3}  ) \boldsymbol{W}  (u  z^2 ) $ is supported in  $\big\{ z : |z| \in \big[1,   2^{1/6} \big] \big\}$ and  satisfies $ (\partial/\partial z)^{\alpha} (\partial/\partial \overline z)^{\beta} a (z) \Lt_{\, \alpha, \shskip \beta, \shskip K } (S /|u|^{2})  X^{\alpha + \beta } $ (by our assumption that $|u| \Gt 1$).

	Without loss of generality, we   assume that $\varLambda > 0$. Letting $z = x e^{i \phi} $ and $u = y e^{i \theta}$ and then replacing $x$ by $x y/\varLambda$, we have
	\begin{align*}%\label{4eq: I (u, r), 2}
	I   ( y  e^{  i \theta} , \varLambda )   =  \frac {2y} {\varLambda}   \int_0^{2\pi} \hskip -1 pt \int_0^\infty \hskip -1 pt e \lp  \frac {2 y^3} {\varLambda^2} f \lp x, \phi; \theta \rp \rp  a  \lp  \frac {x y e^{i \phi}} {\varLambda} \rp    d x d \phi  .
	\end{align*} 
	According to the notation  in \S \ref{sec: oscillatory integrals, 2}, the phase function $f \lp x, \phi; \theta \rp $ is given  by \eqref{4eq: phase f} and  this integral is of the form $I (\lambdaup, \theta)$ as in \eqref{3eq: defn I, 2} if one let $\lambdaup = 2 y^3 / \varLambda^2$, $\rho = \varLambda/y$. Consequently, the estimates in this lemma follow directly from Lemma \ref{lem: off the range} and Lemma \ref{lem: bound for I, sp}. Recall that the variable  $u$ was changed into $u^3$ at an early stage of the proof.
	%\begin{align}\label{4eq: V2, first expression}
	%I   ( y e^{i \theta} , \varLambda )   =   \hskip -2 pt   \int_0^{2\pi} \hskip -1 pt \int_0^\infty \hskip -1 pt e \lp     6 y x^2 \cos (\theta + 2 \phi)   - 4 r x^3 \cos (3 \phi)  - 2 y^3 \cos (3 \theta) / \varLambda^2   \rp  a  \big( x e^{i \phi} \big)    d x d \phi  ,
	%\end{align}
\end{proof}

\begin{rem}
	Note that we have modified the Hankel-transform integral $W $ by a phase factor. It is the same phenomenon as in \cite{CI-Cubic} and \cite{Blomer}{\rm:}  this phase factor  comes up naturally in both the analytic and the arithmetic  part.
\end{rem}

Finally, we conclude this section with the following lemma as the complex analogue of Lemma 10 in Blomer \cite{Blomer}. It should be stressed that our normalization of the Hankel transform is different from that of Blomer, as indicated in Remark \ref{rem: normalization of Hankel}.

\begin{lem} \label{lem: main}
	Fix $\varepsilon > 0$, $T \geqslant 1$, nonnegative integers  $\gamma, \delta$, and integer $A $ sufficiently large.  Then there are integers $A'$ and $A''$ sufficiently large  {\rm(}in terms of $\varepsilon$, $\gamma, \delta $ and $A${\rm)} with the following property.   Let $X \geqslant 1$, and let $\tv (z)$ be a smooth   function with support in $\{ z : |z| \in  [1,    2  ] \}$ satisfying $ (\partial /\partial x) ^{\alpha}  (\partial /\partial \phi)^{\shskip \beta}  \tv  \big( x e^{i\phi}\big)   \Lt_{\, \alpha, \shskip \beta  }     X^{  \alpha + \beta} $ for all $\alpha, \beta$. Define $H (z) $     as in  \eqref{1eq: H's} with the choice of $h (t)  $   in {\rm \S \ref{sec: the Bessel integral}} {\rm(}see also {\rm \S \ref{sec: L-functions})}, with    $\Re \, \varv = \varepsilon$. For $\varLambda \neq 0$, define
	\begin{align}
	\tw (z) = \tw (z, \varLambda) = H \big( \varLambda \sqrt z \big) \tv (z),
	\end{align}
	and its  Hankel transform $ W (u) = W (u, \varLambda )$ as in \eqref{1eq: Hankel, integral}. Let
	\begin{align}\label{6eq: tilde W = e W}
	\widetilde W (u, \varLambda) = e \lp - 2 \shskip \Re \lp u/\varLambda^2 \rp \rp W (u, \varLambda). 
	\end{align}
	Then for any %$(y, \theta) \in (0, \infty) \times [0, 2 \pi)$  and any 
	$Q > (T X)^{A''} (y+1/y) (|\varLambda| + 1 / |\varLambda|)$ we have the following bounds
	\begin{equation}\label{6eq: bounds for tilde W}
	\begin{split}
	 y^{\gamma} & \frac {\partial^{\gamma + \delta}} {\partial y^{\gamma} \partial \theta^{\shskip \delta} } \widetilde W \big(y e^{i \theta}, \varLambda \big) \Lt_{\, \varepsilon, \gamma, \shskip  \delta, \shskip A} T^{5 +\varepsilon} \text{{\scalebox{1.3}{$\cdot$}}} \\
	& \hskip 25 pt \left\{ 
	\begin{split} 
	& Q^{-A}, \hskip 34 pt \text{if } |\varLambda |\leqslant Q^{- \varepsilon/12}, \\
	& Q^{\varepsilon} / |\varLambda|  y^{2/3},   \hskip 10 pt \text{if } y \leqslant Q^{\varepsilon}, |\varLambda | > Q^{-\varepsilon/ 12}, \\
	& Q^{-A}, \hskip 34 pt \text{if } y > Q^{\varepsilon}, |\varLambda | > Q^{- \varepsilon / 12},  |\varLambda |  < y^{1/3}/2  \text{ or } |\varLambda |  > 2 y^{1/3}, \\
	& Q^{\varepsilon} |\varLambda| / y^{5/3}, \hskip 10 pt \text{if } y > Q^{\varepsilon},  y^{1/3}/2  \leqslant |\varLambda |  \leqslant 2 y^{1/3}.
	\end{split}\right.
	\end{split}
	\end{equation}
	In particular, the following uniform bound holds{\rm:}
	\begin{align}\label{6eq: bounds for tilde W, uniform}
y^{\gamma} & \frac {\partial^{\gamma + \delta}} {\partial y^{\gamma} \partial \theta^{\shskip \delta} } \widetilde W \big(y e^{i \theta}, \varLambda \big) \Lt_{\, \varepsilon, \gamma, \shskip  \delta } T^{5 +\varepsilon} Q^{\varepsilon} /    |\varLambda| y.
	\end{align} 
\end{lem}

\begin{proof}
The first and second bound follow directly from Lemma \ref{lem: Bessel integral}	and \ref{lem: pre bound for W}. The
other two bounds will follow mainly from  Lemma \ref{lem: bound for tilde W}, but require some work.  We partition the integral \eqref{1eq: H's} defining $H (z)$ into three parts according to the subdivision of the real line into three ranges: 
\begin{equation*}
 |t| \geqslant T Q^{\varepsilon/12}, \hskip 15 pt    |t| \leqslant \textstyle \sqrt {|\varLambda|} - 1, |t| < T Q^{\varepsilon/12}, \hskip 15 pt  \textstyle \sqrt {|\varLambda|} - 1 < |t| < T Q^{\varepsilon/12}.
\end{equation*} 
According to our decomposition we write $H = H^{\flat} + H^{\natural} + H^{\sharp}$, $ \tw = \tw^{\flat} + \tw^{\natural} + \tw^{\sharp}$ and $\widetilde W = \widetilde W^{\flat} + \widetilde W^{\natural} + \widetilde W^{\sharp}$. 

In the first case, we use the exponential decay of $h (t)$ to see that $$ \| \tw^{\flat} \|_{L^\infty} \Lt T^{5+\varepsilon} Q^{  \varepsilon/2} \exp (- Q^{\varepsilon/6}) .$$ Hence Lemma \ref{lem: pre bound for W} implies that $ \widetilde W^{\flat}$ and its derivatives are negligible. 

In the second case, we insert the expression \eqref{1eq: J = W + W} for the Bessel function $\boldJ_{it} (z)$ from Lemma \ref{lem: properties of J} into \eqref{1eq: H's}. For the main terms, apply Lemma \ref{lem: bound for tilde W}   with $$S = |h (t)| t^2 / |\varLambda| \Lt (|t| + 1)^{2 + \varepsilon} \exp (-t^2/T^2) / |\varLambda|$$ and $X$  as in the present lemma,  getting the desired bounds for $\widetilde W^{\natural}$. For the error term containing $\boldsymbol{E}_K$, Lemma \ref{lem: bound for the Hankel of E} and \eqref{1eq: derivatives of E} in Lemma   \ref{lem: properties of J} show that it is negligibly small. Note that in   this case we have necessarily $|\varLambda | \geqslant 1$. 

In the third case, as $ \boldJ_{it} (z) $ is even, we may artificially write $ 2 \boldJ_{it} (\varLambda \sqrt z) $ in the form of \eqref{4eq: w = e v + e v} with the $\tv$ function being $  \widetilde  \tv_{it} (  z, \varLambda) =  e (4 \Re (\varLambda z) ) \boldJ_{it} (\varLambda   z)$. 
It follows from Lemma \ref{lem: derivatives of J} that 
\begin{align*}
 (\partial/\partial z)^{\alpha} (\partial/\partial \overline z)^{\beta} \hskip -2 pt \lp h (t) t^2 \, \widetilde  \tv_{it} (  z, \varLambda)   \rp   \Lt_{\, \alpha, \shskip \beta  } \hskip -1 pt (|t| + 1)^{4 + \varepsilon} \exp (-t^2/T^2) \lp  |\varLambda| +  (|t|+1)  / |\varLambda| \rp^{\alpha + \beta} \hskip -2 pt .
\end{align*}
Note that $Q^{-\varepsilon / 12} < |\varLambda| < (|t|+1)^2 < T^2 Q^{\varepsilon / 6}$ and hence $$ X + |\varLambda| + (|t| +1) / |\varLambda| < X +2 T^2 Q^{\varepsilon / 6} < Q^{\varepsilon / 4} < y^{1/3} / Q^{\varepsilon / 12}. $$ By the first estimate in Lemma \ref{lem: bound for tilde W}, with the $X$ there being $X + |\varLambda| + (|t| +1) / |\varLambda|$, we infer that $\widetilde W^{\sharp}$ and its derivatives are negligibly small. 
\end{proof}

\section{Setup}

%We have now completed the analytic investigations of the Bessel integral for  $\GL_2 (\BC)$ and its  $\GL_3 (\BC)$ Hankel transform. 

We are now ready to start the proof of Theorem \ref{thm: main}.  We shall follow Blomer \cite{Blomer} very closely.

 We start with introducing the spectral mean of $L$-values,
\begin{align*}
\sideset{}{'}{\sum}_{ j } \omega^{\star}_j k  (t_j)  L \lp \tfrac 1 2 ,   \pi \otimes  f_j \otimes \chiup\rp   +  \frac 1 {2 \pi}   \int_{-\infty}^{\infty} \omega^{\star} (t) k  (t)   \left| L \lp \tfrac 1 2 + it,   \pi \otimes \chiup \rp \right|^2 d t,
\end{align*}
in which the spectral weight $k  (t)$ is defined as in \eqref{4eq: h, 1}.
In view of \eqref{1eq: lower bound for omega j}, \eqref{1eq: lower bound for omega (t)}, \eqref{4eq: h (t) > 0}, \eqref{1eq: p (1/2, t)}, along with the positivity of the $L$-values, we need to prove that this is bounded by $ |q|^{1/2 + \varepsilon} $. Applying the approximate functional equations \eqref{1eq: approximate functional equation, 1}, \eqref{1eq: approximate functional equation, 2},  the spectral mean may be written as
\begin{align*}
\frac 1 {8}   \underset{n_1, \shskip n_2}{\sum\sum} \frac {A (n_1, n_2) \chiup (n_2) } { |n_1^2 n_2   |} \Bigg( &   \sideset{}{'}{\sum}_{ j } \omega^{\star}_j k (t_j)  \lambdaup_j ( n_2)    V \lp   \frac { |n_1^2 n_2   |} {|q|^3} ; t_j  \hskip -1 pt \rp   \\
& + \frac 1 {2 \pi}   \int_{-\infty}^{\infty} \omega^{\star} (t) k (t )    \eta \lp n_2, \tfrac 1 2 + it \rp  V \lp \frac { |n_1^2 n_2   |} {|q|^3} ; t   \hskip -1 pt \rp dt \Bigg)   .
\end{align*}
From now on, we shall not pay attention to the dependence
of implied constants on $T$. It should be noted that the dependency on $T$ in our estimates will be polynomial at each step. 

By \eqref{1eq: derivatives for V(y, t), 1} in Lemma \ref{lem: afq} (1), we may truncate the sum over $n_1, n_2 $ at
$ |n_1^2 n_2   | \leqslant |q|^{3+\varepsilon}$
with the cost of a negligible error. We then apply \eqref{1eq: approx of V} in Lemma \ref{lem: afq} (1), in which   we choose $U = (\log |q|)^2$. The error term is again negligible, and we need to prove 
\begin{align*}
 \underset{ |n_1^2 n_2   | \leqslant |q|^{3+\varepsilon} }{ \sum \sum}   \frac {A (n_1, n_2) \chiup (n_2) } { |n_1^2 n_2   |^{1+2\varv } }  \Bigg( & \sideset{}{'}{\sum}_{ j }   \omega^{\star}_j h (t_j)  \lambdaup_j (n_2)  \\
&
\hskip -1 pt + \hskip -1 pt \frac 1 {2 \pi}   \int_{-\infty}^{\infty} \omega^{\star} (t) h (t)   \eta \lp n_2, \tfrac 1 2 + it \rp dt \hskip -1 pt \Bigg) \hskip -1 pt \Lt   |q|^{1/2 + \varepsilon} ,
\end{align*}
uniformly in $\varv  \in \left[ \varepsilon - i (\log |q|)^2, \varepsilon + i (\log |q|)^2  \right]$.

We now apply the Kuznetsov formula \eqref{1eq: Kuznetsov, new} (with $n_1 = 1$ therein),   obtaining  a diagonal term
\begin{align*}
\frac {2} {\,\pi^2} H \underset{ |n_1   | \leqslant |q|^{3/2+\varepsilon} }{ \sum}   \frac {A (n_1, 1)   } {\,  |n_1     |^{2 +4\varv } },
\end{align*}
and an off-diagonal term 
\begin{align*}
\frac 1 {8 \pi^2} \underset{ |n_1^2 n_2   | \leqslant |q|^{3+\varepsilon} }{ \sum \sum} \hskip -2 pt \frac {A (n_1, n_2) \chiup (n_2) } { |n_1^2 n_2   |^{1+2\varv } } \sum_{  q | c } \frac {S (  n_2, \epsilon  ; c)} {|c|^2}  H \lp  \frac {\sqrt {\epsilon  n_2 } } {2 c} \rp.
\end{align*}
It follows from \eqref{3eq: R-S Fourier}, along with Cauchy-Schwarz,  that the diagonal term is bounded by $|q|^{\varepsilon}$. 
As for the off-diagonal term, we reformulate it by the Hecke relation  \eqref{1eq: Hecke relation} for $A (n_1, n_2)$, getting
\begin{align*}
\underset{ |\delta^3  n_1^2   | \leqslant |q|^{3+\varepsilon} }{ \sum \sum}  \frac { \mu (\delta) \chiup (\delta) A (n_1, 1)} { |\delta^3 n_1^2  |^{1+2\varv} } \hskip -2 pt \sum_{ |n_2| \leqslant |q|^{3+\varepsilon} / |\delta^3 n_1^2|   } \hskip -2 pt \frac { A \lp 1, n_2   \rp \chiup (n_2) } {|n_2|^{1+2\varv}}  \sum_{  q | c } \frac {S ( \delta n_2 , \epsilon  ; c)} {|c|^2}  H \lp  \frac {\sqrt {\epsilon \delta n_2 } } {2 c} \rp \hskip -1 pt.
\end{align*}
Note that in this sum, we may assume that $\delta$ is square-free and that $(\delta, q) = \frO$. Finally, by a smooth partition of unity, it is reduced to proving the following proposition.

\begin{prop}\label{prop: S}
	 Let $q \in \frO$ be prime. Let $\delta \in \frO$ be  square-free such that $(\delta, q) = \frO$. Let $\epsilon = 1$ or $i$. Assume 
	\begin{align}\label{7eq: N < q3}
	 N  \leqslant |q|^{3+\varepsilon}/|\delta|^3.
	\end{align} Put $X = (\log |q|)^2$. Let $\tv $ be a smooth function  supported on $[1 , 2   ]$ satisfying $  \tv^{(\alpha)}  (x) \Lt_{\, \alpha} X^{\alpha} $ for all $\alpha$.  Suppose that $H (z)$ is the Bessel transform of $h (t)$ given by \eqref{1eq: H's}, with $h (t) $ defined as in \eqref{4eq: h}-\eqref{4eq: g and k}. Define 
	 \begin{equation}\label{5eq: def S (N; c; d)}
	 \begin{split}
	 & \mathcal S_\delta^{\epsilon} (q, N)  = \sum_{  q | c } \frac 1 {|c|^2} \underset{n}{ \sum}   A (1, n )  \chiup (n)    {S (  \delta n  , \epsilon ; c)}   \tw \lp \frac  {n} {N} ; \frac {\sqrt {\epsilon \delta N } } {2 c}\rp,
	 \end{split}
	 \end{equation}
	 with
	 \begin{align}
	 \label{5eq: w H}
	 \tw (z; \varLambda) = \tv (|z|)  H \lp \varLambda \sqrt {z} \rp.
	 \end{align} Then
	\begin{equation}\label{7eq: bound for S}
	\begin{split}
	 \mathcal S_{\delta}^{\epsilon} (q, N)  
	\Lt_{\, \varepsilon}  |q|^{7/32+\varepsilon }   |\delta| N + |q|^{1/2+\varepsilon} N, 
	\end{split}
	\end{equation}
	for any $\varepsilon > 0$.
\end{prop}
 
\section{Application of the \Voronoi summation}

We now consider the sum $ \mathcal S_\delta^{\epsilon} (q, N)$ defined in \eqref{5eq: def S (N; c; d)}. In order to prepare for the
application of \Voronoi summation, we open the Kloosterman sum ${S (  \delta n  , \epsilon ; c)}$ and split the
$n$-sum into residue classes modulo $c$, and $\mathcal S_\delta^{\epsilon} (q, N)$ becomes
\begin{align*}
\sum_{  q | c } \frac 1 {|c|^2} \sumx_{b (\mod c) } \hskip -2 pt e \lp \hskip -1 pt \Re \hskip -1 pt \lp \frac {\epsilon \widebar {b} } {c} \rp \hskip -1 pt \rp \hskip -2 pt \sum_{ a (\mod c) } \chiup (a) e \lp \hskip -1 pt \Re \hskip -1 pt \lp \frac { \delta a b} {c} \rp \hskip -1 pt \rp \hskip -2 pt \sum_{ n   \equiv  a (\mod c) } \hskip -2 pt A (1, n) \tw \lp \frac  {n} {N} ; \frac {\sqrt {\epsilon \delta N } } {2 c}\rp \hskip -2 pt.
\end{align*}
We then detect the summation condition $ n   \equiv  a (\mod c)$ by primitive additive characters modulo $c_1 | c$,
\begin{align*}
\mathcal S_\delta^{\epsilon} (q, N) = \sum_{  q | c }  \frac 1 {4 |c|^4} \sum_{ c_1 | c} \  \sumx_{b_1 (\mod c_1 \hskip -0.5 pt)} \ \sumx_{b (\mod c) } e \lp \hskip -1 pt \Re \hskip -1 pt \lp \frac {\epsilon \widebar {b} } {c} \rp \hskip -1 pt \rp \hskip - 2 pt \sum_{ a (\mod c) } \hskip - 2 pt  \chiup (a) e \lp \hskip -1 pt \Re \hskip -1 pt \lp \frac { \delta a {b}    } {c} - \frac {a \widebar b_1} {c_1} \rp \hskip -1 pt\rp & \\
 \sum_{ n   } A (1, n)  e \lp \hskip -1 pt \Re \hskip -1 pt \lp \frac {n \widebar b_1} {c_1} \rp \hskip -1 pt \rp \tw \lp \frac  {n} {N} ; \frac {\sqrt {\epsilon \delta N } } {2 c}\rp &.
\end{align*}
For the innermost sum we can now apply the \Voronoi formula in Proposition \ref{prop: Voronoi}, getting
\begin{equation}\label{8eq: S (q, N)}
\begin{split}
 \mathcal S_\delta^{\epsilon} (q, N)   
 =  \frac {N^2} {64} \sum_{  q | c } \frac {1 } {  |c|^4}   \sum_{ c_1 | c} \frac {1} {  |c_1 |^4 } & \sum_{ n_1 |  c_1     } |n_1|^2 
 \sum_{ n_2   } A (n_2, n_1) \\
&   T^{\epsilon}_{\delta, \shskip n_1, \shskip n_2 } (c, c_1; q) W \lp  \frac {n_1^2 n_2 N} { 8 c_1^3 } ;  \frac {\sqrt {\epsilon \delta N } } {2 c}\rp ,
\end{split}
\end{equation}
where
\begin{equation}\label{8eq: char sum T(n,n;c,c)}
\begin{split}
& T^{\epsilon}_{\delta, \shskip n_1, \shskip n_2 } (c, c_1; q) \\
& =   \sumx_{b_1 (\mod c_1 \hskip -0.5 pt)} \ \sumx_{b (\mod c) }       \sum_{ a (\mod c) }    \chiup (a) S \lp   b_1  ,   n_2; c_1   / n_1 \rp e \hskip -1 pt \lp \hskip -1 pt \Re \lp \frac { \epsilon \widebar {b} + \delta a {b}  } {c} - \frac {a \widebar b_1} {c_1} \rp \hskip -1 pt \rp  \hskip -1 pt,
\end{split}
\end{equation}
and the function $ W \lp   u; \varLambda \rp $ is the Hankel transform of $\tw (z; \varLambda)$, which has been studied in \S \ref{sec:analysis of Hankel transform} (in particular, see  Lemma \ref{lem: main}).

\section{Transformation of the character sum}

\subsection{}
The computations  in \cite[\S 6]{Blomer} may be applied almost identically for  the character sum $T^{\epsilon}_{\delta, \shskip n_1, \shskip n_2 } (c, c_1; q) $, so we shall only give a summary of the final results in the following lemma (see \cite[Lemma 12]{Blomer}). Note that  $\chiup (-1) = 1$.

\begin{lem}\label{lem: T = e V}
Let $q, \delta \in \frO$ be square-free. Let $\epsilon = 1$ or $ i$.	Assume that \begin{align}\label{9eq: conditions}
	(\delta, q) = (r', c_2) = \frO, \hskip 10 pt \delta_0 | c_2,
	\hskip 10 pt n_1 | c_1, \end{align} 
	with the following notation, \begin{align}\label{9eq: notation }
	c = q r = c_1 c_2,  (\delta_0) = (\delta, r) = (\delta, c), \  c' = c / \delta_0,   c_2' = c_2 / \delta_0,   \delta' = \delta' / \delta_0,  r' = r  /\delta_0.  
	\end{align}
	Then
	\begin{equation}\label{9eq: T = e V}
	\begin{split}
	T^{\epsilon}_{\delta, \shskip n_1, \shskip n_2 }  (c, c_1; q) & =    e \lp - \Re \lp  \frac {\overline \epsilon \overline \delta '  c_2'^2 c_2 n_1^2 n_2} {c'} \rp \rp \frac {\varphi (c_1) \varphi (c_1/n_1)} {\varphi (c')^2} \frac {\mu (\delta_0 ) \chiup (\delta) } {|\delta_0|^2} |r^2 q|^2 \\
	& \cdot V^{\epsilon}_{\delta, \shskip n_1, \shskip n_2 } (c, c_1; q) \underset{\sstyle f_1 f_2 d_2' = r' }{\sum \sum \sum} \frac {\mu (f_2)  } {16 |f_1|^2} e \lp \Re \lp \frac {\overline \epsilon \overline \delta ' c_2'^2 c_2 n_1' n_1  n_2  \overline {d_2' q} }  {f_1} \rp \rp ,
	\end{split}
	\end{equation}
	where
	\begin{equation}\label{9eq: V}
	\begin{split}
	V^{\epsilon}_{\delta, \shskip n_1, \shskip n_2 } (c, c_1; q) = \underset{b_2, \shskip b_3 (\mod q)}{\sum \sum} & \chiup (b_2 b_3) \chiup (b_2 r' + c_2 b_3 r' - \overline\epsilon  c_2' c_2 n_1 n_2) \\
	& \chiup (r' b_3 - \overline\epsilon n_1 n_2 c_2') e \lp \Re \lp \frac {\overline \delta'  c_2' n_1 (b_2 + c_2 b_3) } q \rp \rp,
	\end{split}
	\end{equation}
	the summation is subject to the following conditions,
	\begin{align}\label{9eq: conditions, sum}
	f_1 f_2 d_2' = r', \hskip 10 pt (d_2', f_1 n_1 n_2) =   (f_1, f_2) = (f_1 f_2, q) = \frOO, \hskip 10 pt \mu^2 (f_1) = 1 \hskip 10 pt f_2 | n_1  ,
	\end{align}
	and 
	\begin{align}\label{9eq: n1'}
	n_1' =   n_1 / f_2. 
	\end{align}
	We have $T^{\epsilon}_{\delta, \shskip n_1, \shskip n_2 } (c, c_1; q) = 0$ if one of the conditions $ (r', c_2) = \frO$, $\delta_0 | c_2$ is not satisfied.
\end{lem}

\begin{rem}
	In the course of computations, one needs the Selberg-Kuznetsov identity  for Kloosterman sums over the Gaussian field $\BF $,
	\begin{align*}
	S (n_1, n_2; c) = \frac 1 4 \sum_{d| n_1, \shskip d|n_2, \shskip d|c} |d|^2  S \lp \frac {n_1 n_2} {d^2}, 1; \frac c d \rp, 
	\end{align*}
	or, after the M\"obius inversion,
	\begin{align*}
	S (n_1  n_2, 1; c) = \frac 1 4 \sum_{d| n_1, \shskip d|n_2, \shskip d|c} |d|^2 \mu (d) S \lp \frac {n_1} d, \frac {n_2} d; \frac c d \rp, 
	\end{align*} which should be thought of as a kind of the Hecke multiplicativity relation. This identity may be proven using either the Kloosterman-weighted Kuznetsov formula for $\SL_2 (\frO)$ in \cite{B-Mo2} or the elementary arguments in \cite{M-Selberg}.
\end{rem}

\subsection{}

As observed in \cite{Blomer},  the character sum $V^{\epsilon}_{\delta, \shskip n_1, \shskip n_2 } (c, c_1; q)$ defined in \eqref{9eq: V} has been studied in the work of Conrey and Iwaniec \cite{CI-Cubic}. Again, the calculations in \cite[\S 10]{CI-Cubic} are valid over the Gaussian field in an obvious way. 

As in  in \cite[(10.2)]{CI-Cubic}, we define $H_r (m, m_1, m_2; q)$ as follows, 
\begin{align}\label{9eq: Hr(m,m,m;q)}
H_r (m, m_1, m_2; q) = \underset{  u, v   (\mod q) }{  \sum  \sum } \chiup \big( v (u+v m_1) (vr-m) (ur + m m_1 )\big) e \lp \Re \lp \frac {um_2} q \rp \rp.
\end{align}
Moreover, define as in \cite[(10.7)]{CI-Cubic}
\begin{align}
H (w; q)  = \underset{  u, v   (\mod q) }{  \sum  \sum } \chiup \big( u v (u+1) (v+1)\big) e \lp \frac {(uv-1) w} q \rp.
\end{align}
Put 
\begin{align}
(h) = (r, q), \hskip 10 pt (k) = (m m_1 m_2, q/h), \hskip 10 pt \ell = q/ hk.
\end{align} According to \cite[(10.12, 10.13)]{CI-Cubic}, we have
\begin{align}
H_r (m, m_1, m_2; q)  =   \frac {|h|^2} {\varphi (k) } R (m; k) R (m_1; k) R (m_2; k) H ( \overline { hk r} m m_1 m_2; \ell),
\end{align} 
if $(h, m m_1 m_2) = \frO$, or else $ H_r (m, m_1, m_2; q) = 0$,
where $R (m; k) = S (m, 0; k)$ is the Ramanujan sum.

By comparing \eqref{9eq: V} and \eqref{9eq: Hr(m,m,m;q)}, we observe that
\begin{align*}
V^{\epsilon}_{\delta, \shskip n_1, \shskip n_2 } (c, c_1; q) = H_{r'} (\overline\epsilon  c_2' n_1 n_2, - c_2, \overline \delta'  c_2' n_1; q ).
\end{align*} 
Consequently, we have the the following lemma.

\begin{lem}
Let notation  be as in Lemma {\rm\ref{lem: T = e V}}.	When the assumptions in  \eqref{9eq: notation } and \eqref{9eq: conditions, sum} hold, we have
	\begin{align}\label{9eq: V (c, q), 2}
	V^{\epsilon}_{\delta, \shskip n_1, \shskip n_2 } (c, c_1; q) = \frac {|h|^2} {\varphi (k) } R (n_1 n_2 c_2'; k) R (c_2; k) R (n_1 c_2'; k) H ( - \overline {\epsilon  }  c_2'^2 c_2 n_1^2 n_2 \overline{\delta'  h k r'} ; \ell), 
	\end{align}
	with
	\begin{align}
	(h) = (r', q), \hskip 10 pt (k) = ( c_2'^2 c_2 n_1^2 n_2, q/h), \hskip 10 pt \ell = q/ hk.
	\end{align}
\end{lem}

\section{Further simplifications}

In view of (\ref{9eq: conditions}, \ref{9eq: notation }, \ref{9eq: conditions, sum}, \ref{9eq: n1'}), applying reciprocity twice, we infer that the product of the two exponential factors in \eqref{9eq: T = e V} is equal to
\begin{align*}
e \lp - \Re \lp  \frac {\overline \epsilon     c_2'^2 c_2 n_1^2 n_2} {\delta ' c'} \rp \rp e \lp   \Re \lp  \frac {\overline \epsilon  \delta_0   c_2'^2 f_2 n_1'^2 n_2 \overline { c_1' d_2'}} {\delta ' f_1} \rp \rp ,
\end{align*}
in which
\begin{align*}
c_1' = c_1 / r' = q /c_2' .
\end{align*}
The first factor in turn is exactly equal to the  $ e \lp - 2 \shskip \Re \lp u/\varLambda^2 \rp \rp $ in \eqref{6eq: tilde W = e W} when $u =   {n_1^2 n_2 N} / { 8 c_1^3 } $ and $ \varLambda =    {\sqrt {\epsilon \delta N } } / {2 c}$, and hence combines nicely with $W (u, \varLambda)$ to $\widetilde W (u, \varLambda)$.

At this point, we assume that $q$ is prime. Following  \cite[\S 7]{Blomer}, we  substitute \eqref{9eq: T = e V} and \eqref{9eq: V (c, q), 2} into \eqref{8eq: S (q, N)} and simplify the resulting sum. Indeed,  using (\ref{6eq: bounds for tilde W}, \ref{6eq: bounds for tilde W, uniform}) in Lemma \ref{lem: main}, \eqref{7eq: N < q3}, along with   \eqref{3eq: bounds for Fourier, 2}, it may be shown that the contributions from the terms with $c_2' = q$, $h = q$ or $k = q$ are small (negligibly small or at most $  |q|^{7/16 + \varepsilon} N $). Thus we may assume that $c_2' = h = k = 1$ and in particular that $c_1' = \ell = q$, $(d_2' n_1' n_2, q) = \frO$ and $n_1' | f_1$. Let $g = f_1 / n_1'$.  Now the simplified sum that we need to consider is as follows,
\begin{align*}
 \frac {\,N^2} {\, |q|^6} \underset{\delta_0 \delta' = \delta}{\sum \sum} \frac {\mu (\delta_0) \chiup (\delta_0)} {|\delta_0|^2}    \underset{ \sstyle d_2', \shskip f_2, \shskip g, \shskip n_1', \shskip n_2 } {\sum \cdot \hskip -1.5pt \cdot \hskip -1.5pt \cdot \sum } \frac { \mu (f_2 ) A (   n_2   , f_2 n_1'  ) } {\varphi (f_2 n_1')| f_2 g^3 n_1'^2 d_2'^2|^2  }   e \hskip -1.5pt \lp \hskip -1pt \Re \hskip -2pt \lp \frac {\overline \epsilon \delta_0 f_2 n_1' n_2  \overline {q d_2' } }  {\delta' g  } \rp \hskip -1pt \rp & \\ 
 H ( - \overline {\epsilon  }  \delta_0  f_2 n_1' n_2  \overline{\delta'    g   d_2'} ; q ) \widetilde W \lp  \frac {    n_2   N } { 8 f_2 n_1'   (q g d_2'     )^3   } ;  \frac {\sqrt {\epsilon \delta N } } {2 q \delta_0  f_2 g  n_1'   d_2'}\rp & ,
\end{align*}
subject to the conditions
\begin{align*}
  & ( f_2 g  n_1' d_2', \delta   ) = ( f_2 g  n_1' , q d_2') = (f_2, g n_1') = \frO, \\
  & (d_2', n_1'n_2) = (d_2'n_1' n_2, q) = \frO, \hskip 10pt  \mu^2 (g n_1'  ) = 1.
\end{align*}
 Furthermore, let
\begin{align*}
 (s) = (n_2, \delta' g ), \hskip 10 pt n_2' = n_2/s,
\end{align*}
pull the factor $s$ out of the numerator and denominator of the exponential, and then introduce the new variable $r$ to relax the coprimality condition $(d_2', n_2') = \frO$ by M\"obius inversion and separate these  two variables.
%After the simplifications, i
It reduces to proving that the following sum   has  the same bound as in \eqref{7eq: bound for S},
\begin{equation}\label{10eq: S}
 \begin{split}
& \frac {\,N^2} {\, |q|^6} \underset{\delta_0 \delta' = \delta}{\sum \sum} \frac {\mu (\delta_0) \chiup (\delta_0)} {|\delta_0|^2}   \underset{ \sstyle f_2, \shskip g, \shskip n_1', \shskip r \atop {\sstyle ( f_2 g  n_1', q \delta   ) = (f_2, g n_1') = \frO \atop {\sstyle \mu^2 (g n_1'  ) = 1  \atop {\sstyle (r, q \delta f_2 g n_1') = \frOO} } } }{\sum \sum \sum \sum}  \frac { \mu (f_2 r) } {\varphi (f_2 n_1')| f_2 g^3 n_1'^2 r^2|^2  }
  \sum_{ s | \delta' g}    {\mathcal S}_{  \shskip g, \shskip n_1, \shskip r, \shskip s}^{\shskip\epsilon, \shskip \delta_0, \shskip \delta'} (q, N),
 \end{split}
\end{equation}
with
\begin{equation}\label{10eq: S... (q,N)}
\begin{split}
 {\mathcal S}_{  \shskip g, \shskip n_1, \shskip r, \shskip s}^{\shskip\epsilon, \shskip \delta_0, \shskip \delta'} (q, N) = & \underset{ \sstyle  d_2',  \shskip n_2' \atop {\sstyle  (d_2',   q \delta     g  n_1 ) = \frO \atop {\sstyle  (n_2', q \delta' g / s) = \frO } } }{\sum \sum } \frac { A (r s n_2'  ,  n_1  ) } {|d_2'|^4} e \lp \Re \lp \frac {\overline \epsilon \delta_0  n_1 n_2' \overline {q d_2' } }  {\delta' g / s} \rp \rp  \\
& H ( - \overline {\epsilon  }  \delta_0   n_1  s n_2'  \overline{\delta'    g   d_2'} ; q ) \widetilde W \lp  \frac {   s n_2'  N } { 8  n_1 r^2 (q g d_2'     )^3   } ;  \frac {\sqrt {\epsilon \delta N } } {2 q \delta_0   g n_1 r d_2'}\rp.
\end{split}
\end{equation}
Note that the fraction in the exponential is in lowest terms.  

It is left to estimate the sum ${\mathcal S}_{  \shskip g, \shskip n_1, \shskip r, \shskip s}^{\shskip\epsilon, \shskip \delta_0, \shskip \delta'} (q, N)$. For this, we shall prove the following lemma in the last  section.

\begin{lem}\label{lem: final}
	Let notation  be as above. Assume that 
	\begin{align}\label{10eq: D N}
	1 \leqslant D_2 \leqslant |q|^{\varepsilon} \frac {|\delta |^{1/2} N^{1/2} } {| q \delta_0   g n_1 r | }, \hskip 10 pt 1 \leqslant N_2 \leqslant |q|^{\varepsilon} \frac {|\delta|^{3/2} N^{1/2} } {|\delta_0^3 n_1^2 r  s | } .
	\end{align}
	Define 
	\begin{equation}\label{10eq: S (q, N, D2, N2)}
	\begin{split}
&	{\mathcal S}_{  \shskip g, \shskip n_1, \shskip r, \shskip s}^{\shskip\epsilon, \shskip \delta_0, \shskip \delta'} (q, N, D_2, N_2) =   \underset{ \sstyle  D_2 \leqslant |d_2'| < 2 D_2  \atop {\sstyle  (d_2',   q \delta     g  n_1 ) = \frO   } }{\sum   } \underset{ \sstyle  N_2 \leqslant |n_2'| < 2 N_2  \atop {\sstyle  (n_2', q \delta' g / s) = \frO }   }{\sum   } \frac { A (r s n_2'  ,  n_1  ) } {|d_2'|^4}  \\
	& e \lp \hskip -1 pt  \Re \hskip -1 pt  \lp \frac {\overline \epsilon \delta_0  n_1 n_2' \overline {q d_2' } }  {\delta' g / s} \rp \hskip -1 pt  \rp  H ( - \overline {\epsilon  }  \delta_0   n_1  s n_2'  \overline{\delta'    g   d_2'} ; q ) \widetilde W \lp  \frac {   s n_2'  N } { 8  n_1 r^2 (q g d_2'     )^3   } ;  \frac {\sqrt {\epsilon \delta N } } {2 q \delta_0   g n_1 r d_2'}\rp.
	\end{split}
	\end{equation}
	Then 
	\begin{align}
	\label{10eq: estimate for S} 
	{\mathcal S}_{  \shskip g, \shskip n_1, \shskip r, \shskip s}^{\shskip\epsilon, \shskip \delta_0, \shskip \delta'} (q, N, D_2, N_2) \Lt |q |^{\varepsilon} \frac {\big|q^6 \delta_0 g^4 n_1^2 r^3 \big| | n_1 r s|^{7/16} D_2} {| \delta|^{1/2} |s |  N^{3/2}   }    \lp \left| \frac {q \delta' g} {s} \right| + N_2 \rp .
	\end{align}
\end{lem}

Granted that Lemma \ref{lem: final} holds, we can now finish the proof of Proposition \ref{prop: S}. Using \eqref{6eq: bounds for tilde W} in Lemma \ref{lem: main} to determine the range of summation and then applying a dyadic partition to \eqref{10eq: S... (q,N)}, it follows from Lemma  \ref{lem: final} that the sum in \eqref{10eq: S} is bound by
\begin{align*}
|q|^{\varepsilon}    N^{1/2}   \underset{\delta_0 \delta' = \delta}{\sum \sum} \frac { 1 } {|\delta_0|  |\delta|^{1/2}} \underset{ \sstyle f_2, \shskip g, \shskip n_1', \shskip r   \atop {\sstyle |f_2 g n_1' r | \leqslant  |q|^{1/2+ \varepsilon}} }{\sum \sum \sum \sum}  \sum_{ s | \delta' g}   \frac {  | f_2  n_1' r s|^{7/16} } { | f_2  g n_1'^2  |^2 |r s |     }    D_2 \lp \left| \frac {q \delta' g} {s} \right| + N_2 \rp,
\end{align*}
for some $D_2$, $N_2$ in the range \eqref{10eq: D N}, with $n_1 = f_2 n_1'$. Note here that  \eqref{10eq: D N} and \eqref{7eq: N < q3} imply $ |f_2 g n_1' r | \leqslant |q|^{\varepsilon} \sqrt {|\delta| N} /|q| \leqslant |q|^{1/2+ \varepsilon}$. Inserting \eqref{10eq: D N}, the sum  above is further bounded by the sum of
\begin{align*}
 |q|^{\varepsilon}   N  \underset{\delta_0 \delta' = \delta}{\sum \sum} \frac { |\delta'| } {|\delta_0|^2  } \underset{ \sstyle f_2, \shskip g, \shskip n_1', \shskip r    \atop {\sstyle |f_2 g n_1' r | \leqslant  |q|^{1/2+ \varepsilon} } }{\sum \sum \sum \sum}  \sum_{ s | \delta' g}   \frac {  | f_2  n_1' r s|^{7/16} } { \big| f_2^3  g^2 n_1'^5   r^2 s^2 \big|     }      \Lt    |q|^{7/32+\varepsilon }   |\delta| N
\end{align*}
and 
\begin{align*}
|q|^{\varepsilon}  \frac {\, N^{3/2}}  {|q|}  \underset{\delta_0 \delta' = \delta}{\sum \sum} \frac { |\delta|^{3/2} } {|\delta_0|^5   } \underset{ \sstyle f_2, \shskip g, \shskip n_1', \shskip r    \atop {\sstyle |f_2 g n_1' r | \leqslant  |q|^{1/2+ \varepsilon} } }{\sum \sum \sum \sum}  \sum_{ s | \delta' g}   \frac {  \big| f_2  n_1' r s|^{7/16} } { | f_2^5  g^3 n_1'^7  r^3 s^2 \big|     } \Lt |q|^{\varepsilon} \frac{|\delta|^{3/2} N^{3/2}  } {|q|} \Lt   |q|^{1/2+ \varepsilon} N,
\end{align*}
as desired.

\section{Completion of the proof}

%In this final section, we prove Lemma \ref{lem: final}. 
We start the proof of  Lemma \ref{lem: final}  with separating  variables in the first argument of $\widetilde W  $ in \eqref{10eq: S (q, N, D2, N2)}  by a standard Mellin inversion technique. 
Precisely, in the polar coordinates, 
\begin{align*}
	\widetilde W (y e^{i \theta}, \varLambda) = \sum_{m = -\infty }^{\infty}  e^{i m \theta} \int_{ -\infty}^{\infty}  \varUpsilon_Y (t, m; \varLambda) y^{ i t} d t, \hskip 10pt Y / 4 \leqslant y \leqslant 4 Y,
\end{align*}
with
\begin{align*}
 \varUpsilon_Y (t, m; \varLambda) =  \frac 1 {4 \pi^2} \int_{0}^{2 \pi} \int_0^{\infty} \widetilde W (y e^{i \theta}, \varLambda) \tw (y/Y) e^{- i m \theta} y^{- i t} \frac{d y d \theta} y,
\end{align*}
where $\tw (x)$ is a fixed  smooth function such that $\tw (x) \equiv 1$ on $ [1/4 , 4  ]$ and  $\tw (x) \equiv 0$ on $(0, 1/8) \cup ( 8, \infty)$, say, and  $Y =   {  | s | N_2  N } / { 16 \big| n_1 r^2 (q g  )^3 \big| D_2^3  }$.  
It should be pointed out that % the Mellin inversion over $\BC \smallsetminus \{0\}$ involves a 
one has to repeat partial integration for both $y$ and $\theta$ (at least twice) to secure the convergence of the $m$-sum and the $t$-integral. For this, we apply the estimates for the derivatives of $\widetilde W (y e^{i \theta}, \varLambda)$ in Lemma \ref{lem: main}.  It is crucial that the right-hand side of  \eqref{6eq: bounds for tilde W} in Lemma \ref{lem: main} is up to the implied constant independent on the differential orders  $\gamma$ and $\delta$.
From the uniform bound \eqref{6eq: bounds for tilde W, uniform} in  Lemma \ref{lem: main} we infer that $ {\mathcal S}_{  \shskip g, \shskip n_1, \shskip r, \shskip s}^{\shskip\epsilon, \shskip \delta_0, \shskip \delta'} (q, N, D_2, N_2)$ is bounded by
\begin{equation}\label{11eq: double sum}
\begin{split}
|q|^{\varepsilon} \frac {\big|q^4 \delta_0 g^4 n_1^2 r^3 \big| } {| \delta|^{1/2} |s |  N^{3/2} N_2 }   \left|\rule{0 pt}{22 pt}\right.  \underset{ \sstyle  D_2 \leqslant |d_2'| < 2 D_2  \atop {\sstyle  (d_2',   q \delta     g  n_1 ) = \frO   } }{\sum   } & \underset{ \sstyle  N_2 \leqslant |n_2'| < 2 N_2  \atop {\sstyle  (n_2', q \delta' g / s) = \frO }   }{\sum   } \alpha (d_2') \beta(n_2')  A (r s n_2'  ,  n_1  ) \\
  e & \lp \hskip -1 pt  \Re \hskip -1 pt  \lp \frac {\overline \epsilon \delta_0  n_1 n_2' \overline {q d_2' } }  {\delta' g / s} \rp \hskip -1 pt  \rp  H ( - \overline {\epsilon  }  \delta_0   n_1  s n_2'  \overline{\delta'    g   d_2'} ; q ) \left|\rule{0 pt}{22 pt}\right.
\end{split}
\end{equation}
for some complex coefficients $\alpha (d_2')$, $\beta (n_2')$ of absolute value at most $1$.

The main ingredient for estimating \eqref{11eq: double sum} is the following analogue of \cite[Lemma 13]{Blomer}, which in turn is a small variation of Lemma 11.1 in Conrey-Iwaniec \cite{CI-Cubic}. 

\begin{lem}\label{lem: bilinear}
	Let $q \in \frO$ be square-free,  $c \in \frO \smallsetminus \{0\}$, $a, b \in \frO$, with $(a, c) = (b, q) = \frO$. Let $\alpha (d)$, $\beta (n)$ be complex numbers for $d, n \in \frO \smallsetminus \{0\}$ with $1 \leqslant |d| \leqslant D$, $1 \leqslant |n| \leqslant N$. Assume that $| \alpha (d) | \leqslant 1$. Then
	\begin{align}
	\underset{(d n, \shskip q c) = \frO}{\sum \sum} \alpha (d) \beta (n) e (a \overline d n / c) H (b \overline d n, q ) \Lt \|\beta\| |q c|^{\varepsilon} |q/c| (|qc| + N) D (|qc| + D), 
	\end{align} 
	for any $\varepsilon > 0$, the implied constant depending only on $\varepsilon$. As usual, $\|\beta\| = \big( \sum_{\,n} |\beta (n)|^2 \big)^{1/2}$. 
\end{lem}

The proof of  \cite[Lemma 13]{Blomer} may be directly applied here. In our settings, the core would be the following bound of Conrey-Iwaniec
\begin{align}\label{11eq: bound for g}
g (\chiup, \psi) \Lt |q|^{2+\varepsilon},
\end{align}
for the character sum
\begin{align}
g (\chiup, \psi) = \underset{u, v (\mod q)}{\sum \sum} \chiup (uv(u+1) (v+1)) \psi (uv-1),
\end{align}
where $\chiup (\mod q)$  is the non-principal quadratic character  and $\psi (\mod q)$ is any character; see \cite[(11.10, 11.12)]{CI-Cubic}. For proving \eqref{11eq: bound for g}, we may assume that $q$ is prime as $g (\chiup, \psi)$ is multiplicative (the $\varepsilon$ in \eqref{11eq: bound for g} may be removed for $q$ prime). The proof in \cite[\S 13, 14]{CI-Cubic} employs the Riemann hypothesis for varieties over finite fields if $\psi \neq \chiup$ and some elementary arguments if $\psi = \chiup$. Their proof may be extended to any finite field and in particular to $\frO/ q\frO$.

We apply Lemma \ref{lem: bilinear} with $c = \delta' g/s$ for the double sum in \eqref{11eq: double sum}. In this way, we see that \eqref{11eq: double sum} is bound by
\begin{align} \label{11eq: bound, 2}
|q  |^{\varepsilon} \frac {\big|q^6 \delta_0 g^4 n_1^2 r^3 \big| D_2} {| \delta|^{1/2} |s |  N^{3/2} N_2  }  \lp \sum_{ N_2 \leqslant |n_2' | < 2 N_2} |A (r s n_2' , n_1) |^2 \rp^{1/2} \lp \left| \frac {q \delta' g} {s} \right| + N_2 \rp  ,
\end{align}
where $D_2$, $N_2$ are restricted by \eqref{10eq: D N}, which, along with \eqref{7eq: N < q3}, implies in particular $D_2 \leqslant |q \delta' g/s|$. From \eqref{3eq: bounds for Fourier, 2} we infer
\begin{align*}
\lp \sum_{ N_2 \leqslant |n_2' | < 2 N_2} |A (r sn_2' , n_1) |^2 \rp^{1/2} \Lt |q|^{\varepsilon}  | n_1 r s |^{7/16  } N_2,
\end{align*}
so that  \eqref{11eq: bound, 2} is bound by
\begin{align*}
|q |^{\varepsilon} \frac {\big|q^6 \delta_0 g^4 n_1^2 r^3 \big| |n_1 r s|^{7/16} D_2} {| \delta|^{1/2} |s |  N^{3/2}   }    \lp \left| \frac {q \delta' g} {s} \right| + N_2 \rp .
\end{align*}
This completes the proof of Lemma \ref{lem: final}.

\delete{For this we use the Mellin inversion on $\BC \smallsetminus \{0\}$,
\begin{align*}
W \big(y e^{i \theta} \big) = \frac 1 {4 \pi^2 i}\sum_{ k = -\infty}^{\infty}  e^{- i k \theta} \int_{ (\sigma) } \lp   \int_0^\infty \int_0^{2 \pi} W \big( x e^{i \phi} \big) e^{ i k \phi} d\phi \cdot x^{ s - 1} d x \rp y^{-s} d s,
\end{align*}
where $W$ is in the class of functions arising in the study of Hankel transform, in particular, $W$ is Schwartz at infinity and has certain type of singularities near zero,  relying on the parameters $(\mu_1, \mu_2, \mu_3) = (\mu, 0, - \mu)$ as in \S \ref{sec: Voronoi SL3}, and we may chose $\sigma = \frac 2 3$($> \frac 7 {16} \geqslant 2 |\Re \,\mu|$); see \cite[\S 2, 3]{Qi-Bessel} for more details. }

%subject to
%\begin{align*}
%&(d_2' f_2 g  n_1', \delta  )  = (g n_1', f_2) = (f_2 g n_1', %d_2' q) = (n_1' n_2' s, d_2') = \frO \\
%&  (d_2' n_1' n_2' s, q) = (n_2', g \delta' / s) = \frO, %\hskip 10 pt \mu^2 (n_1' g) = 1.
%\end{align*}

%Again, the computations in \cite[\S 10]{CI-Cubic} hold for $\frO$.

%	\bibliographystyle{alphanum}
%	\bibliography{references}
	
	%\end{spacing}
	
	\def\cprime{$'$} \def\cprime{$'$}

\end{document}